\documentclass[11pt,requn]{amsart}
\usepackage[letterpaper,margin=1.1in]{geometry}

\usepackage{nccmath,amssymb,amscd}
\usepackage{mathtools}
\usepackage{graphicx}
\usepackage{mathabx}
\usepackage{upgreek}
\usepackage{bbm}
\usepackage{wasysym,footnote}
\usepackage[all]{xy}
\usepackage{hyperref}
\input{crystal.sty}

\usepackage{xcolor}
\hypersetup{
	colorlinks,
	linkcolor={red!50!black},
	citecolor={blue!50!black},
	urlcolor={blue!75!black}
}

\newtheorem{theorem}{Theorem}[section]
\newtheorem{lemma}[theorem]{Lemma}
\newtheorem{corollary}[theorem]{Corollary}
\newtheorem{proposition}[theorem]{Proposition}

\theoremstyle{definition}
\newtheorem{definition}[theorem]{Definition}

\newtheorem{example}[theorem]{Example}
\newtheorem{conjecture}[theorem]{Conjecture}
\newtheorem*{lemma*}{Lemma}

\theoremstyle{remark}
\newtheorem{remark}[theorem]{Remark}
\newtheorem{warning}[theorem]{Warning}
\newtheorem{observation}[theorem]{Observation}
\newtheorem{construction}[theorem]{Construction}
\newtheorem{question}[theorem]{Question}

\numberwithin{equation}{section}
\allowdisplaybreaks

\DeclareMathOperator{\ad}{ad}
\DeclareMathOperator{\Aut}{Aut}
\DeclareMathOperator{\coker}{coker}

\DeclareMathOperator{\E}{E}
\DeclareMathOperator{\End}{End}
\DeclareMathOperator{\Ext}{Ext}
\DeclareMathOperator{\ext}{ext}

\DeclareMathOperator{\Gr}{Gr}
\DeclareMathOperator{\Hom}{Hom}
\DeclareMathOperator{\Id}{Id}
\DeclareMathOperator{\img}{im}

\DeclareMathOperator{\opp}{op}
\DeclareMathOperator{\PHom}{PHom}
\DeclareMathOperator{\IHom}{IHom}
\DeclareMathOperator{\rank}{rank}
\DeclareMathOperator{\rep}{rep}

\DeclareMathOperator{\SL}{SL}

\DeclareMathOperator{\Tr}{Tr}
\DeclareMathOperator{\PC}{PC}

\DeclareMathOperator{\gen}{gen}
\DeclareMathOperator{\trop}{trop}

\DeclareMathOperator{\red}{red}
\DeclareMathOperator{\triv}{triv}

\newcommand{\Ec}{{\check{\E}}}
\newcommand{\ec}{{\check{\e}}}

\newcommand{\fc}{{\check{f}}}
\newcommand{\dc}{{\check{d}}}

\newcommand{\betac}{{\check{\beta}}}
\newcommand{\gammac}{{\check{\gamma}}}

\newcommand{\dtc}{{\check{\delta}}}

\newcommand{\epc}{{\check{\ep}}}
\newcommand{\lc}{{\check{l}}}

\newcommand{\rc}{{\check{r}}}

\newcommand{\mr}[1]{{\sf #1}}
\newcommand{\e}{{\rm e}}

\newcommand{\g}{{\sf g}}

\renewcommand{\S}{\mc{S}}
\newcommand{\op}[1]{\operatorname{#1}}
\newcommand{\mb}[1]{\mathbb{#1}}
\newcommand{\mc}[1]{\mathcal{#1}}
\newcommand{\mf}[1]{\mathfrak{#1}}

\renewcommand{\b}[1]{\bold{#1}}

\newcommand{\ep}{{\epsilon}}

\newcommand{\proj}{\operatorname{proj}\text{-}}

\newcommand{\br}[1]{\overline{#1}}

\newcommand{\innerprod}[1]{\langle#1\rangle}

\newcommand{\sm}[1]{\left(\begin{smallmatrix}#1\end{smallmatrix}\right)}

\newcommand{\fr}[1]{\framebox[1.2\width]{{$#1$}}}

\newcommand{\dv}{\underline{\dim}}

\newcommand{\wtd}[1]{\widetilde{#1}}

\newcommand{\wt}{\operatorname{wt}}
\newcommand{\rhoc}{\check{\rho}}
\newcommand{\ckQ}{\widehat{k\Delta}}

\newcommand{\lrcoef}{c^\lambda_{\mu\,\nu}}
\newcommand{\uca}{\br{\mc{C}}}
\newcommand{\mub}{\mu_{\b{u}}}
\newcommand{\tauh}{\hat{\tau}}
\renewcommand{\t}{{\op{T}}}
\newcommand{\Bup}{{\rm B}}
\newcommand{\ibar}{{\bar{\imath}}}
\newcommand{\jbar}{{\bar{\jmath}}}
\newcommand{\hdash}{{\text{\textemdash}}}
\newcommand{\edge}[1]{\stackrel{#1}{\hdash}}

\renewcommand{\k}{{\mathbbm{k}}}

\newcommand{\rtau}{\tilde{\tau}_\mu}

\begin{document}
	
\title{Crystal Structure of Upper Cluster Algebras}
\author{JiaRui Fei}
\address{School of Mathematical Sciences, Shanghai Jiao Tong University, China}
\email{jiarui@sjtu.edu.cn}
\thanks{The author was supported in part by National Natural Science Foundation of China (No. 11971305 and No. 12131015)}

\subjclass[2020]{Primary 13F60; Secondary 05E10, 16G10}

\date{}
\dedicatory{}
\keywords{Cluster algebra, Crystal, Kashiwara, Quiver with Potential, Mutation, Presentation, Boundary Representation, Generic Basis, Biperfect Basis}

\begin{abstract}
We construct upper seminormal crystals on the $\mu$-supported tropical points of skew-symmetric cluster algebras 
from their exchange matrices. A generic potential provides an additive categorification
in which boundary representations determine the Cartan matrices and generic ranks describe the crystal operators.
Under frozen-Jacobi-finiteness and reachability, the raising operators lift to locally nilpotent derivations
of the upper cluster algebra. 
For $\tau$-exact boundary pairs, these derivations give integrable actions of derived Kac--Moody algebras in characteristic zero. In particular, the associated Cartan matrices are positive semidefinite.
When the generic characters form a basis, this basis is BK-biperfect and the $\tau$-exact crystals are normal.
We also construct algebraic Tits actions and describe compatibly pointed biperfect bases through a coefficient graph. 
The theory is illustrated by classical and new examples.
\end{abstract}
\maketitle
\setcounter{tocdepth}{2}
\tableofcontents

\section*{Introduction}
\subsection{Motivations}
The theory of crystals grew out of the closely
intertwined theories of Lusztig's canonical bases and Kashiwara's crystal bases for quantum groups \cite{L,L2,K,K2,BS}.
The canonical basis of the negative half of a quantum group is indexed by the crystal $\mc B(\infty)$, 
linking Lusztig's distinguished basis with Kashiwara's crystal operators.
Lusztig's geometric construction of canonical bases through quiver representations and perverse sheaves
\cite{Lgeom} established a profound connection between quantum groups and quiver geometry.
He further developed semicanonical bases using constructible functions 
on varieties of representations of preprojective algebras \cite{Lsemi}.
Building on this geometric framework, Kashiwara and Saito realized $\mc B(\infty)$ 
through irreducible components of Lusztig's nilpotent varieties \cite{KS},
while Littelmann developed a complementary path model \cite{Li}.
These developments reveal the deep interplay among crystal structures, distinguished bases, and the geometry
of quiver representations.

Let $G$ be a simple simply-connected complex algebraic group
and $U$ a maximal unipotent subgroup. The coordinate ring $\k[U]$
is a fundamental example of an upper cluster algebra \cite{BFZ,GLSa};
its dual canonical basis is indexed by the crystal $\mc B(\infty)$.
Base-affine spaces, partial flag varieties,
and their strata provide further connections between cluster algebras
and representation theory \cite{GLSp,Fart,Le}. These examples motivate
the following questions:
\begin{enumerate}
\item How can one detect and describe crystal structures of upper
cluster algebras intrinsically, i.e., from their exchange matrices
without prior knowledge of the underlying spaces?
\item How can one categorify these crystal structures using
representations of quivers with potential?
\item How can one algebraically lift these combinatorial crystals to
the upper cluster algebras themselves?
\end{enumerate}

We seek to extend this interplay of crystals, distinguished bases, and
quiver representations to general skew-symmetric cluster algebras.
Crystal structures in particular cluster-algebra settings have been
studied in \cite{GKS,KN,Ko,Fu}; related connections with tropical
combinatorics appear in \cite{Kir}. Our starting datum is the exchange
matrix $B$, without a prescribed group action or Cartan type. We choose
a generic potential and use the representation theory of quivers with
potential \cite{DWZ1,DWZ2} as a categorification tool. The boundary
representations attached to frozen vertices encode the Cartan matrix
through homological pairings and the crystal operators through generic
ranks.

\subsection{Tropical Crystal Structures}
For the general construction, let $(\Delta,\S)$ be a
nondegenerate ice quiver with potential, and let $J=J(\Delta,\S)$.
We use the conventions of
Section~\ref{ss:conventions}: a mutation class starts from a reference
seed with no arrows between frozen vertices, and subsequent QP
mutations preserve every arrow and potential term that they create.
Write $\Delta_0^\mu$ and $\Delta_0^{\op{fr}}$ for the mutable
and frozen vertex sets. The full skew-symmetric matrix is
$$ B(\Delta)(u,v)=\lvert\{u\to v\}\rvert-\lvert\{v\to u\}\rvert, $$
and $B_\Delta$ consists of its mutable rows.

For $\delta\in\mathbb Z^{\Delta_0}$, the presentation space of
weight $\delta$ is
$$ \PHom(\delta)=\Hom_J(P([-\delta]_+),P([\delta]_+)), \qquad P(\beta)=\bigoplus_{u\in\Delta_0}\beta(u)P_u, $$
where $P_u$ is the indecomposable projective at $u$. The cokernel
of a general presentation has a constant dimension vector
$\dv(\delta)$. Injective presentations similarly have weights
$\check\delta$.
\begin{definition}[\cite{Fs1}]
A presentation weight $\delta$ is \emph{$\mu$-supported} if
$\dv(\delta)$ is supported on mutable vertices. We write
$\trop(\Delta,\S)$ for the set of these weights.
\end{definition}
\noindent This set will underlie the crystal.

For a presentation $d:P_-\to P_+$ and a representation $N$, set
$$ \E(d,N)=\coker\bigl(\Hom_J(P_+,N)\longrightarrow\Hom_J(P_-,N)\bigr). $$
We write $\e(\delta,N)$ for the generic value of its dimension on
$\PHom(\delta)$. The $\E$-pairing extends to decorated representations; a
decorated representation $\mc M$ is rigid if
$\e(\mc M,\mc M)=0$.

Each frozen vertex $i$ has a boundary representation $E_i$ and
a dual boundary representation $E_i^\star$, constructed in
Section~\ref{S:boundary}. 
Let $\mc E_i^\mu$ be the mutable decorated representation
associated with the presentation extracted at the frozen vertex $i$.
The vertex $i$ is \emph{reachable} (resp. rigid) if $\mc E_i^\mu$ is reachable (resp. rigid). 
Write $\epsilon_i$ and $\check\epsilon_i$ for the projective and injective weights of $E_i$.

The boundary representations also characterize the underlying set.
When $(\Delta,\S)$ is frozen-Jacobi-finite,
$$ \delta\in\trop(\Delta,\S) \quad\Longleftrightarrow\quad \hom(\delta,E_j^\star)=0 \quad\text{for every frozen vertex }j, $$
where $\hom$ denotes the generic dimension of the Hom space.
Theorem~\ref{T:musupp} gives an intrinsic boundary version of the
criterion proved in \cite{Fs1}; see also \cite{Ft}.

\begin{definition}
For a set $I$ of rigid frozen vertices, its \emph{Cartan type}
is the symmetric generalized Cartan matrix $C_I=(c_{ij})$, where
$$ c_{ij}=2\updelta_{ij}-\e_J(E_i,E_j)-\e_J(E_j,E_i). $$
A linear map $\wt=(\wt_i)_{i\in I}:\mathbb Z^{\Delta_0}\to\mathbb Q^I$
is \emph{compatible} if it annihilates the row space of $B_\Delta$,
and \emph{adapted to $I$} if
$\wt_i(\check\epsilon_j)=c_{ij}$.
\end{definition}
The boundary-pairing formulas of Section~\ref{ss:minimal} also give
$$ c_{ij}=2\updelta_{ij} -\e_{J_\mu}(\mc E_i^\mu,\mc E_j^\mu) -\e_{J_\mu}(\mc E_j^\mu,\mc E_i^\mu). $$
By mutation invariance of the $\E$-pairing, this also shows
that $C_I$ is mutation-invariant.

The QP mutation induces a tropical bijection between the sets of
$\mu$-supported weights and transports compatible adapted gradings.
Let $\mf T$ be the mutation tree rooted at the reference seed.
\begin{definition}
A \emph{crystal cluster structure} is a family of crystal structures on
$\trop(\Delta,\S)_t$, $t\in\mf T$, for which every tropical
mutation is a crystal isomorphism.
\end{definition}
Thus mutation intertwines the operators and preserves the string functions
and weights. We use the notation
$(r_i,l_i;\rho_i,\lambda_i;\wt)$. The crystal is \emph{upper
seminormal} if
$\rho_i(\delta)=\max\{m\ge0:r_i^m(\delta)\ne0_{\mathrm{cr}}\}$,
and \emph{seminormal} if $\lambda_i$ likewise measures the length
of the lowering string. Here $0_{\mathrm{cr}}$ is the auxiliary
crystal element, distinct from the zero weight.

Our weight maps initially take values in coroot coordinates:
$$ P_{\mathrm{cor}}=\Hom_{\mathbb Z}(Q^\vee,\mathbb Z), \qquad Q^\vee=\bigoplus_{i\in I}\mathbb Z h_i. $$
Write $\bar\varpi_i$ for the dual basis. A rational adapted grading
gives weights in $P_{\mathrm{cor}}\otimes\mathbb Q$ and a weaker
rational version of the crystal axioms in Section~\ref{ss:crystal}.
An integral adapted grading gives integral weights and parameters.

\begin{theorem}[Theorem~\ref{T:upper}]\label{intro:upper}
Assume that $(\Delta,\S)$ is frozen-Jacobi-finite. Let $I$ be
a set of reachable frozen vertices, and choose a compatible grading
adapted to $I$. Then $\mc B=\trop(\Delta,\S)$ has an upper
seminormal crystal cluster structure of Cartan type $C_I$, with
$\wt=\sum_i\wt_i\bar\varpi_i$, given by
\begin{align*}
 r_i(\delta)&=\delta+\epsilon_i+\rank(\epsilon_i,\tau\delta)B(\Delta),
 &\rho_i(\delta)&=\e(\delta,E_i),\\
 l_i(\delta)&=\delta-\check\epsilon_i+\rank(\delta,\epsilon_i)B(\Delta),
 &\lambda_i(\delta)&=\rho_i(\delta)+\wt_i(\delta).
\end{align*}
The operators give $0_{\mathrm{cr}}$ if the corresponding weight
is not $\mu$-supported.
\end{theorem}
Here $\rank$ denotes the generic dimension vector of the
image of a morphism between general decorated
representations of the indicated weights \cite{Fg}.
The operators can be computed by mutation: at a seed where
$E_i=S_i$, raising adds $\check\delta_{S_i}$ whenever the
string length is positive (Lemma~\ref{L:mu2S}).

For dually reachable vertices, the opposite QP gives a dual upper
seminormal crystal on the $\mu$-supported injective weights.
The change from projective to injective weights of the same general
decorated representation identifies its underlying set with $\mc B$.
This gives the dual boundary crystal, whose operators and parameters
we denote by a superscript $\star$.

The two constructions combine to give seminormal crystals when the
boundary representations form $\tau$-exact pairs.
\begin{definition}
A pair of frozen vertices $(i,\ibar)$ is \emph{$\tau$-exact}
if $\tau^{-1}E_i\cong E_{\ibar}^\star$.
\end{definition}
For such a pair, the natural weight function is
$$ \wt_i(\delta)=\delta\bigl(\dv E_i-\dv(\tau^{-1}E_i)\bigr). $$
\begin{theorem}[Theorem~\ref{T:crystal}]\label{intro:crystal}
Assume that $(\Delta,\S)$ is frozen-Jacobi-finite. Let
$\{(i,\ibar)\}_{i\in I}$ be $\tau$-exact pairs, with $i$
reachable and $\ibar$ dually reachable. The natural integral
grading makes $\mc B$ a seminormal crystal cluster structure of
Cartan type $C_I$, with the operators of Theorem~\ref{intro:upper}
and
$$ \rho_i(\delta)=\e(\delta,E_i),\qquad \lambda_i(\delta)=\ec(\tau^{-1}E_i,\check\delta). $$
Here $\check\delta$ is the injective weight corresponding to
$\delta$.
\end{theorem}
Both theorems use coroot-coordinate weights. A compatible weight map
to a full realization can be chosen under \eqref{eq:existwt}, which
requires the classes of $\check\epsilon_i$, $i\in I$, to be
linearly independent modulo the exchange-row space. This distinction
matters in affine type; see Remark~\ref{r:affine-classical-weights}.

Section~\ref{S:auto} studies two comparisons of boundary crystals.
The first concerns direct and opposite cluster automorphisms and their
relations with the ordinary and dual crystal operators. It also relates
generalized Kashiwara maps to Lusztig's involution on finite-type
components. The second concerns relative $\tau$-translation: when the
mutable boundary representations $\mathbf E^\mu$ and
$\tau_\mu\mathbf E^\mu$ are realized by the corresponding QPs, we
construct a bijection of their supported weights and prove that it is a
crystal isomorphism with the transported compatible grading
(Theorem~\ref{T:taumu-enhanced-upper}).

\subsection{Algebraic Lifts of Tropical Crystals}
The next question is whether these combinatorial operations come from
actions on the upper cluster algebra. Put $\k=\mathbb Z$. For a
seed $t=(\Delta_t,\mathbf x_t)$, let
$$
 \mc L_t=\k[x_{t,u}^{\pm1},x_{t,v}:
       u\in\Delta_0^\mu,\ v\in\Delta_0^{\op{fr}}],
 \qquad \uca(\Delta)=\bigcap_{t\in\mf T}\mc L_t.
$$
The frozen variables remain polynomial. The generic character
$C_{\gen}$ associates an element of this algebra to each
$\mu$-supported presentation weight \cite{DWZ2,P}.
Throughout this subsection, we assume that $(\Delta,\S)$ is frozen-Jacobi-finite. 
Generic bases are known for important classes of upper cluster algebras
\cite{P,FW,Qb}; Remark~\ref{r:middle} discusses their relation to
the middle cluster algebra of Gross--Hacking--Keel--Kontsevich.

Let $I$ and $I^\star$ be sets of reachable and dually reachable
frozen vertices, respectively. 
We construct derivations $R_i$, $i\in I$, and $R_j^\star$, $j\in I^\star$, 
lifting the two families of raising operators. At a seed where $E_i=S_i$,
the first derivation has the simple form
$$ R_i=\left(\prod_{u\to i}x_u\right)\frac{\partial}{\partial x_i}. $$
The construction is independent of the optimizing seed. 
Thus no group action is needed to define the operators; their Lie-algebra relations
are governed by the boundary $\E$-pairings.

\begin{theorem}[Theorem~\ref{T:crystalUCA}]\label{intro:crystalUCA}
The derivations $R_i$, $i\in I$, and $R_j^\star$,
$j\in I^\star$, define $U(\mathfrak n_+(C_I))$- and
$U(\mathfrak n_+(C_{I^\star}))$-module algebra structures on
$\uca(\Delta)$, respectively.
\end{theorem}
\noindent Here $\mathfrak n_+(C)$ is the positive part of
the Kac--Moody Lie algebra with Cartan matrix $C$. The mixed
relations are determined by $\e(E_j^\star,E_i)$; in particular,
the two actions commute exactly when all these pairings vanish
(Proposition~\ref{P:mixed}).

To lift the crystal to a basis, we require triangularity with respect
to its string parameters. A basis
$B=\{B(\delta):\delta\in\mc B\}$, homogeneous for both weight
maps and normalized by $B(\mathbf0)=1$, is \emph{BK-biperfect} if
\begin{align*}
 R_iB(\delta)&=\rho_i(\delta)B(r_i\delta)+v_i,&&i\in I,\\
 R_j^\star B(\delta)&=\rho_j^\star(\delta)B(r_j^\star\delta)+v_j^\star,
 &&j\in I^\star,
\end{align*}
where $v_i$ is a linear combination of $B(\eta)$ with
$\rho_i(\eta)<\rho_i(\delta)-1$, and the analogous strict
inequality holds for the starred remainder. We set
$B(0_{\mathrm{cr}})=0$. These are the perfectness conditions of
Berenstein--Kazhdan \cite{BK}, applied to both crystal structures.
\begin{theorem}[Theorem~\ref{T:genericCR}]\label{intro:biregular}
Assume that the generic characters form a $\k$-basis of
$\uca(\Delta)$.
The generic basis of $\uca(\Delta)$ is BK-biperfect with respect
to $(R_i)_{i\in I}$ and $(R_j^\star)_{j\in I^\star}$.
\end{theorem}
Perfectness describes the kernels of powers of the root
derivations in terms of the crystal string parameters.
In particular, the $\mathfrak n_+(C_I)$-invariant
subalgebra has a basis consisting of the generic
characters $C_{\mathrm{gen}}(\delta)$ with
$\rho_i(\delta)=0$ for all $i\in I$
(Corollary~\ref{C:string-kernels-hw}).
Moreover, maximal raising preserves generalized cluster
monomials (Corollary~\ref{C:Rmaxcv}), as follows from
its interpretation through $\E$-orthogonalization and
the preservation of rigidity (Proposition~\ref{P:max-approximation}).

For a $\tau$-exact pair $(i,\ibar)$, the starred raising derivation
$R_{\ibar}^\star$ is the lowering derivation $L_i$. Hence, whenever
the generic characters form a basis, they are perfect for both raising
and lowering operators. We conjecture that the theta and triangular
bases are likewise BK-biperfect whenever they form bases of
$\uca(\Delta)$ with polynomial frozen variables and are indexed by
$\mc B$ \cite{GHKK,Qt}.

The $\tau$-exact pairs also yield the diagonal Chevalley relations.
For a nonzero element $f$ of $H_i$-weight $m$, the nilpotence degrees
of $R_i$ and $L_i$ satisfy $\deg_{L_i}f-\deg_{R_i}f=m$.
Together with the character-string estimate, this identity gives
$[R_i,L_i]=H_i$. The Serre and mixed relations then give the
following integrable Kac--Moody action.

\begin{theorem}[Theorem~\ref{T:normalUCA}]\label{intro:normalUCA}
In the situation of Theorem~\ref{intro:crystal},
the derivations $R_i,L_i,H_i$, $i\in I$, define a
$U(\mathfrak g'(C_I))$-module algebra structure on
$\uca(\Delta)$, integrable after scalar extension to a field of characteristic zero. 
Under \eqref{eq:existwt}, the additional Cartan derivations extend this action to $\mathfrak g(C_I)$.
\end{theorem}
\noindent Here $\mathfrak g'(C_I)$ denotes the derived Kac--Moody algebra.
The existence of this action imposes a strong restriction on the
Cartan type. Under the hypotheses of
Theorem~\ref{intro:crystal}, the matrix $C_I$ is positive semidefinite,
so every connected component is of finite or affine type
(Corollary~\ref{C:Cartan-positive}).
See Proposition~\ref{P:Cartan-growth} for a more general statement.
The proof compares Lie-algebra growth with the polynomial growth 
of finitely generated Lie subalgebras of Laurent derivations. 

We use \emph{normality} in the sense of finite-type restrictions:
for every $J\subseteq I$ of finite Cartan type, the restricted
crystal is a disjoint union of highest-weight crystals.
\begin{corollary}[Corollary~\ref{C:normal}]\label{intro:normal}
In the situation of Theorem~\ref{intro:crystal}, suppose that
the generic characters form a basis of $\uca(\Delta)$. Then the
seminormal crystal is normal.
\end{corollary}

The combinatorial and algebraic actions give two related kinds of
symmetry. The seminormal crystal has mutation-compatible string
involutions $s_i$ (Corollary~\ref{C:Weyl}). When the crystal is normal,
these satisfy the Coxeter relations and define a Weyl-group action
\cite{K2}. The integrable action of Theorem~\ref{intro:normalUCA},
on the other hand, gives algebra automorphisms after scalar extension
to characteristic zero:
$$
 \dot s_i=\exp(R_i)\exp(-L_i)\exp(R_i).
$$
They satisfy the Tits relations (Theorem~\ref{T:Tits-lift}); in
particular, for a homogeneous element $f$ of weight $\nu$,
$\dot s_i^2(f)=(-1)^{\nu(h_i)}f$. Thus the algebraic action is naturally
an action of the Tits extension, rather than necessarily of the Weyl
group itself.

The connection between the two actions appears on a BK-perfect basis
$B$ for the seminormal crystal. Put $N_i=\rho_i+\lambda_i$, and let
$F_m^{(i)}$ be the span of the $B(\eta)$ with $N_i(\eta)\le m$.
Then Corollary~\ref{C:Tits-perfect-basis} gives
$$
 \dot s_i B(\delta)\equiv
 (-1)^{\lambda_i(\delta)}B(s_i\delta)
 \pmod{F_{N_i(\delta)-1}^{(i)}}.
$$
Thus the associated graded action is the signed string reflection;
compare \cite{Bws}.

\subsection{Biperfect Bases}
Perfectness alone leaves considerable freedom in a basis. In
Example~\ref{ex:SL5}, a strictly lower change of a BK-biperfect
basis of $\k[U]$ produces another BK-biperfect basis that no
longer contains a particular cluster monomial, without changing
either crystal structure. This illustrates why one should also use the
orders supplied by the cluster algebra. Our study of these changes is
motivated by the work of Baumann, Baumann--Kamnitzer--Knutson, and Qin
\cite{B,BKK,Qb}.

We call a BK-biperfect basis \emph{biperfect} if it is compatibly
pointed at its tropical points. Explicitly, in every seed $t$,
$$ B(\delta)=\mathbf x_t^{-\delta_t}F_t(\check{\mathbf y}_t), \qquad F_t(0)=1, $$
where $\delta_t$ is the representative of $\delta$ at $t$,
$\check y_{t,u}=\mathbf x_t^{-b_{t,u}}$, and $b_{t,u}$ is a
row of $B_{\Delta_t}$. The factor $F_t$ is polynomial in these
$\check y$-variables. This is the good-basis condition of
\cite{Qb}, which is satisfied by the generic basis.

For the classification, take $\k=\mathbb Z$ or $\mathbb C$,
assume that $B_\Delta$ has full rank, and suppose that the
initial seed is reachable. Here seed reachability is the
condition on the mutable quiver defined in Section~\ref{ss:genericB}.
Qin's basis theorem,
in the relative form of Proposition~\ref{T:relative-good}, makes any
transition between compatibly pointed bases unitriangular, with finite
lower sets for the order $\prec_{\mf T}$. This order means dominance
in every seed. If $B'(\delta)=\sum_{\eta}T_{\delta,\eta}B(\eta)$,
perfectness imposes equalities among the off-diagonal entries and
forces certain entries to vanish (Theorem~\ref{C:goodcr}).
Definition~\ref{D:coefficient-graph} records these conditions in a
graph on ordered pairs of distinct tropical points. 
We mark those pairs whose coefficients are forced to vanish. Let
$\mathfrak C$ be the connected components containing no marked pair.

\begin{theorem}[Theorem~\ref{T:biperfect-parameters}]
\label{intro:allcr}\label{intro:biperfect-parameters}
Assume that $B_\Delta$ has full rank and that the initial seed
is reachable. Fix the two crystal
structures and weight maps on $\Theta=\trop(\Delta,\S)$, and
choose a biperfect basis $B$ indexed by $\Theta$. All biperfect
bases with these structures and weight maps are parametrized by
$\prod_{C\in\mathfrak C}\k$.
Each component contributes one independent scalar, and the expansion
of every basis element in $B$ is finite.
\end{theorem}
This gives criteria for uniqueness of a biperfect basis and of its
individual elements. Over $\mathbb C$, the restrictions of biperfect bases to any
finite set of tropical points form a finite-dimensional affine space.
If a symmetry group preserves the reference basis and the crystal data,
then requiring the basis to be invariant forces the parameters to be
constant along its orbits on $\mathfrak C$
(Corollary~\ref{C:biperfect-parameter-symmetry}). For arbitrary
BK-biperfect bases, Proposition~\ref{T:allcr} instead gives a transition
criterion in terms of the string filtrations; its proof and the
comparison of string profiles are in Appendix~\ref{S:string-profiles}.

\subsection{Examples and Further Results}
For the standard unipotent model, the boundary derivations are
simple-root translations and the resulting crystal is
$\mc B(\infty)$. In finite simply-laced Dynkin type, we also
obtain a simultaneous linear description of the string coordinates.
For every reduced word $\mathbf i$ of $w_0$,
Theorem~\ref{T:string-linear} gives a seed $t_{\mathbf i}$
and a nonnegative unimodular matrix $D_{\mathbf i}$ such that
$$
 \mathbf s_{\mathbf i}^{\,t_{\mathbf i}}(\delta)
 =-\delta D_{\mathbf i}
 \qquad
 \bigl(\delta\in\trop(\Delta_{t_{\mathbf i}},\S_{t_{\mathbf i}})\bigr).
$$
The same seed and matrix work for every $\mu$-supported weight.
The proof combines the orthogonalization description of maximal
raising with the factorization coordinates of \cite{FOstring}.
String parametrizations and their relation to valuations also
appear in \cite{Ka}.

For the standard full base-affine QP of simply-laced Dynkin
type, Section~\ref{ss:G/U} identifies the local boundary derivations
with the usual infinitesimal right $G$-action on $A=\mathbb C[U^-\backslash G]$.
Consequently, the degree-$\lambda$ part of the
generic basis therefore has crystal $\mc B(\lambda)$, and its
internal-weight-$\gamma$ fiber has cardinality
$\dim L(\lambda)_\gamma$ (Corollary~\ref{C:G/U}). 
The simultaneous-kernel description then expresses tensor-product multiplicities by string inequalities,
recovering the formula of Berenstein--Zelevinsky \cite{BZ} in cluster coordinates.

Grassmannian models exhibit affine Cartan types: the standard model
for $\Gr(k,n)$, with $2\le k\le n-2$, has type
$\widetilde A_{n-1}$, while the exceptional $E_6$
Grassmannian in Section~\ref{ss:Grass} gives type
$\widetilde E_6$. These affine examples use classical weights,
as explained in Remark~\ref{r:affine-classical-weights}.
The simple canonical QPs and exact coefficient patterns provide
further examples.

Several representation-theoretic constructions underlie these
results. We extend a QP by a decorated representation, generalizing
\cite{FW}, and use this construction to define mutation directly on
presentations. The intrinsic boundary representations and general
ranks \cite{Fg} then give the support criterion and the raising and
lowering formulas. Appendix~\ref{S:KD} relates maximal raising to
unreduced $\E$-orthogonalization \cite{Fop} and develops its applications to
Kashiwara's data.

\subsection{Organization}
Sections~\ref{S:QP}--\ref{S:rl} develop the QP, presentation, and
general-rank constructions. Section~\ref{S:boundary} studies the
boundary representations, their support criterion, and the associated
Cartan matrix and gradings. Section~\ref{S:crystal} constructs the
crystals, and Section~\ref{S:auto} proves their compatibility with
cluster automorphisms and relative $\tau$-translation.
It also compares generalized Kashiwara maps with Lusztig's
involution on finite-type restrictions.

After the review of upper cluster algebras and generic bases in
Section~\ref{S:uca}, Section~\ref{S:lifting} constructs the
root derivations and the $\tau$-exact actions, proves the restriction
on Cartan type, and then treats perfectness, normality, and
Tits-lift compatibility under their respective hypotheses. 
Section~\ref{S:biperfect} classifies biperfect bases,
and Section~\ref{S:example} treats the examples. Appendix~\ref{S:KD}
studies maximal operators and Kashiwara's data; Appendix~\ref{S:string-profiles}
gives the general transition criterion and string-profile comparison.

\subsection{Notations and Conventions}\label{ss:conventions}
The representation field is $k=\mathbb C$, distinct from the cluster
coefficient ring $\k$. Quivers are finite, without loops, and we use
reduced $2$-acyclic representatives. We write $J=J(\Delta,\S)$, or
$A$ when there is no change of setting. We define the
Jacobian ideal using cyclic derivatives with respect to all arrows,
including frozen arrows. Generic potentials are
understood to be nondegenerate. A mutation at a frozen vertex is called an \emph{auxiliary mutation}.

\emph{Reference seeds.} Each mutation class used for the
boundary-crystal construction has a distinguished reference ice QP
$(\Delta_{t_0},\S_{t_0})$ with no arrows between frozen vertices.
All subsequent seeds are obtained by DWZ mutation, with the frozen
arrows created by mutation and the potential terms involving them
preserved. Auxiliary constructions are specified separately.

We assume that the mutable restriction $(\Delta,\S)_\mu$ is
Jacobi-finite. The frozen-Jacobi-finiteness assumption
appearing in several results is equivalent to
finite-dimensionality of the full Jacobian algebra by
Lemma~\ref{L:ice-finite}. These finiteness conditions
are preserved under QP mutation \cite[Corollary~6.6]{DWZ1}.

The traditional crystal notation is
$r_i=\widetilde e_i$, $l_i=\widetilde f_i$,
$\rho_i=\varepsilon_i$, and $\lambda_i=\varphi_i$.
The auxiliary crystal value $0$ is distinct from the zero lattice
vector. We write $\Delta_0^\mu$ and $\Delta_0^{\op{fr}}$ for the
mutable and frozen vertex sets. The rectangular exchange matrix
$B_\Delta$ consists of the mutable rows of the full skew-symmetric
matrix $B(\Delta)$.

All modules are right modules, and all vectors are row vectors.
For the direct sum of $n$ copies of $M$, we write $nM$.
We write $\hom,\ext$ and $\e$ for $\dim\Hom, \dim\Ext$, and $\dim \E$. The superscript $*$ is the trivial dual for vector spaces.
Unadorned $\Hom$ and $\E$ are understood over the Jacobian algebra of an appropriate quiver with potential.

In the literature of cluster algebras, the {\em final-seed} mutation was first introduced in \cite{FZ1}.
But in this article, all mutations except \eqref{eq:exrel} are the {\em initial-seed} mutations (see Section \ref{ss:uca} for the meaning).
The mutation defined for quivers with potentials in \cite{DWZ1} is to model the initial-seed mutation.
A typical example is Lemma \ref{L:Cmu}.
\begin{align*}
	& B(\Delta)  && \text{the full skew-symmetric matrix of $\Delta$}\\
	& B_\Delta  && \text{the submatrix of $B(\Delta)$ with rows indexed by $\Delta_0^\mu$}\\
	& \rep J && \text{the category of finite-dimensional representations of $J$} &\\
	& \iota_! && \text{the extension-by-zero functor along an evident inclusion $\iota$} &\\
	& S_u && \text{the simple representation supported on the vertex $u$} &\\
	& P_u,\ I_u && \text{the projective cover and the injective envelope of $S_u$} &\\
	& \dv M && \text{the dimension vector of $M$} & \\
	& \tauh \mc{M} && \text{the representation obtained by forgetting the decorated part of $\tau \mc{M}$} & \\
	& \trop(\Delta,\S) && \text{the set of $\mu$-supported $\delta$-vectors of $(\Delta,\S)$} \\
	& E_i,\ E_i^\star && \text{the boundary and dual boundary representations attached to $i\in\Delta_0^{\op{fr}}$}\\
	& \mc{E}_i^\mu && \text{a general representation of $(\Delta,\S)_\mu$ of weight $-b_i$} \\
	& \ep_i,\ \epc_i && \text{the projective and injective weight vectors of $E_i$} \\
	& r_i= r_{\ep_i},\ l_i= l_{\ep_i} &&  \text{the raising and lowering operators attached to $i\in\Delta_0^{\op{fr}}$} \\
	& \rc_i^\star = \rc_{\ep_i^\star},\ \lc_i^\star = \lc_{\ep_i^\star} && \text{the dual raising and lowering operators attached to $i\in\Delta_0^{\op{fr}}$}\\
	& \rho_i, \lambda_i;\ \rho_i^\star, \lambda_i^\star && \text{the crystal parameters, and string lengths in the seminormal case}\\
	& \wt_i &&  \text{the $i$-th coordinate for the weight function $\wt$} \\
	& \uca(\Delta) && \text{the upper cluster algebra with the seed $\Delta$} \\
	& C_{\op{gen}} && \text{the generic cluster character}\\
	& R_i,L_i;\ R_i^\star,L_i^\star &&  \text{the $\k$-derivation of $\uca(\Delta)$ lifting $r_i,l_i$ and $r_i^\star,l_i^\star$} \\
	& \prec_{\mf{T}},\ \preceq_{\mathrm{prof}} && \text{the dominance order across seeds and the string-profile preorder} 
\end{align*}

\section{Representation Theory of Quivers with Potentials} \label{S:QP}
\subsection{Decorated Representations and Presentations} \label{ss:QPrep}
Let $\Delta$ be a finite quiver with no loops. For such a quiver, we associate a skew-symmetric matrix $B(\Delta)$ given by
$$ B(\Delta)(u,v) = |\text{arrows $u\to v$}| - |\text{arrows $v\to u$}|. $$ 
Following \cite{DWZ1}, we define a potential $\mc{S}$ on a quiver $\Delta$ as a (possibly infinite) linear combination of oriented cycles in $\Delta$.
More precisely, a {\em potential} is an element of the {\em trace space} $\Tr(\ckQ):=\ckQ/[\ckQ,\ckQ]$,
where $\ckQ$ is the completion of the path algebra $k\Delta$ and $[\ckQ,\ckQ]$ is the closure of the commutator subspace of $\ckQ$.
The pair $(\Delta,\mc{S})$ is a {\em quiver with potential}, or QP for short.
For each arrow $a\in \Delta_1$, the {\em cyclic derivative} $\partial_a$ on $\widehat{k\Delta}$ is defined to be the linear extension of
$$ \partial_a(a_1\cdots a_d)=\sum_{k=1}^{d}a^*(a_k)a_{k+1}\cdots a_da_1\cdots a_{k-1}. $$
For each potential $\mc{S}$, its {\em Jacobian ideal} $\partial \mc{S}$ is the closed (two-sided) ideal in $\ckQ$ generated by all $\partial_a \mc{S}$.
The {\em Jacobian algebra} $J(\Delta,\mc{S})$ is $\widehat{k\Delta}/\partial \mc{S}$.
A QP is {\em Jacobi-finite} if its Jacobian algebra is finite-dimensional.
We assume that the mutable restriction is Jacobi-finite. When a result
uses finite-dimensional projective or injective modules over the full
ice QP, we also assume that its Jacobian algebra $J$ is
finite-dimensional.
The one-vertex extension and the mutation of finite-dimensional
decorated representations are defined over completed path algebras.

\begin{definition} A decorated representation of the Jacobian algebra $J$ is a pair $\mc{M}=(M,M^-)$,
	where $M\in \rep J$, and $M^-$ is a finite-dimensional $k^{\Delta_0}$-module.
\end{definition}
\noindent 	By abuse of language, we also say that $\mc{M}$ is a representation of $(\Delta,\mc{S})$.
When appropriate, we will view an ordinary representation $M$ as the decorated representation $(M,0)$.

Following \cite{DF} we call a homomorphism between two projective representations a {\em projective presentation} (or presentation in short). 
Here projective presentations are complexes of finitely generated projective $J$-modules; they are complexes in $\rep J$ when $J$ is finite-dimensional. Their category is Krull--Schmidt.
Sometimes it is convenient to view a presentation $P_-\to P_+$ as elements in the homotopy category $K^b(\proj J)$ of bounded complexes of projective representations of $J$.
Our convention is that $P_-$ sits in degree $-1$ and $P_+$ sits in degree $0$.
We denote this subcategory of presentations by $K^{[-1,0]}(J)$.

We denote by $P_u$ (resp. $I_u$) the indecomposable projective (resp. injective) representation of $J$ corresponding to the vertex $u$ of $\Delta$. 
For $\beta \in \mb{Z}_{\geq 0}^{\Delta_0}$ we write $P(\beta)$ for $\bigoplus_{u\in \Delta_0} \beta(u)P_u$. 
\begin{definition}\footnote{The $\delta$-vector is the same one defined in \cite{DF}, but is the negative of the $\g$-vector defined in \cite{DWZ2}. }
	The {\em $\delta$-vector} (or {\em weight vector}) of a presentation 
	$$ d: P(\beta_-)\to P(\beta_+) $$
	is the difference $\beta_+-\beta_- \in \mb{Z}^{\Delta_0}$.
	When working with injective presentations 
	$$ \dc: I(\betac_+)\to I(\betac_-), $$
	we call the vector $\betac_+ - \betac_-$ the {\em $\check{\delta}$-vector} of $\dc$.
\end{definition}
\noindent The $\delta$-vector is just the corresponding element in the Grothendieck group of $K^b(\proj J)$.

Let $\nu$ be the Nakayama functor $\Hom(-,J)^*$.
There is a map still denoted by $\nu$ sending a projective presentation to an injective one
$$ P_-\to P_+\ \mapsto\ \nu(P_-) \to \nu(P_+). $$
Note that if there is no direct summand of the form $P_i\to 0$, then $\ker(\nu d) = \tau\coker(d)$ where $\tau$ is the classical Auslander-Reiten translation. 

Let $\mc{R}ep(J)$ be the set of decorated representations of $J$ up to isomorphism. When $J$ is finite-dimensional, there is a bijection between the isomorphism classes in $\mc{R}ep(J)$ and $K^{[-1,0]}(\proj J)$ mapping any representation $M$ to its minimal presentation in $\rep J$, and the simple representation $S_u^-$ of $k^{\Delta_0}$ to $P_u\to 0$.
We also denote $P_u\to 0$ by $P_u[1]$.
Dually, $M$ corresponds to its minimal injective presentation and
$S_u^-$ to $0\to I_u$. 
Now we can naturally extend the classical AR-translation to decorated representations:
$$\xymatrix{\mc{M} \ar[r]\ar@{<->}[d] & \tau \mc{M} \ar@{<->}[d] \\ d_{\mc{M}} \ar[r] & \nu(d_{\mc{M}})}$$
Note that this definition agrees with the one in \cite{DF}.

Suppose that $\mc{M}$ corresponds to a projective presentation
$d_{\mc{M}}: P(\beta_-)\to P(\beta_+)$. 
Consider the resolution of the simple module $S_u$
\begin{align}
	\label{eq:Su_proj} \cdots \to \bigoplus_{h(a)=u} P_{t(a)}\xrightarrow{_a(\partial_{[ab]})_b} \bigoplus_{t(b)=u} P_{h(b)} \xrightarrow{_b(b)}  P_u \to S_u\to 0,\\
	\label{eq:Su_inj} 0\to S_u \to I_u \xrightarrow{(a)_a} \bigoplus_{h(a)=u} I_{t(a)} \xrightarrow{_a(\partial_{[ab]})_b} \bigoplus_{t(b)=u} I_{h(b)} \to \cdots.
\end{align}
Applying $\Hom(M,-)$ and $\Hom(-,M)$ to \eqref{eq:Su_inj} and \eqref{eq:Su_proj} gives the multiplicities in the minimal projective and injective presentations:
\begin{align} \label{eq:Betti} \beta_-(u)&=\dim(\ker \alpha_u/\img \gamma_u)+\dim M^-(u), \text{ and } \beta_+(u)=\dim \coker \alpha_u, \\
	\label{eq:Bettidual} \betac_-(u)&=\dim(\ker \gamma_u/\img \beta_u)+\dim M^-(u), \text{ and } \betac_+(u)=\dim \ker \beta_u. 
\end{align}
Here, the maps $\alpha_u$, $\beta_u$, and $\gamma_u$ are depicted in the following diagram as in \cite[Section 10.1]{DWZ1}.
\begin{equation} \label{eq:abc} \vcenter{\xymatrix@C=5ex{
		& M(u) \ar[dr]^{\beta_u} \\
		\bigoplus_{h(a)=u} M(t(a)) \ar[ur]^{\alpha_u} && \bigoplus_{t(b)=u} M(h(b)) \ar[ll]^{\gamma_u} \\
}} \end{equation}

The $\delta$-vector $\delta_{\mc{M}}$ of $\mc{M}$ is by definition the $\delta$-vector of $d_{\mc{M}}$.
If working with the injective presentations, we can define the $\dtc$-vector $\dtc_{\mc{M}}$ of $\mc{M}$.
It follows from \eqref{eq:Betti} and \eqref{eq:Bettidual} that $\delta_{\mc{M}}$ and $\dtc_{\mc{M}}$ are related by
\begin{equation}\label{eq:delta2dual} \dtc_{\mc{M}} = \delta_{\mc{M}} + (\dv M) B(\Delta).\end{equation}

\begin{definition}[{\cite{DWZ2,DF}}] \label{D:HomE} Given any projective presentation $d: P_-\to P_+$ and any $N\in\rep(J)$, we define $\Hom(d,N)$ and $\E(d,N)$ to be the kernel and cokernel of the induced map:
	\begin{equation}\label{eq:HE} 0\to \Hom(d,N)\to \Hom(P_+,N) \xrightarrow{} \Hom(P_-,N) \to \E(d, N)\to 0.\end{equation}
	Similarly for an injective presentation $\dc: I_+\to I_-$, we define $\Hom(M,\dc)$ and $\Ec(M,\dc)$ to be the kernel and cokernel of the induced map $\Hom(M,I_+) \xrightarrow{} \Hom(M,I_-)$.
	It is clear that 
	$$ \Hom(d,N) = \Hom(\coker(d),N)\ \text{ and }\ \Hom(M,\dc) = \Hom(M,\ker(\dc)). $$
	We set $\Hom(\mc{M},\mc{N})=\Hom(d_{\mc{M}},N)=\Hom(M,\dc_{\mc{N}})$, 
	$\E(\mc{M},\mc{N}) := \E(d_{\mc{M}},N)$ and $\Ec(\mc{M},\mc{N}) := \Ec(M,\dc_{\mc{N}})$.
\end{definition}	
\noindent Note that according to this definition, we have that $\Hom(\mc{M},\mc{N}) = \Hom(M,N)$.\footnote{This definition is slightly different from the one in \cite{DWZ2}, which involves the decorated part.}
We also set $\E(d_{\mc{M}},d_{\mc{N}}) = \E(\mc{M},\mc{N})$ and $\Ec(\dc_{\mc{M}},\dc_{\mc{N}}) = \Ec(\mc{M},\mc{N})$.
We refer readers to \cite{DF} for an interpretation of $\E(\mc{M},\mc{N})$ in terms of the presentations $d_{\mc{M}}$ and $d_{\mc{N}}$.
We call $\mc{M}$ or $d_{\mc{M}}$ {\em rigid} if $\E(\mc{M},\mc{M})=0$.

\begin{lemma}[{\cite[Corollary 10.8 and Proposition 7.3]{DWZ2}, \cite[Corollary 7.6]{DF}}]  \label{L:H2E} We have the following equalities:
	\begin{enumerate}
		\item{} $\E(\mc{M},\mc{N})=\Hom(\mc{N},\tau\mc{M})^*\text{ and }\Ec(\mc{M},\mc{N})=\Hom(\tau^{-1}\mc{N},\mc{M})^*.$
		\item{} $\E(\mc{M},\mc{M})=\Ec(\mc{M},\mc{M})=\E(\tau\mc{M},\tau\mc{M})$.
	\end{enumerate}
\end{lemma}

\subsection{Mutation of Quivers with Potentials}
In \cite{DWZ1} and \cite{DWZ2}, the mutation of quivers with potentials is introduced to model the cluster algebras.
The {\em mutation} $\mu_u$ of a QP $(\Delta,\mc{S})$ at a vertex $u$ is defined as follows.
The first step is to define the following new QP $\wtd{\mu}_u(\Delta,\mc{S})=(\wtd{\Delta},\wtd{\mc{S}})$.
We put $\wtd{\Delta}_0=\Delta_0$ and $\wtd{\Delta}_1$ is the union of three different kinds
\begin{enumerate}
	\item[$\bullet$] all arrows of $\Delta$ not incident to $u$,
	\item[$\bullet$] a composite arrow $[ab]$ from $t(a)$ to $h(b)$ for each $a,b$~with~$h(a)=t(b)=u$,
	\item[$\bullet$] an opposite arrow $a^\star$ (resp. $b^\star$) for each incoming arrow $a$ (resp. outgoing arrow $b$) at $u$.
\end{enumerate}
The new potential on $\wtd{\Delta}$ is given by
$$ \wtd{\mc{S}}:=[\mc{S}]+\sum_{h(a)=t(b)=u}b^\star a^\star[ab], $$
where $[\mc{S}]$ is obtained by substituting $[ab]$ for each word $ab$ occurring in $\mc{S}$. Finally we define $(\Delta',\mc{S}')=\mu_u(\Delta,\mc{S})$ as the {\em reduced part} (\cite[Definition 4.13]{DWZ1}) of $(\wtd{\Delta},\wtd{\mc{S}})$.
For this last step, we refer readers to \cite[Section 4,5]{DWZ1} for details.
A mutation sequence is \emph{admissible} for $(\Delta,\S)$ if each of its steps is defined. 
If all sequences are admissible for $(\Delta,\S)$ then we call $(\Delta,\S)$ {\em nondegenerate}.

Now we start to define the mutation of decorated representations of $J:=J(\Delta,\mc{S})$.
Recall the triangle of linear maps \eqref{eq:abc} with $\alpha_u\gamma_u=0$ and $\gamma_u\beta_u =0$.
We first define a decorated representation $\wtd{\mc{M}}=(\wtd{M},\wtd{M}^-)$ of $\wtd{\mu}_u(\Delta,\mc{S})$.
We set \begin{align*}
	&\wtd{M}(v)=M(v),\quad  \wtd{M}^-(v)=M^-(v)\quad (v\neq u); \\
	&\wtd{M}(u)=\frac{\ker \gamma_u}{\img \beta_u}\oplus \img \gamma_u \oplus \frac{\ker \alpha_u}{\img \gamma_u} \oplus M^-(u),\quad \wtd{M}^-(u)=\frac{\ker \beta_u}{\ker \beta_u\cap \img \alpha_u}.
\end{align*}
We then set $\wtd{M}(a)=M(a)$ for all arrows not incident to $u$, and $\wtd{M}([ab])=M(ab)$.
It is defined in \cite{DWZ1} a choice of linear maps
$\wtd{M}(a^\star)$ and $\wtd{M}(b^\star)$ making
$\wtd{M}$ a representation of $(\wtd{\Delta},\wtd{\mc{S}})$.
We refer readers to \cite[Section 10]{DWZ1} for details.
Finally, we define $\mc{M}'=\mu_u(\mc{M})$ to be the {\em reduced part} (\cite[Definition 10.4]{DWZ1}) of $\wtd{\mc{M}}$.

Let us recall several formula relating the $\delta$-vector of $\mc{M}$ and its mutation $\mu_{u}(\mc{M})$. We will use the notation $[b]_+$ for $\max(b,0)$.
\begin{lemma}[{\cite[Lemma 5.2]{DWZ2}}] \label{L:gdmu} Let $\delta=\delta_{\mc{M}}$ and $\delta'=\delta_{\mu_u(\mc{M})}$.
	We use the similar notation for $\dtc=\dtc_{\mc{M}}$ and the dimension vectors $d=\dv(M)$. Then 
	\begin{align*} \delta'(v) &= \begin{cases} -\delta(u) & \text{if $v=u$}\\ \delta(v) - [b_{v,u}]_+\beta_-(u) + [-b_{v,u}]_+\beta_+(u) & \text{if $v\neq u$.} \end{cases}\\
		\dtc'(v) &= \begin{cases} -\dtc(u) & \text{if $v=u$}\\ \dtc(v) - [b_{u,v}]_+\betac_-(u) + [-b_{u,v}]_+\betac_+(u) & \text{if $v\neq u$.} \end{cases}\\
		d'(v) &= \begin{cases} d [b_u]_+ - d(u) + \beta_+(u) + \betac_-(u) &\qquad \text{ if $v= u$}\\ d(v) &\qquad \text{ if $v\neq u$}. \end{cases}
	\end{align*} 
	where $b_u$ is the $u$-th column of the matrix $B(\Delta)$.
\end{lemma}
\noindent We remark that the mutated $\delta$-vector $\delta'$ is {\em not} completely determined by $\delta$ (we need $\beta_-$ and $\beta_+$). But see also Remark \ref{r:genmu}.

\begin{lemma}\label{L:HEmu} {\em \cite[Proposition 6.1, and Theorem 7.1]{DWZ2}} Let $\mc{M}'=\mu_u(\mc{M})$ and $\mc{N}'=\mu_u(\mc{N})$. 
	We have that \begin{enumerate}
		\item{} $\hom(\mc{M}',\mc{N}')-\hom(\mc{M},\mc{N})=\beta_{-,\mc{M}}(u)\betac_{-,\mc{N}}(u)-\beta_{+,\mc{M}}(u)\betac_{+,\mc{N}}(u)$;
		\item{} $\e(\mc{M}',\mc{N}')-\e(\mc{M},\mc{N})=\beta_{+,\mc{M}}(u)\beta_{-,\mc{N}}(u)-\beta_{-,\mc{M}}(u)\beta_{+,\mc{N}}(u)$;  
		\item[($2^*$)] $\ec(\mc{M}',\mc{N}')-\ec(\mc{M},\mc{N})=\betac_{-,\mc{M}}(u) \betac_{+,\mc{N}}(u)-\betac_{+,\mc{M}}(u)\betac_{-,\mc{N}}(u)$.
	\end{enumerate}
	In particular, $\e(\mc{M},\mc{M})$ and $\ec(\mc{M},\mc{M})$ are mutation invariant. So any reachable representation is rigid.
\end{lemma}

\begin{lemma}\label{L:taucommu} {\em \cite[Proposition 7.10]{DF}} The AR-translation $\tau$ commutes with the mutation $\mu_u$ at any vertex $u$.
\end{lemma}

\begin{definition}\label{D:extdmu} An {\em extended mutation sequence} is a composition of mutations $\mu_u$ and the $AR$-translation $\tau$ or its inverse $\tau^{-1}$.
We also denote $\tau$ and $\tau^{-1}$ by $\mu_+$ and $\mu_-$ respectively, though they are not involutions in general.
\end{definition}
\noindent Lemmas \ref{L:HEmu} and \ref{L:H2E} implies that $\e(\mc{M},\mc{M})$ is extended-mutation-invariant.

\indent We say a decorated representation $\mc{M}$ of $(\Delta,\mc{S})$ {\em negative reachable} or just {\em reachable} if there is a sequence of mutations $\mub$ such that $\mub(\mc{M})$ is {\em negative}, i.e., $\mub(\mc{M})$ has only the decorated part.
Similarly we say $\mc{M}$ {\em positive reachable} if there is a sequence of mutations $\mub$ such that $\mub(\mc{M})$ is a projective representation.
More generally, we say $\mc{M}$ {\em extended reachable} if there is an extended sequence of mutations $\mub$ such that $\mub(\mc{M})$ is negative.

\subsection{Extension of Quivers with Potentials}
Let $(\Delta,\mc{S})$ be a quiver with potential, and $\mc{V}$ be a decorated representation of $(\Delta,\S)$.
By the {\em extension} of $(\Delta,\mc{S})$ by $\mc{V}$, we mean the following construction.

\begin{construction} \label{cons:ext}
We start with $(\Delta,\mc{S})$ and a new vertex $v$.
Take the projective presentation $d_{\mc{V}}: P(\beta_-)\xrightarrow{} P(\beta_+)$ corresponding to $\mc{V}$. We assume that $P(\beta_-)$ and $P(\beta_+)$ share no common summands.
Note that this is always the case if $d_{\mc{V}}$ is in general position \cite{IOTW}.
Then we draw $\beta_+(w)$ arrows from $v$ to $w$ and $\beta_-(u)$ arrows from $u$ to $v$.
We choose lifts of the matrix entries of $d_{\mc{V}}$ to the completed path algebra. For rigid $d_{\mc{V}}$, the choice of lifts does not affect the isomorphism class of the extended Jacobian algebra (Lemma~\ref{L:lift-independence}).
For each entry of $c:=d_{\mc{V}}(u,w): P_u \to P_w$, we add the potential $a b c$ to the original potential $\mc{S}$
where $a$ is the added arrow corresponding to $P_u$ and $b$ is the added arrow corresponding to $P_w$.
We denote the resulting quiver with potential by $(\Delta[\mc{V}], \mc{S}[\mc{V}])$ or, for short $(\Delta,\mc{S})[\mc{V}]$ or $(\Delta,\mc{S})[d_{\mc{V}}]$, and abbreviate its Jacobian algebra to $J[\mc{V}]$.
If we restrict $(\Delta,\mc{S})[\mc{V}]$ on $\Delta$ in the sense of \cite[Definition 8.8]{DWZ1},
then we get the original QP $(\Delta,\mc{S})$ back.
\end{construction}

There is an obvious dual construction $(\Delta,\mc{S})[\mc{V}]^*$ from the injective presentation $\dc_{\mc{V}}: I_+\to I_-$ corresponding to $\mc{V}$.
It is easy to see that $(\Delta,\mc{S})[\mc{V}] = (\Delta,\mc{S})[\tau\mc{V}]^*$.
Here the star inside the brackets denotes the injective construction, not vector-space duality.
Let $\iota=\iota_{\mc{V}}:\Delta\hookrightarrow \Delta[\mc{V}]$ be the natural inclusion.
For any decorated representation $\mc{M}$ of $(\Delta,\mc{S})$, we denote by $\iota_!(\mc{M})$ the extension of $\mc{M}$ by zero to $\Delta[\mc{V}]$; explicitly, both the ordinary and decorated parts are zero at the new vertex $v$.
It is easy to see from the definition of the new potential $\mc{S}[\mc{V}]$ that
$\iota_!(\mc{M})$ is in fact a representation of $(\Delta,\mc{S})[\mc{V}]$.
When writing a vector in $\mb{Z}^{\Delta[\mc{V}]}$, our convention is to let the new vertex $v$ correspond to the last coordinate.
We will find it convenient to introduce the notation $\tauh \mc{V}$ to denote the representation obtained from $\tau \mc{V}$ by forgetting the decorated part.

\begin{lemma} \label{L:iota-zero} Consider $\iota_!(\mc{M})$ as a representation of $(\Delta,\mc{S})[\mc{V}]$. We have that
	\begin{enumerate}
		\item The $\dtc$-vector of $\iota_!(\mc{M})$ is equal to $(\dtc_{\mc{M}}, -\hom(V,M))$.
		\item The $\delta$-vector of $\iota_!(\mc{M})$ is equal to $(\delta_{\mc{M}}, -\e(\mc{V},M))$.
		\item Let $C_v=e_vJ[\mc V]e_v$, and let $\mathfrak m_v$ be its augmentation ideal, the closed span of positive-length return paths at $v$. Then
$\mathfrak m_v/\overline{\mathfrak m_v^2}\cong\E(\mc V,\mc V)$,
where the square is taken in $C_v$ and the closure in the path-length
topology. In particular, $\mc V$ is rigid if and only if $C_v=ke_v$.
	\end{enumerate}
\end{lemma}

\begin{proof} (1). Recall from the equality \eqref{eq:Bettidual}. 
	We observe that the spaces $\ker(\beta_u),\ \ker(\gamma_u)$, and $\img(\beta_u)$ at $u\in \Delta_0$ are invariant under the extension-by-zero functor. So the $\dtc$-vector is invariant at each $u$.
	Recall the construction of $(\Delta,\mc{S})[\mc{V}]$ using the presentation $d_{\mc{V}}$.
	We see that the map $\gamma_v$ is exactly $d_{\mc{V}}(M)$ so $\ker(\gamma_v)$ can be identified with $\Hom(V,M)$, while $\ker(\beta_v)$ and $\img(\beta_v)$ vanish for $\iota_!(M)$.
	
	(2). The proof is similar to (1).
	
	(3). We observe that the quotient of $J[\mc{V}]$ by the ideal generated by all incoming arrows to $v$ is isomorphic to the one-point extension algebra $\sm{A & 0 \\  V & k}$ 
	while the quotient of $J[\mc{V}]$ by the ideal generated by all outgoing arrows from $v$ is isomorphic to the one-point coextension algebra $\sm{A & ({\tauh \mc{V}})^* \\  0 & k}$.

Write $d=d_{\mc V}:P_-\to P_+$, and index the new incoming and
outgoing arrows by $a_r$ and $b_s$, respectively. The extension
potential is $W=\S+\sum_{r,s}a_rb_sD_{rs}$, where $D_{rs}$ are the
chosen lifts of the entries of $d$.

Modulo $\overline{\mathfrak m_v^2}$, paths visiting $v$ internally
vanish, leaving only paths $b_spa_r$ with $p$ in $\Delta$.
For an old arrow $c$, every term of $\partial_cW-\partial_c\S$
passes through $v$; completing it to a return path therefore gives
a product of two return paths. Thus the old-arrow derivatives impose
exactly the old Jacobian relations on $p$. The new-arrow derivatives
$\partial_{a_r}W=\sum_s b_sD_{rs}$ and
$\partial_{b_s}W=\sum_rD_{rs}a_r$ impose, respectively,
$d\End_A(P_-)$ and $\End_A(P_+)d$. Hence
$$ \mathfrak m_v/\overline{\mathfrak m_v^2} \cong \frac{\Hom_A(P_-,P_+)}{\End_A(P_+)d+d\End_A(P_-)} \cong\E(d,V). $$
For the second isomorphism, compose with $P_+\twoheadrightarrow V$.
Projectivity gives kernel $d\End_A(P_-)$, while
$\End_A(P_+)d$ maps onto $\Hom_A(P_+,V)d$.
For the first, taking the first-return component and then this matrix
quotient gives an inverse. It is continuous because sufficiently long
paths act trivially on $V$, so it also annihilates the closed
Jacobian ideal.

If $\E(\mc V,\mc V)=0$, then
$\mathfrak m_v=\overline{\mathfrak m_v^2}$. For the closed path-length
filtration $F^nC_v$, this implies
$\mathfrak m_v\subseteq F^{2^n}C_v$ for every $n$.
Separatedness gives $\mathfrak m_v=0$. The converse follows from the
isomorphism above.
\end{proof}

\begin{corollary} \label{C:iota-zero} If $\mc{V}$ is rigid, then \begin{enumerate}
		\item $\delta_{\iota_!(\mc{V})}=(\delta_{\mc{V}}, 0)$ and $\dtc_{\iota_!(\mc{V})}=(\dtc_{\mc{V}}, -\dtc_{\mc{V}}(\dv V))$.
		\item We have exact sequences $0\to \iota_!(V)\to P_v\to S_v\to 0$ and $0\to S_v\to I_v\to \iota_!({\tauh\mc{V}})\to 0$.
	\end{enumerate}
\end{corollary}
\begin{proof} (1) follows immediately from Lemma \ref{L:iota-zero}.
	
	(2). By Lemma \ref{L:iota-zero}.(3), every positive-length return
path at $v$ vanishes. Thus killing the incoming arrows leaves
$P_v=e_vJ[\mc V]$ unchanged. The one-point extension quotient in that
proof identifies its radical with $\iota_!(V)$, giving the first exact
sequence. Dually, killing the outgoing arrows leaves $J[\mc V]e_v$
unchanged, and the one-point coextension quotient gives the second
sequence after dualizing. 
\end{proof}

\begin{lemma}[Lift-independence]\label{L:lift-independence}
Let $(\Delta,\mathcal S)$ be reduced, and let
$d:P(\beta_-)\to P(\beta_+)$ be a rigid presentation over
$J(\Delta,\mathcal S)$, with no common indecomposable summands in its two
terms. Fix bases for these summands, and let $D_0,D_1$ be two lifts of its
matrix entries to the completed path algebra. Let $\Delta[d]$ be the
extended quiver of Construction~\ref{cons:ext}, with new vertex $v$, and put
$$ W_\ell=\mathcal S+\sum_{r,s}a_r b_s(D_\ell)_{rs}, \qquad A_\ell=J(\Delta[d],W_\ell),\qquad \ell=0,1. $$
Then $A_0\cong A_1$ by an isomorphism preserving all vertex idempotents
and inducing the identity on the quotient $J(\Delta,\mathcal S)$ obtained
by deleting $v$. In fact, the two extension QPs are right-equivalent.
\end{lemma}

\begin{proof}
Write $R=\widehat{k\Delta[d]}$ and let $\mathfrak J(W)$ denote the closed
Jacobian ideal of a potential $W$ in $R$. All products of closed ideals
below are understood with closure. Since the entries of $D_1-D_0$ belong
to the old Jacobian ideal, cyclic rearrangement gives
$W_1-W_0\sim_{\mathrm{cyc}}\sum_{c\in\Delta_1}H_c\,\partial_c\mathcal S$,
where $H_c$ has the same endpoints as $c$ and is a convergent linear
combination of paths passing through $v$ exactly once. For formal
potentials this factorization follows from \cite[Lemma~4.12]{DWZ1},
followed by taking the component containing exactly one passage through
$v$. In particular, each $H_c$ has length at least two.

The substitutions $c\mapsto c+H_c$ for $c\in\Delta_1$, with the new arrows
fixed, define an automorphism $\phi$ of $R$ by
\cite[Proposition~2.4]{DWZ1}. Expansion gives
$\phi(W_0)\sim_{\mathrm{cyc}}W_1+T$,
where every cyclic path occurring in $T$ passes through $v$ at least
twice: the first-order change in $\mathcal S$ is $W_1-W_0$, whereas every
remaining term contains at least two passages through $v$.

By rigidity and Corollary~\ref{C:iota-zero},
$e_vA_1e_v=ke_v$. Hence every positive-length path from $v$ to $v$
belongs to $\mathfrak J(W_1)$. Cutting a cyclic path of $T$ at two
successive occurrences of $v$ writes it, up to cyclic equivalence, as a
product of two such paths. Thus $T$ is cyclically equivalent to an element
of $\mathfrak J(W_1)^2$. The potentials are reduced, so
\cite[Proposition~4.10]{DWZ1} yields
$\mathfrak J(W_1+T)=\mathfrak J(W_1)$.
Together with \cite[Proposition~3.7]{DWZ1}, this gives
$\phi\bigl(\mathfrak J(W_0)\bigr)=\mathfrak J\bigl(\phi(W_0)\bigr)=\mathfrak J(W_1)$,
proving the algebra isomorphism. Since each $H_c$ passes through $v$,
$\phi$ induces the identity after deleting $v$. Finally,
\cite[Proposition~4.10(2)]{DWZ1} shows that $W_1+T$ is right-equivalent
to $W_1$, proving the last assertion.
\end{proof}

\noindent Warning: One might think that if $V$ is rigid, then the Jacobian algebra is given by the matrix algebra $\sm{A & (\tau V)^* \\  V & k}$. But this is not true in general because the construction may introduce new paths between two vertices of $\Delta$.

\subsection{Mutation of Presentations}
Conversely, given a quiver with potentials $(\Delta,\mc{S})$ and a vertex $v\in \Delta_0$,
let $(\Delta,\mc{S})_{\hat{v}}$ be the restriction of $(\Delta,\mc{S})$ to the full subquiver of $\Delta_0\setminus\{v\}$.
\begin{definition} We call a vertex $v$ {\em simple in} $(\Delta,\S)$ if for each pair of arrows $a:u\to v$ and $b:v\to w$, the partial derivative $\partial_{[ab]}[\S]$ contains no arrows to $v$ or from $v$.
\end{definition}
\noindent The new vertex is simple in an extension from
Construction~\ref{cons:ext}. For an arbitrary reduced QP, we use the
following projected presentation.

Let $\psi_{\hat v}$ be the projection of the completed path algebra onto
the full subquiver obtained by deleting $v$. The \emph{presentation at
$v$} over $J((\Delta,\S)_{\hat v})$ is
\begin{equation}\label{eq:dv}
d_v:\bigoplus_{a:u\to v}P_u
\xrightarrow{\bigl(\psi_{\hat v}(\partial_{[ab]}[\S])\bigr)_{a,b}}
\bigoplus_{b:v\to w}P_w.
\end{equation}
Only cycles visiting $v$ exactly once contribute to its matrix entries.
We also put $\dc_v=\nu(d_v)$. Since the QP is reduced, the entries of
$d_v$ lie in the radical, so it is a minimal two-term complex and
corresponds to a decorated representation. When $v$ is simple and the
two terms have no common indecomposable summands, this is the presentation
used in Construction~\ref{cons:ext}, and
$(\Delta,\S)_{\hat v}[d_v]=(\Delta,\S)$.

Let $\iota$ be the inclusion of the subquiver obtained by deleting
$v$. For this projected presentation, the proof of
Lemma~\ref{L:iota-zero}(1)--(2) applies: on $\iota_!\mc M$, terms
visiting $v$ more than once vanish in $\gamma_v$.
Thus, with $v$ the last coordinate,
\begin{equation}\label{eq:vertex-extraction}
\begin{aligned}
\dtc_{\iota_!\mc M}=(\dtc_{\mc M},-\hom(d_v,M)),\qquad
\delta_{\iota_!\mc M}=(\delta_{\mc M},-\e(d_v,M)).
\end{aligned}
\end{equation}

We next show that the condition $e_vJe_v=ke_v$ allows us to choose a representative in which the potential contains no cycles visiting $v$ more than once.
\begin{lemma}\label{L:vertex-normal-form}
Let $(\Delta,\S)$ be a reduced QP, with Jacobian algebra $A$, and
suppose that $e_vAe_v=ke_v$. Let $\S_{\le1,v}$ be the part of $\S$
consisting of cycles that visit $v$ at most once. Then
$$ \mathfrak J(\S_{\le1,v})=\mathfrak J(\S), \qquad (\Delta,\S_{\le1,v})\sim_{\mathrm{r.e.}}(\Delta,\S). $$
In particular, $v$ is simple in $(\Delta,\S_{\le1,v})$, and the
restriction to $\Delta_0\setminus\{v\}$ is unchanged. The
right-equivalence can be chosen to induce the identity on the common
Jacobian algebra.
\end{lemma}
\begin{proof}
Write $\S=\S_{\le1,v}+T$ and set $\mathfrak J=\mathfrak J(\S)$.
Every positive-length path from $v$ to $v$ belongs to $\mathfrak J$,
because its image lies in the radical of $e_vAe_v=ke_v$. Cutting a
cycle in $T$ at two occurrences of $v$, up to cyclic permutation,
writes it as a product of two such paths. Hence $T$ is cyclically
equivalent to an element of the closed ideal $\overline{\mathfrak J^2}$.
Apply \cite[Proposition~4.10]{DWZ1} to the reduced potentials $\S$ and
$\S-T$. Its second part gives a right-equivalence congruent to the
identity modulo $\mathfrak J$. The other assertions follow from the
definition of $\S_{\le1,v}$.
\end{proof}

\begin{lemma}\label{L:vsimple}
Let $(\Delta,\S)$ be a reduced QP such that
$e_vJ(\Delta,\S)e_v=ke_v$. Every sequence of mutations away
from $v$ has a reduced representative in which $v$ is simple.
In particular, this applies to an extension by a rigid presentation
in Construction~\ref{cons:ext}.
\end{lemma}
\begin{proof}
For $u\ne v$, \cite[Corollary~6.6]{DWZ1} identifies the corner algebras
at $1-e_u$ before and after mutation, preserving the vertex
idempotents. Therefore the corner at $v$ remains $ke_v$.
Apply Lemma~\ref{L:vertex-normal-form} to a reduced representative after
each mutation. For an extension by a rigid presentation, the initial
corner condition follows from Corollary~\ref{C:iota-zero}.
\end{proof}

\begin{definition}\label{D:mupre}
Let $d_{\mc V}$ satisfy the hypotheses of Construction~\ref{cons:ext}.
Mutate its extension at $u\ne v$, take a reduced representative, and
extract $d'_v$ by \eqref{eq:dv}. We call $d'_v$ the \emph{mutation}
$\mu_u(d_{\mc V})$ of the presentation. 
\end{definition}
If $d_{\mc V}$ is rigid, we
may first use the simple representative supplied by Lemma~\ref{L:vsimple}.
This normalization does not change the extracted presentation.
Presentation identities are understood up to isomorphism; Lemma \ref{L:muprep}
shows that the definition agrees with DWZ mutation of representations.
The following lemma was proved in \cite{DF} as an easy consequence of a result of K. Bongartz \cite{Bo}.
\begin{lemma} \label{L:equalM} Two decorated representations $\mc{M}$ and $\mc{M'}$ are isomorphic 
if and only if for any $\mc{N}\in\mc{R}ep(J)$ we have that $\hom(\mc{M},\mc{N})=\hom(\mc{M'},\mc{N})$ and $\e(\mc{M},\mc{N})=\e(\mc{M'},\mc{N})$.
\end{lemma}

\begin{lemma}\label{L:muprep}
Assume that the terms of the presentation $d_v$
in \eqref{eq:dv} have no common indecomposable summands. If $\mc V$
corresponds to $d_v$, then the presentation extracted at $v$ after
mutation at $u\neq v$ corresponds to $\mu_u(\mc V)$. In particular,
$$ \mu_u(d_{\mc V})\cong d_{\mu_u(\mc V)}. $$
For a rigid extending presentation, the reduced mutation can consequently
be chosen as the extension of the mutated base QP by $\mu_u(\mc V)$.
\end{lemma}
\begin{proof}
Let $\mc V'$ correspond to the extracted presentation after mutation,
and let $\iota,\iota'$ denote the inclusions of the base quivers before
and after mutation. Restriction commutes with mutation up to the usual
identifications, and
$$ \iota'_!(\mu_u(\mc N))\cong\mu_u(\iota_!\mc N) $$
for every decorated representation $\mc N$ of the base QP; this follows
from the DWZ formulas, since the ordinary and decorated spaces at $v$
are zero. Use compatible reduced representatives in these identifications.

By \eqref{eq:vertex-extraction} and Lemma~\ref{L:gdmu},
\begin{align*}
\hom(\mc V',\mu_u(\mc N))-\hom(\mc V,\mc N)&=[b_{u,v}]_+\betac_{-,\mc N}(u)-[-b_{u,v}]_+\betac_{+,\mc N}(u)\\
&=\beta_{-,\mc V}(u)\betac_{-,\mc N}(u)-\beta_{+,\mc V}(u)\betac_{+,\mc N}(u).
\end{align*}
The last equality uses the absence of common summands in $d_v$.
Lemma~\ref{L:HEmu}(1) gives the same difference with $\mu_u(\mc V)$
in place of $\mc V'$. Similarly, the projective-weight equality in
\eqref{eq:vertex-extraction} gives
$$ \e(\mc V',\mu_u(\mc N))-\e(\mc V,\mc N)= \beta_{+,\mc V}(u)\beta_{-,\mc N}(u)-\beta_{-,\mc V}(u)\beta_{+,\mc N}(u), $$
which agrees with Lemma~\ref{L:HEmu}(2). Since mutation is a bijection
on the relevant decorated isomorphism classes, Lemma~\ref{L:equalM}
identifies $\mc V'$ with $\mu_u(\mc V)$.

For an extension by a rigid presentation, the extracted mutation is rigid. The corner hypothesis is
preserved by mutation, and Lemma~\ref{L:vertex-normal-form} gives a simple
representative without changing the extracted matrix. Reconstruction by
Construction~\ref{cons:ext} proves the last assertion. Lift-independence
is supplied by Lemma~\ref{L:lift-independence}.
\end{proof}


\section{General Presentations and Tropical $F$-polynomials} \label{S:GPTF}
\subsection{General Presentations} We shall start our discussion by reviewing some results in \cite{DF}. 
We will consider a more general setting where the algebra $A$ is any basic finite-dimensional $k$-algebra, which can be presented as $k\Delta / I$.
All presentation spaces and genericity assertions in this section are
in this finite-dimensional setting. The standing assumption supplies
it for $J_\mu$; application to a full ice algebra $J$ additionally
uses frozen-Jacobi-finiteness, by Lemma~\ref{L:ice-finite}.

Any $\delta\in \mb{Z}^{\Delta_0}$ can be written as $\delta = \delta_+ - \delta_-$ where $\delta_+=\max(\delta,0)$ and $\delta_- = \max(-\delta,0)$. Here the maximum is taken coordinate-wise. We put 
$$ \PHom_A(\delta):=\Hom_A(P(\delta_-),P(\delta_+)). $$
We say that a {\em general} presentation in $\PHom_A(\delta)$ has property $\heartsuit$ if there is some nonempty Zariski-open (and thus dense) subset $U$ of $\PHom_A(\delta)$ such that all presentations in $U$ have property $\heartsuit$.
For example, a general presentation $d$ in $\PHom_A(\delta)$ has the following properties:
\begin{enumerate}
	\item $\Hom(d,N)$ has constant dimension for a fixed $N\in \rep A$.
	\item $\Gr_\gamma(\coker(d))$ has constant topological Euler characteristic.
\end{enumerate}
Note that (1) implies that $\E(d,N)$ has constant dimension on $U$ as well. 
We denote these two generic values by $\hom(\delta,N)$ and $\e(\delta,N)$.
If we apply (1) to $N=A^*$, then $\coker(d)$ has a constant dimension vector, which will be denoted by $\dv(\delta)$.

Dually, we put
$\IHom_A(\dtc)=\Hom_A(I([\dtc]_+),I([-\dtc]_+))$, and use
$\hom(M,\dtc)$ and $\ec(M,\dtc)$ for its generic pairings.
When both arguments are weights, $\hom,\e,\ec$ denote joint generic
values for independently varying presentations, also when the weights
coincide. Checked weights use injective presentations; for example,
$\hom(\delta,\epc)$ is the generic value of
$\hom(\coker d,\ker\check e)$.

\begin{remark} \label{r:genmu} It is known \cite{IOTW,DF} that the $\delta$-vector of a general presentation satisfies $\beta_+ = [\delta]_+$ and $\beta_- = [-\delta]_+$.
In particular, for general presentations, Lemma \ref{L:gdmu}.(1) reduces to the following rule of Fock-Goncharov \cite{FGc}:
	\begin{equation}\label{eq:mug}  \delta'(v)= \begin{cases} -\delta(u) & \text{if $v=u$,}\\ 
		\delta(v) - b_{v,u}[-\delta(u)]_+  & \text{if $b_{v,u}\geq 0$,} \\
		\delta(v) - b_{v,u}[\delta(u)]_+ & \text{if $b_{v,u}<0$.} 
	\end{cases} \end{equation}
	If one likes, one can combine the cases into one $\delta'(v)=\delta(v) + [b_{v,u}]_+\delta(u) - b_{v,u}[\delta(u)]_+$. 	
\end{remark}

The presentation space $\PHom_A(\delta)$ comes with a natural group action by $$ \Aut_A(\delta):=\Aut_A(P(\delta_-))\times \Aut_A(P(\delta_+)). $$
A rigid presentation in $\PHom_A(\delta)$ has a dense $\Aut_A(\delta)$-orbit \cite{DF}.
In particular, a rigid presentation is always general.


Due to the relation $\delta_{\tau^{-1}\mc{M}} = -\dtc_{\mc{M}}$ and \eqref{eq:delta2dual},
we have that for a general presentation $d$ of weight $\delta$, the $\delta$-vector of $\tau^{-1} d$ is constant.
We denote this constant vector by $\tau^{-1}\delta$.
Here $\dtc=\delta+\dv(\delta)B(\Delta)$.
\begin{theorem}[{\cite[Theorem 1.10]{GLFS}, \cite[Theorem 3.11]{Fg}}]\label{T:genpi} The following are equivalent
	\begin{enumerate}
		\item  $\mc{M}$ is a general representation of weight $\delta$; 
		\item  $\mc{M}$ is a general representation of dual weight $\dtc$;
		\item  $\tau^{-1} \mc{M}$ is a general representation of weight $\tau^{-1}\delta$;
		\item  $\mu_k(\mc{M})$ is a general representation of weight $\mu_k(\delta)$.
	\end{enumerate}
\end{theorem}

\subsection{Generic Extensions of Quivers with Potentials}\label{ss:generic-ext}
\begin{lemma}\label{L:DWZlocal}
	Let $Q=(\Delta,S)$ be a QP and let $I\subset \Delta_0$. Then we have the right equivalence:
	$$ (Q_{\red})|_I \sim_{\rm r.e.} (Q|_I)_{\red}. $$
	Moreover, for $k\in I$,
	$$ (\mu_k Q)|_I \sim_{\rm r.e.} \mu_k(Q|_I). $$
\end{lemma}

\begin{proof}
	By the Splitting Theorem \cite[Theorem 4.6]{DWZ1}, we may write
	$Q \sim_{\rm r.e.} Q_{\triv}\oplus Q_{\red}$, 
	and the reduced summand is unique up to right-equivalence
	\cite[Definition 4.13]{DWZ1}.  Restrict this right-equivalence to $I$ using
	the homomorphism $\psi_I$ of \cite[Definition 8.8]{DWZ1}.  The resulting
	map is again a right-equivalence: indeed, its linear part is still an
	isomorphism on each arrow space $A_{ij}$ with $i,j\in I$, and hence it is
	an isomorphism by \cite[Proposition 2.4]{DWZ1}.  Thus
	$	Q|_I \sim_{\rm r.e.} (Q_{\triv})|_I \oplus (Q_{\red})|_I $.
	The first summand is trivial, possibly zero, and the second summand is
	reduced.  Hence, by uniqueness of the reduced part,
	$(Q_{\red})|_I \sim_{\rm r.e.} (Q|_I)_{\red}$.
	
	For the second assertion, the premutation $\wtd{\mu}_k$ is local from formulas
	\cite[(5.3), (5.8)]{DWZ1}; hence
	$(\widetilde\mu_k Q)|_I=\widetilde\mu_k(Q|_I)$.
	Taking reduced parts and using the first assertion gives
	$$
	(\mu_k Q)|_I
	=
	((\widetilde\mu_k Q)_{\red})|_I
	\sim_{\rm r.e.}
	((\widetilde\mu_k Q)|_I)_{\red}
	=
	(\widetilde\mu_k(Q|_I))_{\red}
	=
	\mu_k(Q|_I).
	$$
\end{proof}

A quiver with designated frozen vertices is called an {\em ice quiver}. 
The vertices of $\Delta$ split into two disjoint sets $\Delta_0 = \Delta_0^{\mu} \sqcup \Delta_0^{\op{fr}}$.
An arrow between frozen vertices is called a {\em frozen arrow}.
The $B$-matrix $B_{\Delta}$ of an ice quiver $\Delta$ is obtained from the original $B$-matrix $B(\Delta)$ by removing the rows corresponding to the frozen vertices.
Note that the information on frozen arrows is lost in $B_{\Delta}$.

\begin{proposition}\label{P:gendv} Suppose that $\mc{S}$ is a generic potential on $\Delta$. Then \begin{enumerate}
\item If $v$ is simple in $(\Delta,\S)$, then $d_v$ is a general presentation of $(\Delta,\mc{S})_{\hat{v}}$.
\item If $\mc{V}$ corresponds to a very general presentation of $J(\Delta,\mc{S})$, then the extended $QP$ $(\Delta,\mc{S})[\mc{V}]$ is nondegenerate if we freeze $v$.
\end{enumerate}
\end{proposition}
\begin{proof} (1). By the construction of $d_v$, the matrix coefficients of $d_v$ come from the coefficients in the potential $\S$. The inverse image of a nonempty open subset is nonempty and open, proving (1).
	
(2). Fix a finite mutation sequence in the base quiver. By
Theorem~\ref{T:genpi}, the presentations and all their successive
mutations are general on a nonempty open subset of the initial
presentation space.
By Lemma~\ref{L:DWZlocal}, the restriction to the base QP remains
nondegenerate, so it suffices to exclude 2-cycles involving $v$.
At the initial seed the extracted presentation is $d_{\mc V}$, which
is general and has no common summands. Inductively, Lemma~\ref{L:muprep}
identifies the presentation extracted from a reduced mutation with the
mutated decorated representation, which is general by
Theorem~\ref{T:genpi}(4). The differential of the extracted
presentation lies in the radical, so the presentation is minimal.
Generality gives
$\beta_+=[\delta]_+$ and $\beta_-=[-\delta]_+$, so no vertex is
connected to $v$ in both directions. This excludes the required
2-cycles and allows the induction to continue. There are countably many finite sequences. Their nonempty open subsets
have nonempty intersection over $k=\mathbb C$; choosing the presentation
in this intersection proves the assertion. Here \emph{very general}
means satisfying such a countable collection of open conditions.
\end{proof}

\begin{construction}\label{cons:iQP} A reference ice QP can be constructed from an extended $B$-matrix as follows.
Let $B_\mu$ be the principal part of $B$, and $\Delta^\mu$ its associated quiver.
We choose a generic potential $\S^\mu$ on $\Delta^\mu$ satisfying the standing Jacobi-finiteness convention; the exchange matrices considered here are restricted to those admitting such a choice.
For each frozen column $b_v$ of $B$, we associate to it a very general presentation $d_v$ of weight $-b_v^{\mathsf T}$,
and let $\b{V}=\{d_v\}_{v}$.
Then we extend $(\Delta^\mu,\S^\mu)$ to $(\Delta^\mu,\S^\mu)[\b{V}]$ by Construction \ref{cons:ext}.
The extensions are simultaneous: add no arrows between new vertices, and sum their added potentials.
By Proposition \ref{P:gendv}.(2), $(\Delta^\mu,\S^\mu)[\b{V}]$ is a nondegenerate ice quiver with potential.
Indeed, by Lemma~\ref{L:DWZlocal}, mutation commutes with restriction
to the mutable vertices together with any one frozen vertex. Every new
$2$-cycle meeting a mutable vertex would appear in such a restriction.

Compatibility with mutation is expressed by the boundary presentations.
For each frozen vertex $v$, restrict to the mutable vertices together
with $v$, mutate, and extract the presentation by \eqref{eq:dv}.
Lemma~\ref{L:muprep} identifies it with $\mu_k(d_v)$, and
Theorem~\ref{T:genpi}(4) and \eqref{eq:mug} give the general presentation
of the required mutated weight. If $d_v$ is rigid, this one-frozen
restriction has a simple reduced representative right-equivalent to
$\mu_k(\Delta^\mu,\S^\mu)[\mu_k(d_v)]$.
For general nonrigid data, mutation compatibility concerns the
extracted presentations.
\end{construction}

%

\subsection{Tropical $F$-polynomials} \label{ss:tropF}
Motivated by the $F$-polynomial of $M$ defined in \cite{DWZ2}, we introduced its tropical version in \cite{Ft}. We will review the $F$-polynomial of $M$ in Section \ref{ss:genericB}.
\begin{definition}[{\cite{Ft}}] The {\em tropical $F$-polynomial} $f_M$ of a representation $M$ is the function $(\mb{Z}^{\Delta_0})^* \to \mb{Z}_{\geq 0}$ defined by
	$$ \delta \mapsto \max_{L\hookrightarrow M}{\delta(\dv L)}. $$
	The {\em dual} tropical $F$-polynomial $\fc_M$ of $M$ is the function $(\mb{Z}^{\Delta_0})^* \to \mb{Z}_{\geq 0}$ defined by
	$$ \delta \mapsto \max_{M\twoheadrightarrow N}{\delta(\dv N)}. $$
Here, a weight $\delta$ is viewed as an element in $(\mb{Z}^{\Delta_0})^*$ via the usual dot product.
\end{definition}
\noindent Clearly $f_M$ and $\fc_M$ are related by $f_M(\delta)-\fc_M(-\delta)= \delta(\dv M)$.
Moreover, it follows from \eqref{eq:HE} that for any presentation $d$ of weight $\delta$,
\begin{align}\label{eq:heform} \delta(\dv M) &= \hom(d,M) - \e(d,M);\\
\label{eq:hecform}	\check{\delta} (\dv M) &= \hom(M,\dc) - \ec(M,\dc).
\end{align}

\begin{theorem}[{\cite{Ft}}] \label{T:HomE} If $M$ is negative reachable, then for any $\delta,\dtc \in\mb{Z}^{\Delta_0}$ we have that
	\begin{align}
		\label{eq:HomEQP}	{f}_M(\delta) &= \hom(\delta,M), & \fc_M(-\delta) &= {\e}(\delta,M);\\
		\label{eq:HomEQPdual}	\fc_M(\check{\delta}) &= \hom(M,\check{\delta}), & {f}_M(-\check{\delta}) &= \ec(M,\check{\delta}).
	\end{align}
\end{theorem}
\begin{remark}\label{r:HomE} \cite[Theorem 3.6]{Ft} is a more general statement holding for any finite-dimensional basic algebra.
One special case is that when $\delta$ (resp. $\dtc$) satisfies $\e(\delta,\delta)=0$ (resp. $\ec(\dtc,\dtc)=0$), then \eqref{eq:HomEQP} (resp. \eqref{eq:HomEQPdual}) holds without any assumption for $M$.
\end{remark}

If $M$ is general of injective weight $\epc$ then we will write $f_\epc(\delta)$ for $f_M(\delta)$, and similarly write $\fc_\delta$ for a general representation of projective weight $\delta$.
\begin{conjecture}\label{c:dualityhom} For a nondegenerate quiver with potential, and any $\delta$ and $\epc$ we have that
	$$ \check{f}_{\delta}(\epc) = \hom(\delta,\epc)= {f}_\epc(\delta). $$
\end{conjecture}
\noindent This is a stronger version of a conjecture in \cite{Ft}, where we only conjecture the reciprocity $\check{f}_{\delta}(\epc) = {f}_\epc(\delta)$.

\section{The Lowering and Raising Operators} \label{S:rl}
\subsection{Lowering and Raising Operators for Extended-Reachable $\ep$} 
Schofield introduced the general rank for quiver representations in his theory of general representations \cite{S}.
The following lemma is a straightforward generalization of \cite[Lemma 5.1]{S}.
\begin{lemma} \label{L:genrank} Let $A$ be a finite-dimensional algebra. Given any two irreducible closed sets $X$ and $Y$ in representation varieties of $A$, there is a nonempty open subset $U$ of $X \times Y$ and a dimension vector $\gamma$ such that for $(M,N)\in U$ we have that $\hom_A(M,N)$ is minimal and $\{\phi\in\Hom_A(M,N)\mid \rank \phi = \gamma \}$ is open and non-empty in $\Hom_A(M,N)$. 
\end{lemma}
Set $P_-=P([-\delta]_+)$ and $P_+=P([\delta]_+)$. Let $\alpha$ be the maximal rank vector of $\Hom(P_-, P_+)$,
and $U$ be the open subset of $\Hom(P_-, P_+)$ attaining the maximal rank $\alpha$.
The cokernels of homomorphisms in $U$, up to isomorphism, form a dense subset of a component of $\rep_{\dv(\delta)}(A)$ \cite{DF, P, Fg},
and we call this component the {\em principal component} of $\delta$, denoted by $\PC(\delta)$.

\begin{definition} \label{D:genrank} If one of $X$ and $Y$ is a single representation, say $Y=\{E\}$, and $X$ is the principal component $\PC(\delta)$, then the above dimension vector is denoted by $\rank(\delta,E)$.
	If $X=\PC(\delta)$ and $Y=\PC(\ep)$, then this $\gamma$ is called the {\em general rank} from $\delta$ to $\ep$, denoted by $\rank(\delta,\ep)$. There are obvious variations if we replace $\delta$ or $\ep$ by a $\dtc$-vector.
\end{definition}

\noindent For quivers with potentials, we have that $\PC(\delta)=\PC(\dtc)$ by \cite[Theorem 3.11]{Fg}, so $\rank(\delta,\ep) = \rank(\dtc,\ep)$. 

\begin{definition} \label{D:rlgrank} For any decorated representation $\mc{E}=(E,E^-)$ of weight $\ep$, we define the two operators $r_{\mc{E}}$ and $l_{\mc{E}}$ on the set of $\delta$-vectors as follows:
\begin{align*} 	r_{\mc{E}} (\delta) &= \delta+\ep +\rank(E, \tau\delta) B(\Delta); \\
		 l_{\mc{E}} (\delta) &= \delta-\epc +\rank(\delta,E) B(\Delta). 
\intertext{We also define the two operators $\rc_{\mc{E}}$ and $\lc_{\mc{E}}$ on the set of $\dtc$-vectors}
	\rc_{\mc{E}} (\dtc) & =  \dtc +\epc -\rank(\tau^{-1}\dtc, E) B(\Delta);\\
		\lc_{\mc{E}} (\dtc) & =  \dtc -\ep -\rank(E, \dtc) B(\Delta).
	\end{align*}	
If $\mc{E}$ is general of weight $\ep$, then we will write $\ep$ instead of $\mc{E}$ in $r_{\mc{E}}$ and $l_{\mc{E}}$.
\end{definition}

\begin{remark}\label{r:rl}  In \cite{Fg} we also defined another two pairs of operators: $(r^{\ep},l^{\ep})$ on the set of $\dtc$-vectors and $(\rc^{\ep},\lc^{\ep})$ on the set of $\delta$-vectors.
\begin{align*}
	\notag   r^{\ep} (\dtc) & =  \dtc - \tau^{-1}\ep -\rank(\tau^{-1}\ep, \dtc) B(\Delta);	\\	
	\notag	 l^{\ep} (\dtc) &= \dtc + \tau^{-1}\epc - \rank(\tau^{-1}\dtc, \tau^{-1}\ep) B(\Delta),  
\shortintertext{and}
  	\notag \rc^\ep (\delta) &= \delta+\ep +\rank(\delta,\tau\epc) B(\Delta);   \\
    \notag \lc^\ep (\delta) & =  \delta + \tau\ep +\rank(\tau\ep, \tau\dtc) B(\Delta).	
\end{align*}
If $\ep$ is extended-reachable, Theorem~\ref{T:rlrigid} implies that the $\dtc$-vector of $r_{\ep}(\delta)$ is $r^{\ep}(\dtc)$ rather than $\rc_{\ep}(\dtc)$; 
and the $\delta$-vector of $\rc_{\ep}(\dtc)$ is $\rc^{\ep}(\delta)$ rather than $r_{\ep}(\delta)$.
It is also clear from the definition that $\rc_{\tau^{-1}\ep} = l^{\ep}$ and $\lc_{\tau^{-1}\ep}=r^{\ep}$.
\end{remark}

The following theorem is a direct consequence of \cite[Lemmas 5.7 and 5.9]{Fg}.
\begin{theorem} \label{T:rle} Suppose that
$\ep$ is extended-reachable. Then the operators $r_\ep$ and $l_\ep$ commute with any sequence of mutations and $\tau^i$:
\begin{align*} \mub(r_{\ep}(\delta)) &= r_{\mub(\ep)} (\mub(\delta)) \ &{ and }&& \ \mub(l_{\ep}(\delta)) &= l_{\mub(\ep)} (\mub(\delta)); \\
	\tau^i(r_{\ep}(\delta)) &= r_{\tau^i\ep} (\tau^i\delta) \ &{ and }&& \ \tau^i(l_{\ep}(\delta)) &= l_{\tau^i\ep} (\tau^i\delta).
\end{align*}
In particular, if $\b{u}$ is an extended sequence of mutations such that $\mub(\ep)$ is negative, then the operators $r_\ep$ and $l_\ep$ on the $\delta$-vectors of $(\Delta,\mc{S})$ are given by 
	\begin{align*} r_{\ep} (\delta) &= \mub^{-1} (\mub(\delta) + \mub(\ep)); \\
		l_{\ep} (\delta) &= \mub^{-1} (\mub(\delta) - \mub(\ep)).
	\end{align*}
If $\b{u}$ is an extended sequence of mutations such that $\mub(\epc)$ is negative, then the operators $\rc_\ep$ and $\lc_\ep$ on the $\dtc$-vectors of $(\Delta,\mc{S})$ are given by
\begin{align*}  \rc_{\ep} (\dtc) &= \mub^{-1} (\mub(\dtc)+\mub(\epc)) ; \\
		\lc_{\ep} (\dtc) &= \mub^{-1} (\mub(\dtc) - \mub(\epc)).
\end{align*}
\end{theorem}

\begin{lemma}\label{L:rlid}
Suppose that $\ep$ is extended-reachable. Then the compositions
$r_\ep l_\ep$, $l_\ep r_\ep$, $\rc_\ep\lc_\ep$, and
$\lc_\ep\rc_\ep$ are all identities.
\end{lemma}
\begin{proof}
By Theorem~\ref{T:rle}, an extended mutation sequence taking the
corresponding representation to a negative one conjugates each pair
of operators to translations by opposite vectors. These translations
are inverse. The same argument applies to the dual weights.
\end{proof}
\noindent We remark that if $\ep$ is not rigid, the compositions may not be identities. In fact, we gave a necessary and sufficient condition in \cite{Fg} for such a composition being the identity.
The following result is also proved there.
\begin{theorem}[{\cite[Corollary 5.25]{Fg}}] \label{T:rlrigid} Assume that $\ep$ is extended-reachable. Then \begin{enumerate}
\item general representations $M$ and $R$ of weight $\delta$ and $r_\ep(\delta)$ fit into the exact sequence
\begin{equation*}\cdots \to  \tauh^{-1} \mc{M}\xrightarrow{f_{-1}} \tauh^{-1} \mc{R} \xrightarrow{g_{-1}} \tauh^{-1} \mc{E}\xrightarrow{h_{-1}} M\xrightarrow{f_0} R \xrightarrow{g_0} E \xrightarrow{h_0} \tauh \mc{M} \xrightarrow{f_1} \tauh \mc{R} \xrightarrow{g_1} \tauh \mc{E} \xrightarrow{h_1} \tauh^2 \mc{M}\to \cdots.\end{equation*}
such that the ranks of $h_i$ and $g_i$ are all general ranks.
Moreover, we have that \begin{align*}
		r_{\ep}(\delta) = \delta + \epc - \rank(g_0) B(\Delta)\ &\text{ and }\ 
		\hom(r_{\ep}(\delta), \ep) = \hom(\delta, \ep) + \epc(\rank(g_0));\\
		r_{\ep}(\delta) = \delta + \ep + \rank(h_0) B(\Delta)\ &\text{ and }\  
		\e(r_{\ep}(\delta),\ep) = \e(\delta, \ep)  - \ep(\rank(h_0));\\
		r^{\ep}(\dtc) = \dtc - \tau^{-1}\ep - \rank(h_{-1}) B(\Delta)\ &\text{ and }\  
		\hom(\tau^{-1}\ep, r^{\ep}(\dtc)) = \hom(\tau^{-1}\ep, \dtc)  - \tau^{-1}\ep(\rank(h_{-1}));\\
		r^{\ep}(\dtc) = \dtc - \tau^{-1}\epc + \rank(g_{-1}) B(\Delta)\ &\text{ and }\  
		\ec(\tau^{-1}\ep, r^{\ep}(\dtc)) = \ec(\tau^{-1}\ep, \dtc)  + \tau^{-1}\epc(\rank(g_{-1})).	
		\end{align*}		
	\item general representations $M$ and $L$ of weight $\delta$ and $l_\ep(\delta)$ fit into the exact sequence
	\begin{equation*}\cdots \to  \tauh^{-1} \mc{L}\xrightarrow{f_{-1}} \tauh^{-1} \mc{M} \xrightarrow{g_{-1}} \tauh^{-1} \mc{E}\xrightarrow{h_{-1}} L\xrightarrow{f_0} M \xrightarrow{g_0} E \xrightarrow{h_0} \tauh \mc{L} \xrightarrow{f_1} \tauh \mc{M} \xrightarrow{g_1} \tauh \mc{E} \xrightarrow{h_1} \tauh^2 \mc{L}\to \cdots.\end{equation*}
such that the ranks of $g_i$ and $h_i$ are all general ranks.
		Moreover, we have that \begin{align*}
		l_{\ep}(\delta) = \delta - \epc + \rank(g_0) B(\Delta)\ &\text{ and }\ 
		\hom(l_{\ep}(\delta),\ep) = \hom(\delta,\ep) - \epc(\rank(g_0));\\
		l_{\ep}(\delta) = \delta - \ep - \rank(h_0) B(\Delta)\ &\text{ and }\  
		\e(l_{\ep}(\delta),\ep) = \e(\delta,\ep)  + \ep(\rank(h_0));\\
		l^{\ep}(\dtc) = \dtc + \tau^{-1}\ep + \rank(h_{-1}) B(\Delta)\ &\text{ and }\  
		\hom(\tau^{-1}\ep, l^{\ep}(\dtc)) = \hom(\tau^{-1}\ep, \dtc)  + \tau^{-1}\ep(\rank(h_{-1}));\\
		l^{\ep}(\dtc) = \dtc + \tau^{-1}\epc - \rank(g_{-1}) B(\Delta)\ &\text{ and }\  
		\ec(\tau^{-1}\ep, l^{\ep}(\dtc)) = \ec(\tau^{-1}\ep, \dtc)  - \tau^{-1}\epc(\rank(g_{-1})).
	\end{align*}
\end{enumerate}
\end{theorem}

\begin{remark} Both Theorems \ref{T:rle} and \ref{T:rlrigid} have analogues for $r^\ep,\ l^\ep$ and $\rc^\ep,\ \lc^\ep$. For example, if $\b{u}$ is an extended sequence of mutations such that $\mub(\epc)$ is positive, then the operators $l^\ep$ and $r^\ep$ on the $\dtc$-vectors of $(\Delta,\mc{S})$ are given by 
	$$ r^{\ep} (\dtc) = \mub^{-1} (\mub(\dtc) + \mub(\epc)),\qquad l^{\ep} (\dtc) = \mub^{-1} (\mub(\dtc) - \mub(\epc)). $$
One can find these statements in \cite{Fg}.
\end{remark}

\subsection{Minimally Exceptional Representations}
A dimension vector $\gamma$ is called a quotient (resp. sub-)dimension vector of $E$ if $\gamma$ is the dimension vector of some quotient (sub-) representation of $E$.
Here $\coker(\ep)$ and $\ker(\epc)$ refer to general projective and injective presentations, respectively.
\begin{definition} A $\delta$-vector $\ep$ is called {\em minimally exceptional}
	if $\coker(\ep)\ne0$ and $\ep(\gammac)=1$ for any nonzero quotient dimension vector $\gammac$ of $\coker(\ep)$.
	A $\dtc$-vector $\epc$ is called {\em minimally exceptional}
	if $\ker(\epc)\ne0$ and $\epc(\gamma)=1$ for any nonzero subdimension vector $\gamma$ of $\ker(\epc)$.
\end{definition}
\noindent We need to point out that minimal exceptionality of $\ep$ is {\em not} equivalent to the corresponding $\epc$ being minimally exceptional. Recall that $\ep$ is called {\em Schur} if $\hom(E,E)=1$ for some $E$ of weight $\ep$.

\begin{lemma} \label{L:minexc} A minimally exceptional $\delta$-vector $\ep$ is Schur and rigid; and satisfies $\hom(E,\ep)=1$ if $\ep$ is reachable, where $E=\coker(\ep)$ and $\ep$ in the second argument is viewed as a $\dtc$-vector. Conversely, if $\ep$ is rigid and satisfies $\hom(E, \ep)=1$, then $\ep$ is minimally exceptional.
\end{lemma}
\begin{proof} Let $e$ be a general presentation of weight $\ep$ with cokernel $E$. We first show that $\hom(E,E)=1$.
If $\hom(E,E)> 1$, then there is a non-invertible, nonzero homomorphism $f:E \to E$.
Consider the exact sequence $0\to \img(f) \to E \to C \to 0$.
We have that $1 = \ep(\dv(E)) = \ep(\dv(C)) + \ep(\dv \img(f))$.
Since both $\img(f)$ and $C$ are quotient representations of $E$, by our assumption one of them has to be trivial.
The contradiction shows that $\hom(E,E)=1$, so $E$ is indecomposable.
It follows that $\e(e,E)=\hom(E,E)-\ep(\dv(E))=0$, that is, $\ep$ is rigid.
If $\ep$ is reachable, then $\hom(E,\ep)=1$ by Theorem \ref{T:HomE}, where the second $\ep$ is viewed as a $\dtc$-vector.

Conversely, by \cite[Lemma 5.15]{Ft} we have $\ep(\gammac)>0$ for any nonzero quotient dimension vector $\gammac$ of $E$. By \cite[Lemma~3.5]{Ft}, $\ep(\gammac)\leq\fc_E(\ep)\leq\hom(E,\ep)=1$. Hence, $\ep$ is minimally exceptional.
\end{proof}
\noindent The lemma shows in particular $E$ is indecomposable and in fact {\em exceptional}, that is, both Schur and $\E$-rigid.
As any rigid $\ep$ satisfies $\ep(\gammac)>0$, this explains the name minimal. For the extended-reachable representations in the following corollary, Theorem~\ref{T:rlrigid} applies; Lemma~\ref{L:minexc} alone does not assert extended reachability.

\begin{corollary}\label{C:minexc}
Assume that the decorated representation $\mc E$ is extended-reachable,
and write $E$ for its ordinary part. If its projective weight $\ep$
is minimally exceptional, then
\begin{align*}
\e(r_\ep(\delta),E)&=\e(\delta,E)-1
&&\text{if }\e(\delta,E)>0,\\
\e(l_\ep(\delta),E)&=\e(\delta,E)+1
&&\text{if }\e(l_\ep(\delta),E)>0.
\end{align*}
If its injective weight $\epc$ is minimally exceptional, then
\begin{align*}
\ec(E,\rc_\ep(\dtc))&=\ec(E,\dtc)-1
&&\text{if }\ec(E,\dtc)>0,\\
\ec(E,\lc_\ep(\dtc))&=\ec(E,\dtc)+1
&&\text{if }\ec(E,\lc_\ep(\dtc))>0.
\end{align*}
The operators in this statement are the ambient operators of
Definition~\ref{D:rlgrank}, before restriction to a support set.
\end{corollary}
\begin{proof}
Let $h:E\to\tauh\mc M$ be the general map in
Theorem~\ref{T:rlrigid}, where $\mc M$ is general of weight $\delta$.
That theorem gives
$$ \e(r_\ep(\delta),E) =\e(\delta,E)-\ep(\dv\img h). $$
If $\e(\delta,E)>0$, then $h\ne0$ by Lemma~\ref{L:H2E}.
Its image is a nonzero quotient of $E$, so minimal exceptionality
makes the last term equal to one. Applying this identity to
$l_\ep(\delta)$ and using Lemma~\ref{L:rlid} proves the second
formula with the stated $\e$-positive condition. The two injective
formulas follow by passing to the opposite QP. In particular, the dual
raising operator decreases $\ec$, while its inverse increases it.
\end{proof}

\section{Boundary Representations} \label{S:boundary}
\subsection{Boundary Representations} \label{ss:boundary}
Let $(\Delta,\S)$ be an ice quiver with potential, and $\Delta^\mu$ be the mutable part of $\Delta$.
We write $(\Delta,\S)_\mu$ for the restriction of $(\Delta,\S)$ to $\Delta^\mu$.
We denote by $\Delta^\mu[i]$ the full subquiver of the mutable vertices together with (a frozen vertex) $i$.
We shall denote $(\Delta,\mc{S})$ restricting to $\Delta^{\mu}[i]$ and its Jacobian algebra by  $(\Delta,\S)_{\mu[i]}$ and $J_{\mu[i]}$ respectively.
Note that we have a chain of extensions $\Delta^\mu \subset \Delta^\mu[i] \subset \Delta$.
Let $P_{[i]}$ and $I_{[i]}$ be respectively the indecomposable projective and injective representation of $J_{\mu[i]}$ corresponding to $i$. This distinguishes them from the representations $P_i$ and $I_i$ of $(\Delta,\mc{S})$.
Let $\iota_i:\Delta^\mu[i]\hookrightarrow \Delta$ be the natural inclusion.
\begin{definition} \label{D:Brep} The {\em boundary representation} $E_i$ attached to a frozen vertex $i$ is $(\iota_i)_!(P_{[i]})$.
The {\em dual boundary representation} $E_i^\star$ attached to a frozen vertex $i$ is $(\iota_i)_!(I_{[i]})$.
\end{definition}
\noindent It follows from the definition that the dual boundary representation $E_i^\star$ is dual of the boundary representation of the opposite quiver with potential.
We also remark that in general $E_i$ is {\em not} the projective representation of $(\Delta,\S)$. 

Construction~\ref{cons:iQP} supplies the nondegenerate reference ice
QPs used below: they have no frozen arrows, and each frozen vertex is
simple in its one-frozen restriction.
We use the actual presentation $d_i$ extracted by \eqref{eq:dv}
inside $(\Delta,\S)_{\mu[i]}$. It is a presentation over $J_\mu$,
of weight $-b_i^{\mathsf T}$, where $b_i$ is the frozen column of
$B_\Delta$. If it is rigid, the return-path argument preceding
Lemma~\ref{L:vertex-normal-form} gives a simple representative of the
one-frozen restriction without changing its Jacobian algebra. An
existence construction for another potential is not needed here.

Let $\mc{E}_i^\mu$ be the decorated representation of $J_\mu$ corresponding to $d_i$.
Let $\iota:\Delta^\mu\hookrightarrow \Delta$ be the natural inclusion.
If $\mc{E}_i^\mu$ is rigid, then by Corollary \ref{C:iota-zero}.(2) there are exact sequences \begin{align}
	\label{eq:exactEi} 0 \to	\iota_!(E_i^\mu) \to &E_i \to S_i \to 0.
\shortintertext{Dually, we can work with $\mc E_i^{\star\mu}=\tau_\mu\mc E_i^\mu$, of injective weight $b_i^{\mathsf T}$:}
	 0 \to	S_i \to &E_i^\star \to  \iota_!(E_i^{\star\mu}) \to 0,
\notag
\end{align}
Here $\tau_\mu$ is the decorated AR-translation over $J_\mu$, and
$E_i^\mu,E_i^{\star\mu}$ denote the ordinary parts.

We will use the following easy observation (see also \cite[Lemma 5.2]{Fs2}) more than once.
Recall the map $\alpha_u,\beta_u$ and $\gamma_u$ in \eqref{eq:abc}.
\begin{observation}\label{O:inv} For a representation $M$ supported on a full subquiver $\hat{\Delta}$ of $\Delta$,
	the spaces $\coker(\alpha_u)$ and $\ker(\alpha_u)/ \img(\gamma_u)$ are invariant under the extension-by-zero functor if $u\in \hat{\Delta}_0$.
\end{observation}
	
\begin{lemma}\label{L:HEres} Suppose that two representations $M$ and $N$ of $(\Delta,\mc{S})$ are supported on a full subquiver $\hat{\Delta}$ of $\Delta$.
	Let $\hat{J}$ be the Jacobian algebra of $(\Delta, \mc{S})$ restricted to $\hat{\Delta}$.
	Then $\Hom_{\hat{J}}(M,N) \cong \Hom_{J}(M,N)$, $\E_{\hat{J}}(M,N)\cong \E_{J}(M,N)$, and $\Ec_{\hat{J}}(M,N)\cong \Ec_{J}(M,N)$.
In particular, (dual) boundary representations are $\E$-rigid representations of $(\Delta,\mc{S})$.
\end{lemma}
\begin{proof} The invariance of $\Hom(M,N)$ is obvious. Recall the identification of $\Hom(M, S_u)$ and $\Ext^1(M, S_u)$ from the sequence \eqref{eq:Su_inj}. 
	By Observation \ref{O:inv}, $\coker(\alpha_u)$ and $\ker(\alpha_u)/ \img(\gamma_u)$ are invariant under the extension-by-zero functor if $u\in\hat{\Delta}_0$.
Thus the multiplicities in a minimal projective presentation of $M$ are invariant at every $u\in\hat\Delta_0$. 
If $u\notin\hat\Delta_0$, then $M(u)=0$, so
$\beta_+^J(M)(u)=\dim\coker(\alpha_u)=0$; hence any additional
projective summands can occur only in the source $P_-$ of the
minimal projective presentation. Since $N(u)=0$, they contribute
nothing to $\Hom(-,N)$.
Hence, $\E_{\hat{J}}(M,N)\cong \E_{J}(M,N)$. The dual argument proves $\Ec_{\hat{J}}(M,N)\cong \Ec_{J}(M,N)$.
\end{proof}

\begin{lemma}\label{L:E-esym}
Suppose that $\mc E_i^\mu$ is rigid. For every decorated representation $\mc M$ of $J_\mu$, we have that
$$ \e_J(\iota_!(\mc M),E_i) = \e_{J_\mu}(\mc M,\mc{E}_i^\mu) + \e_{J_\mu}(\mc{E}_i^\mu,\mc M). $$
\end{lemma}

\begin{proof}
	Since $\mc M$ is supported on the mutable part, $\Hom_J(\iota_!(\mc M),S_i)=0$.  Hence
	$\hom_J(\iota_!(\mc M),E_i)=\hom_J(\iota_!(\mc M),\iota_!(\mc{E}_i^\mu))$.  
	From the long exact sequence
	$$0=\Hom(\iota_!(\mc{M}),S_i) \to \E(\iota_!(\mc{M}), \iota_!(\mc{E}_i^\mu)) \to \E(\iota_!(\mc{M}), E_i) \to \E(\iota_!(\mc{M}), S_i) \to 0$$
	and Lemma \ref{L:HEres}, we get
	$$ \e_J(\iota_!(\mc M),E_i)=\e_{J_\mu}(\mc M,\mc{E}_i^\mu)-\delta_{\iota_!(\mc M)}(i). $$
	Lemma \ref{L:iota-zero}.(2), applied to the extension by
	$\mc{E}_i^\mu$, gives $\delta_{\iota_!(\mc M)}(i)=-\e_{J_\mu}(\mc{E}_i^\mu,\mc M)$,
	proving the claim.
\end{proof}


\begin{lemma} \label{L:bdinv}	The mutation $\mu_u(E_i)$ is the boundary representation of $i$ for $\mu_u(\Delta,\mc{S})$.
\end{lemma}
\begin{proof} Let $\iota_i:\Delta^\mu[i]\hookrightarrow \Delta$ and $\iota_i':\mu_u(\Delta^\mu[i])\hookrightarrow \mu_u(\Delta)$ be the natural inclusions. We have that
$\mu_u(E_i) = \mu_u((\iota_i)_!(P_{[i]})) = (\iota_i')_!(\mu_u(P_{[i]}))$.
By Lemma \ref{L:gdmu} and \cite[Theorem 1.10]{GLFS}, $\mu_u(P_{[i]}) = P_{[i]}'$ which is the indecomposable projective representation for $\mu_u\left((\Delta,\S)_{\mu[i]}\right) = (\mu_u(\Delta,\S))_{\mu[i]}$.
Hence $\mu_u(E_i)$ is the boundary representation of $i$ for $\mu_u(\Delta,\mc{S})$. 
\end{proof}


\begin{corollary} If $d_i$ is negative reachable from a sequence of mutation $\b{\mu}_{\b{u}}$, then $\mub(E_i)=S_i$
is the simple representation of $\mub(\Delta,\S)$. 
\end{corollary}

\begin{definition}\label{D:boundary-reachability} 
A frozen vertex $i$ is \emph{rigid} (resp. \emph{reachable}) if $d_i$ is rigid (resp. reachable).
It is \emph{dually reachable} if $\dc_i$ is reachable (as an injective presentation).
Equivalently, a frozen vertex $i$ is reachable (resp. dually reachable) if some mutation sequence sends $E_i$ (resp. $E_i^\star$) to $S_i$. 
\end{definition}
\noindent An auxiliary mutation at $i$ sends $S_i$ to a negative simple. Equivalently, $i$ is dually reachable if some mutation sequence makes $E_i$ isomorphic to $P_i$ in the resulting seed.

By the {\em frozen dimension} of a representation $M$, we mean the total dimension of its restriction on the frozen part of $\Delta$. From now on, we set $\ep_i$ and $\epc_i$ to be $\delta$-vector and $\dtc$-vector of $E_i$ respectively.
\begin{proposition} \label{P:Brep}
For a rigid frozen vertex $i$, the boundary representation $E_i$ has the following properties: \begin{enumerate}
\item $E_i$ has frozen dimension $1$ and all its proper subrepresentations are supported on $\Delta_0^\mu$.
\item The $\delta$-vector of $E_i$ is only supported on the frozen part and the only positive coordinate is $1$ at $i$.
\item $\ep_i$ is minimally exceptional, in particular, indecomposable.
\end{enumerate}
There are also dual statements for the dual boundary representation $E_i^\star$.
\end{proposition}
\begin{proof} 	(1). $E_i$ has frozen dimension $1$ by \eqref{eq:exactEi}.
Recall that $E_i=(\iota_i)_!(P_{[i]})$.
Since $S_i$ is the top of $P_{[i]}$, the radical of $E_i$ is the unique maximal proper submodule of $E_i$. So $E_i$ satisfies (1).
		
(2). We first note that the $\delta$-vector of $E_i$ is only supported on frozen vertices by Observation \ref{O:inv}.
If $u$ is frozen, we see from the exact sequence \eqref{eq:exactEi} that $\Hom(E_i, S_u)=0$ except when $u=i$, and
$\Hom(E_i, S_i)=k$.
So the only positive coordinate is $1$ at $i$.

(3). Let $\ep_i$ be the $\delta$-vector of $E_i$.
The only possible positive coordinates of $\ep_i$ are frozen. In fact, there is only one such frozen vertex $i$ with $\ep_i(i)=1$ due to the fact that the frozen dimension of $E_i$ is $1$. Then the property (1) implies that $\ep_i$ is minimally exceptional.
\end{proof}
\begin{remark}\label{r:(2)} Without the rigid assumption, (1) and (3) may fail, but (2) still holds. 
In general, $\epc$ is {\em not} minimally exceptional. 
\end{remark}

\begin{corollary} \label{C:HomEij} \label{C:Brep} Let $i$ and $j$ be rigid frozen vertices. Then
	$$ \hom(E_i, E_j)=\updelta_{i,j}\ \text{ and }\ \e(E_i, E_j) = \max(0, -\ep_i(j)). $$
Dually we have that 
	$$ \hom(E_i^\star, E_j^\star)=\updelta_{i,j}\ \text{ and }\ \ec(E_i^\star, E_j^\star) = \max(0, -\epc_j^\star(i)). $$
\end{corollary}
\begin{proof} This follows from Remark \ref{r:HomE} and Proposition \ref{P:Brep}.(1,\ 2).
\end{proof}
\noindent Lemma~\ref{L:E^mu} will identify both multiplicities through
the frozen corner, proving $\ep_i(j)=\epc_j^\star(i)$.


\subsection{The $\mu$-supported $\delta$-vectors} \label{ss:musupp}
\begin{definition}[{\cite{Fs1}}] A $\delta$-vector of $(\Delta,\S)$ is called {\em $\mu$-supported} if $\dv(\delta)$ is only supported on the mutable part $\Delta_0^\mu$. We denote by $\trop(\Delta,\S)$ the set of all $\mu$-supported $\delta$-vectors of $(\Delta,\S)$.
Similarly we define $\mu$-supported $\dtc$-vectors, and denote by $\check{\trop}(\Delta,\S)$ the set of all $\mu$-supported $\dtc$-vectors of $(\Delta,\S)$.
\end{definition}

\begin{lemma}\label{L:musupp}
Let $\mub$ be a sequence of mutations, and put
$(\Delta',\S')=\mub(\Delta,\S)$. Tropical mutation gives bijections
$$
 \mub:\trop(\Delta,\S)\xrightarrow{\sim}\trop(\Delta',\S'),
 \qquad
 \mub:\check\trop(\Delta,\S)\xrightarrow{\sim}
 \check\trop(\Delta',\S'),
$$
using projective and injective weight mutation, respectively.
\end{lemma}
\begin{proof}
It suffices to consider one mutation $\mu_u$. Choose a
general decorated representation $\mc M$ of a supported weight
$\delta$. Mutation changes its ordinary vector space only at
$u$; its spaces at all frozen vertices remain zero. By
Theorem~\ref{T:genpi}(4), mutation takes general representations to
general representations of the mutated weight
\cite[Theorem~1.10]{GLFS}. Thus
$\mu_u(\delta)\in\trop(\mu_u(\Delta,\S))$.
Involutivity gives the inverse, and iteration proves the first
assertion. The injective-weight assertion is dual.
\end{proof}

\begin{definition}\label{D:frozen-finite}
An ice quiver with potential $(\Delta,\S)$ is called
\emph{frozen-Jacobi-finite} if $P_i(j)$ is finite-dimensional for every
pair $i,j\in\Delta_0^{\op{fr}}$. Equivalently, $eJe$ is
finite-dimensional, where $e=\sum_{i\in\Delta_0^{\op{fr}}}e_i$.
\end{definition}

\begin{lemma}\label{L:ice-finite}
Suppose that $(\Delta,\S)_\mu$ is Jacobi-finite. Then
$(\Delta,\S)$ is frozen-Jacobi-finite if and only if its full Jacobian
algebra $J$ is finite-dimensional.
\end{lemma}
\begin{proof}
Only the forward implication needs proof. Let $\mathfrak m$ be the
image of the arrow ideal in $J$, and let
$\mathfrak a=\overline{JeJ}$. All ideal products and powers below are
closed in the arrow-adic topology. Since
$J/\mathfrak a\cong J_\mu$ is finite-dimensional, there is an integer
$a\geq1$ such that $\mathfrak m^a\subseteq\mathfrak a$.
The descending sequence $e\mathfrak m^b e$ of subspaces of the
finite-dimensional space $eJe$ has zero intersection, because $J$
is separated. Hence $e\mathfrak m^b e=0$ for some $b\geq1$. It follows
that
$e\mathfrak m^{a+b}
\subseteq e\mathfrak m^b\mathfrak a
\subseteq\overline{e\mathfrak m^b eJ}=0$.
Consequently,
$\mathfrak m^{2a+b}
\subseteq\mathfrak a\mathfrak m^{a+b}
\subseteq\overline{Je\mathfrak m^{a+b}}=0$.
There are only finitely many paths of bounded length in the finite
quiver $\Delta$, so nilpotence of the arrow ideal makes $J$
finite-dimensional.
\end{proof}

\begin{lemma}\label{L:filter} Suppose that $(\Delta,\S)$ is frozen-Jacobi-finite. For each frozen vertex $i$, the injective representation $I_i$ admits a finite filtration whose successive subquotients embed into finite direct sums of the dual boundary representations $E_v^\star$.
\end{lemma}
\begin{proof} By the dual of Proposition \ref{P:Brep}.(2) and Remark~\ref{r:(2)}, $E_i^\star$ has a minimal injective presentation 
	$$ 0\to E_i^\star \to I_i \xrightarrow{f_{1}} \bigoplus_{j} m_{j}^1 I_j. $$
It remains to filter the image of $f_1$ by subquotients embedding into finite direct sums of the $E_v^\star$.
If the image of $f_{1}$ lies in $\bigoplus_{j} m_{j}^1 E_j^\star$, then we are done.
Otherwise consider the minimal injective presentation of $\bigoplus_{j} m_{j}^1 E_j^\star$
	$$ 0\to \bigoplus_{j} m_{j}^1 E_j^\star \to \bigoplus_{j} m_{j}^1 I_j \xrightarrow{f_{2}} \bigoplus_{k} m_{k}^2 I_k. $$
If the image of the composition $f_{2}f_{1}$ lies in $\bigoplus_{k} m_{k}^2 E_k^\star$, then we are done (because $I_i$ is filtered by $E_i^\star$, $\img f_1 \cap \bigoplus_{j} m_{j}^1 E_j^\star$, and a subrepresentation of $\bigoplus_{k} m_{k}^2 E_k^\star$).
Otherwise, continue this way, and this procedure must end in finite number of steps.
If not, then we get an infinite sequence of maps $I_i \xrightarrow{f_{1}} \bigoplus_{j}m_{j}^1 I_j \xrightarrow{f_{2}} \bigoplus_{k}m_{k}^2 I_k \to \cdots$, whose composition is nonzero. This is impossible: minimality puts every matrix entry of each $f_s$ in the radical of $\End(\bigoplus_{j\in\Delta_0^{\op{fr}}}I_j)$, which is nilpotent by frozen-Jacobi-finiteness.
\end{proof}

The following theorem was proved in \cite{Fs1} (see also \cite{Ft}) for ``non-intrinsically'' defined boundary representations.
\begin{theorem}\label{T:musupp} Suppose that $(\Delta,\S)$ is frozen-Jacobi-finite.
A $\delta$-vector is $\mu$-supported if and only if $\Hom(\delta, E_i^\star)=0$ for every frozen vertex $i$;
dually a $\dtc$-vector is $\mu$-supported if and only if $\Hom(E_i, \dtc)=0$ for every frozen vertex $i$.
\end{theorem}
\begin{proof} We note that a $\delta$-vector is $\mu$-supported if and only if $\Hom(\delta, I_i)=0$ for each frozen vertex $i$.	
By the dual of Proposition \ref{P:Brep}.(2) (see Remark \ref{r:(2)}), $E_i^\star$ is a subrepresentation of $I_i$. So if $\delta$ is $\mu$-supported then $\Hom(\delta, E_i^\star)=0$.
Conversely, suppose that $\Hom(\delta, E_i^\star)=0$. Choose one general presentation of weight $\delta$ realizing these vanishings for all frozen vertices simultaneously.
By Lemma~\ref{L:filter}, every successive subquotient in a filtration of $I_i$ embeds into a finite direct sum of the $E_j^\star$. 
Therefore, we conclude that $\Hom(\delta, I_i)=0$. The dual statement can be proved similarly.
\end{proof}

\begin{remark}
If $i$ is dually reachable, choose a mutation sequence
sending $E_i^\star$ to $S_i$, followed by the auxiliary mutation
at $i$ of Definition~\ref{D:boundary-reachability}.
Consequently, $E_i^\star$ is negative reachable in the representation-theoretic sense used in
Theorem~\ref{T:HomE}, which identifies
$\hom(\delta,E_i^\star)$ with $f_{E_i^\star}(\delta)$. Its
vanishing is therefore a finite system of linear inequalities.
Conjecture~\ref{c:dualityhom} would remove this reachability assumption.
\end{remark}

\subsection{The Boundary algebra and Boundary Weights}\label{ss:minimal}
Let $e=\sum_{i\in\Delta_0^{\op{fr}}}e_i$, and
$\Gamma=eJe$ be the {\em boundary algebra}. For the simple right $\Gamma$-module $T_i$, we put
$$
 a_{ij}=\ext^1_\Gamma(T_i,T_j),\qquad
 a_i=(a_{ij})_{j\in\Delta_0^{\op{fr}}},\qquad a_i'=(a_{ji})_{j\in\Delta_0^{\op{fr}}}.
$$
These are the outgoing and incoming arrow multiplicities in the
Gabriel quiver of $\Gamma$.
\begin{lemma}\label{L:E^mu}
For a rigid frozen vertices $i$ and $j$, we have that
$$ \ep_i=(0,e_i-a_i),\qquad \epc_i^\star=(0,e_i-a_i'),\qquad a_{ij}=\e_J(E_i,E_j)=\ec_J(E_i^\star,E_j^\star). $$
In particular, $a_{ii}=0$ and
$\ep_i(j)=\epc_j^\star(i)$. The algebra $\Gamma$, with its vertex
idempotents, and the numbers $a_{ij}$ are invariant under mutation.
\end{lemma}
\begin{proof}
By Proposition~\ref{P:Brep}, $E_ie=T_i$, and both terms in the
minimal projective presentation of $E_i$ are indexed by frozen
vertices. The exact functor $(-)e$ sends this presentation to the
minimal projective presentation of $T_i$ over $\Gamma$.
To show that the induced projective presentation remains minimal at
its source, let $K=\ker(P_i\to E_i)$. Since $K$ is generated
at frozen vertices, $K=(Ke)J$, and
$(K\operatorname{rad}J)e=(Ke)\operatorname{rad}\Gamma$.
Hence the frozen components of the top of $K$, namely
$Ke/(K\operatorname{rad}J)e$, agree with the top of the
$\Gamma$-module $Ke$. Their projective-cover multiplicities
are therefore equal, proving the formula for $\ep_i$. The
injective argument is dual.
Applying $\Hom(-,E_j)$ to the presentation gives the same cokernel as
applying $\Hom_\Gamma(-,T_j)$ to the restricted presentation.
This proves the first identity for $a_{ij}$, and duality proves the second.
Finally, the boundary algebra is preserved by mutation by \cite[Corollary~6.6]{DWZ1}.
\end{proof}

For every frozen $i$, let $\mc E_i^\mu$ be the extracted
mutable decorated representation and put
$$ q_{ij}(t)=\e_{J_{\mu,t}}(\mc E_i^\mu(t),\mc E_j^\mu(t)), \qquad b_{ij}(t)=B(\Delta_t)(i,j). $$

The next lemma computes the invariant boundary pairings at every seed.
\begin{lemma}\label{L:boundary-propagation}
If $i$ is rigid, then for every frozen $j$ and every seed $t$,
\begin{equation}\label{eq:boundary-propagation} a_{ij}=q_{ji}(t)+b_{ij}(t).\end{equation}
The same identity holds if $j$ is rigid. In particular, for rigid
$i,j$,
$$ a_{ij}+a_{ji}=q_{ij}(t)+q_{ji}(t). $$
\end{lemma}
\begin{proof}
	First consider the reference seed $t_0$, where there
	are no arrows between frozen vertices. Let $i\ne j$
	be frozen vertices, and suppose that $i$ is rigid.
	The restriction of $E_i$ to the mutable subquiver
	is $E_i^\mu$, and $E_i$ vanishes at every frozen vertex
	other than $i$. Moreover, every arrow entering $i$
	acts trivially on $E_i$, since $E_i$ is obtained
	from $P_{[i]}$ and $e_iJ_{\mu[i]}e_i=ke_i$.
	Consequently, any path through a frozen vertex
	contributing to $\gamma_j(E_i)$ acts trivially.
	Thus $\gamma_j(E_i)$ agrees with the map induced
	by $d_j$ on $E_i^\mu$. By \eqref{eq:Betti} and
	Lemma~\ref{L:E^mu}, we obtain
	$$ a_{ij} =\ext^1_J(E_i,S_j) =\e_{J_\mu}(\mc E_j^\mu,\mc E_i^\mu)=q_{ji}(t_0). $$
For $i=j$, both sides vanish by rigidity. The column assertion at
$t_0$ follows by the dual argument.

Under mutation at $u$, the terms of the extracted presentation at
$i$ have multiplicities $[b_{iu}]_+$ and $[-b_{iu}]_+$.
Lemma~\ref{L:HEmu} and matrix mutation give
$$ q_{ij}'-q_{ij} =[b_{iu}]_+[b_{uj}]_+-[-b_{iu}]_+[-b_{uj}]_+ =b_{ij}'-b_{ij}. $$
Hence $q_{ji}+b_{ij}$ is invariant. Since $b_{ij}(t_0)=0$, this
proves \eqref{eq:boundary-propagation}; skew-symmetry gives its
symmetrized form.
\end{proof}

By Lemma~\ref{L:boundary-propagation}, if $E_i=S_i$, then
$\mc E_i^\mu$ is negative and $b_{ij}=a_{ij}\ge0$ for every frozen
$j$. Thus no arrow from a frozen vertex enters $i$. Dually,
if $E_i^\star=S_i$, no arrow from $i$ goes to a frozen vertex.
The first assertion allows outgoing frozen arrows, so $S_i$ need
not be projective over the full Jacobian algebra; the dual assertion
allows incoming frozen arrows.
Write $\ep_i^\mu,\epc_i^\mu$ for the two weights of $\mc E_i^\mu$,
and put
$$ h_i(j)=\hom(\mc E_j^\mu,\mc E_i^\mu),\qquad h_i^\tau(j)=\hom(\tau_\mu\mc E_i^\mu,\tau_\mu\mc E_j^\mu). $$

In the frozen-coordinate formulas below, $e_i$ denotes the
$i$-th standard basis vector of $\mathbb Z^{\Delta_0^{\op{fr}}}$.
\begin{lemma}\label{L:epci}
For a rigid frozen vertex $i$,
\begin{enumerate}
\item $(\dv E_i)B_\Delta^{\mathsf T}=-\epc_i^\mu$ and
$(\dv E_i^\star)B_\Delta^{\mathsf T}=\tau_\mu\ep_i^\mu$.
\item At every seed,
\begin{equation}\label{eq:boundary-weights}
 \epc_i=(\epc_i^\mu,e_i-h_i),\qquad
 \ep_i^\star=(\tau_\mu\ep_i^\mu,e_i-h_i^\tau).
\end{equation}
\end{enumerate}
\end{lemma}
\begin{proof}
Write $B_\Delta=(B_\mu,B_{\op{fr}})$. Since
$\dv E_i=(\dv E_i^\mu,e_i)$ and the $i$-th frozen column is
$-{\ep_i^\mu}^{\mathsf T}$,
$$ (\dv E_i)B_\Delta^{\mathsf T} =-\dv E_i^\mu B_\mu-\ep_i^\mu=-\epc_i^\mu. $$
The dual identity follows in the same way. 
Take the coordinate at the frozen vertex $j$ in
\eqref{eq:delta2dual}:
\begin{align*}
 \epc_i(j)
 &=\updelta_{i,j}-a_{ij}-\ep_j^\mu(\dv E_i^\mu)+b_{ij}\\
 &=\updelta_{i,j}-a_{ij}-h_i(j)+q_{ji}+b_{ij}
 =\updelta_{i,j}-h_i(j),
\end{align*}
where the last equality is \eqref{eq:boundary-propagation}. Dually,
$$ \ep_i^\star(j) =\updelta_{i,j}-a_{ji}+q_{ij}-h_i^\tau(j)-b_{ij} =\updelta_{i,j}-h_i^\tau(j). $$
The mutable-coordinate formulas follow from part~(1).
\end{proof}


\subsection{Cartan Type and Weight Functions} \label{ss:Cartanwt}
Let $I$ be a set of rigid frozen vertices of $\Delta_0$.
\begin{definition}\label{D:CartanI}
The \emph{Cartan type} of $I$ is the symmetric generalized Cartan matrix
\begin{equation}\label{eq:wtC} c_{ij}=2\updelta_{i,j}-\e_J(E_i,E_j)-\e_J(E_j,E_i) =2\updelta_{i,j}-a_{ij}-a_{ji}.\end{equation}
\end{definition}
\noindent Lemma~\ref{L:E^mu} shows that the Cartan matrix
obtained from the dual boundary representations is also $C_I$. The propagation lemma also gives, at every seed,
$$ c_{ij}=2\updelta_{i,j}-q_{ij}-q_{ji}. $$
Thus the formula computes the intrinsic Cartan matrix.

\begin{definition}\label{D:adI}	A linear map $\wt=(\wt_i)_{i\in I}: \mb{Z}^{\Delta_0} \to \mb{Q}^I$ is called {\em adapted to} $I$ if it satisfies
	\begin{align*}  \wt_i(\epc_j) &= c_{ij}. 
\intertext{It is called dually adapted to $I$ if it satisfies}
  \wt_i(\ep_j^\star) &= c_{ij}. \end{align*}
\end{definition}

\begin{definition}[\cite{Fs1}] Let ${\sf L}$ be a subgroup of $\mb{Q}^I$. 
An ${\sf L}$-grading of $\Delta_0$, that is, a $\mb{Z}$-linear map $\mb{Z}^{\Delta_0} \to {\sf L}$, is called {\em compatible} to $\Delta$ if it annihilates the row space of $B_\Delta$.
\end{definition}
\noindent If ${\sf L}=\mb{Q}^I$, then we will drop $\mb{Q}^I$ and simply call it a compatible grading.
If ${\sf L}=\mb{Z}^I$, then we will call it an integral compatible grading.
For a compatible grading $\wt$, one can define its mutation at $u$
$$\mu_u(\wt)(v) = \begin{cases}\sum_{u\to w}\wt(w) - \wt(u) & \text{if $v=u$} \\ \wt(v) & \text{if $v\neq u$.} \end{cases}$$

\begin{lemma}\label{L:muwt}
Mutation preserves compatibility and adaptation, and
$\wt(\delta)=\mub(\wt)(\mub(\delta))$. The same holds for dual
weights and dual adaptation.
\end{lemma}
\begin{proof}
Compatibility and the equality of weights are checked in \cite{Fs1}.
The boundary representations mutate by Lemma~\ref{L:bdinv}, and the
Cartan matrix is invariant by Lemma~\ref{L:E^mu}. Applying the weight
equality to $\epc_j$, or to $\ep_j^\star$, proves adaptation.
\end{proof}

\begin{remark}
Put $R_\Delta=\operatorname{row}_{\mathbb Q}(B_\Delta)$.
A compatible grading adapted to $I$ exists over $\mathbb Q$ precisely
when every relation
$\sum_{j\in I}a_j\check\epsilon_j\in R_\Delta$ implies
$\sum_{j\in I}a_jc_{i,j}=0$ for every $i\in I$.
In particular, existence follows if the classes of
$\check\epsilon_j$, $j\in I$, are linearly independent modulo
$R_\Delta$. This stronger condition is exactly the condition needed
for a compatible linear weight map taking these classes to the simple
roots of a full realization; see Lemma~\ref{L:wtc}.
Rational existence alone does not imply the existence of an integral
adapted grading.
\end{remark}

Now we shall give a natural construction of weight functions.
Let $E=E_i$ be a boundary representation.
We set $\wt_\ep: \mb{Z}^{\Delta_0}\to\mb{Z}$ to be the linear functional given by 
\begin{equation}\label{eq:wti} \delta \mapsto \delta(\dv E - \dv(\tau^{-1} E)).\end{equation}
\begin{proposition}\label{P:weight}
For any $\dtc$-vectors with $\hom(\delta,\tau E)=\hom(E,\dtc)=0$, we have that
\begin{equation}\label{eq:wtfun} \ec(\tau^{-1}E,\dtc)-\e(\delta,E)=\wt_\ep(\dtc).\end{equation}
In particular, the equality \eqref{eq:wtfun} holds for $\mu$-supported $\dtc$.
\end{proposition}
\begin{proof} Due to the equalities \eqref{eq:hecform} and $\hom(E, \dtc)=0$, the equation is equivalent to
	$$ \e(\delta, E) = \hom(\tau^{-1} E, \dtc) + \ec(E,\dtc), $$	
	which is equivalent to
	$$ \hom(\delta,\tau E) + \e(\delta,E)=\hom(\tau^{-1} E, \dtc) + \ec(E,\dtc). $$
	This always holds due to Lemma \ref{L:H2E}.	
	For $\mu$-supported $\dtc$, we have $\hom(E, \dtc)=0$ by Theorem \ref{T:musupp}. 
	Moreover, by Proposition \ref{P:Brep}.(2) we have $\e(E,\coker(\delta))=0$ as well.
\end{proof}


\begin{lemma}\label{L:Cartan} For the fixed set $I$, define a weight function $\wt: \mb{Z}^{\Delta_0}\to \mb{Z}^I$ by 
	$$ \wt(\dtc) = (\wt_i(\dtc))_{i\in I}. $$ 
	Then we have that
	$$\wt_i(\epc_j) = \updelta_{i,j} + \hom(\tau^{-1}E_j, \tau^{-1}E_i)  - \left( \max(0, -\ep_i(j)) + \max(0, -\ep_j(i)) \right).$$		
\end{lemma}
\begin{proof} We do some straightforward calculation:
\begin{align*} &\epc_j(\dv E_i - \dv \tau^{-1} E_i) \\
=& (\hom(E_i, E_j) - \ec(E_i, E_j)) - (\hom(\tau^{-1}E_i, E_j) - \ec(\tau^{-1}E_i, E_j)) &&  \eqref{eq:hecform} \\
=& (\hom(E_i, E_j) - \ec(E_i, E_j))  - (\ec(E_j, E_i) - \hom(\tau^{-1}E_j, \tau^{-1}E_i)) && (\text{Lemma } \ref{L:H2E}.(1))\\
=& \hom(E_i, E_j) + \hom(\tau^{-1}E_j, \tau^{-1}E_i) - (\e(E_i, E_j)  + \e(E_j, E_i)) && (\text{Lemma } \ref{L:H2E}.(2)).
\end{align*}
	Then the desired equality follows from Corollary \ref{C:HomEij}.
\end{proof}

\begin{definition}
A pair of frozen vertices $(i,\ibar)$ is called {\em $\tau$-exact} if $\tau^{-1} E_i = E_{\ibar}^\star$.
In this definition we allow $i = \ibar$.
\end{definition}
\noindent The representation $E_i^\mu$ and the boundary representation $E_i$ do not depend on other frozen vertices. However, $\tau^{-1} E_i$ will depend. So being $\tau$-exact for $(i,\ibar)$ depends on other frozen vertices as seen in the following lemma.

\begin{lemma}\label{L:tauexact}
For rigid frozen vertices $i$, the pair $(i,\ibar)$ is
$\tau$-exact if and only if
$\tau_\mu^{-1}\mc E_i^\mu\cong\tau_\mu\mc E_{\ibar}^\mu$ and
\begin{equation}\label{eq:tauexact}
 \hom(\mc E_j^\mu,\mc E_i^\mu)
 +\hom(\tau_\mu^{-1}\mc E_i^\mu,\tau_\mu\mc E_j^\mu)
 =\updelta_{i,j}+\updelta_{\ibar,j}\quad(j\in\Delta_0^{\op{fr}}).
\end{equation}
In particular, $\tau$-exactness is mutation-invariant. Moreover, $\hom(E_i^\mu,E_i^\mu)=1$
unless one of $\tau_\mu^{-1}\mc E_i^\mu,\mc E_i^\mu,\tau_\mu\mc E_i^\mu$
is negative.
\end{lemma}
\begin{proof}
The projective weights of $\tau^{-1}E_i$ and $E_{\ibar}^\star$
are, by Lemma~\ref{L:epci},
$-(\epc_i^\mu,e_i-h_i)$ and $(\tau_\mu\ep_{\ibar}^\mu,e_{\ibar}-h_{\ibar}^\tau)$.
The isomorphism $\tau_\mu^{-1}\mc E_i^\mu\cong
\tau_\mu\mc E_{\ibar}^\mu$ identifies their mutable coordinates,
while equality of the frozen coordinates is precisely
\eqref{eq:tauexact}. The representations being compared are rigid, so
equality of their weights is equivalent to isomorphism. Mutation
invariance follows from Lemmas~\ref{L:bdinv} and~\ref{L:taucommu}.
For the last assertion, put $j=i$. If $\ibar\ne i$, the sum is
one and its first term is at least one when $E_i^\mu\ne0$. If
$\ibar=i$, the two Hom terms are
the endomorphism dimensions of $E_i^\mu$ and $\tau_\mu E_i^\mu$; when both
ordinary parts are nonzero their sum is two.
\end{proof}

\begin{remark}\label{r:Cartan-positivity}
Assume that $(\Delta,\S)$ is frozen-Jacobi-finite and that
$\{(i,\ibar)\}_{i\in I}$ are $\tau$-exact pairs with $i$ reachable
and $\ibar$ dually reachable, as in Theorem~\ref{T:crystal}.
Then Corollary~\ref{C:Cartan-positive} proves that $C_I$ is positive semidefinite.
\end{remark}

\begin{corollary}\label{C:tauexact}
If each $i\in I$ belongs to a $\tau$-exact pair of rigid frozen
vertices, then $(\wt_i)_{i\in I}$ from \eqref{eq:wti} is an integral
compatible grading adapted to $I$.
\end{corollary}
\begin{proof}
Put $F_i=\tau^{-1}E_i=E_{\ibar}^\star$. The projective--injective
weight identities give
$$ (\dv E_i-\dv F_i)B(\Delta)=-\ep_i-\check\delta_{F_i}. $$
Both terms on the right vanish at mutable vertices. Thus the
mutable coordinates give the compatibility condition. The partner map $i\mapsto\ibar$
is injective, since the boundary representations are pairwise
nonisomorphic. Hence
$\hom(F_j,F_i)=\updelta_{i,j}$, and Lemma~\ref{L:Cartan} proves
$\wt_i(\epc_j)=c_{ij}$.
\end{proof}

\begin{definition}[Coroot-coordinate weights]\label{D:coroot-weights}
Let $Q^\vee=\bigoplus_{i\in I}\mathbb Z h_i$ and
$$
P_{\mathrm{cor}}=\Hom_{\mathbb Z}(Q^\vee,\mathbb Z)
=\bigoplus_{i\in I}\mathbb Z\bar\varpi_i,
\qquad
\bar\alpha_j=\sum_{i\in I}c_{i,j}\bar\varpi_i.
$$
For an adapted compatible grading, the \emph{coroot-coordinate weight
map} is $\wt=\sum_i\wt_i\bar\varpi_i$. It takes values in
$P_{\mathrm{cor}}\otimes\mathbb Q$, and in $P_{\mathrm{cor}}$
when the grading is integral. It satisfies
$\wt(\check\epsilon_j)=\bar\alpha_j$ and $\langle\wt(\delta),h_i\rangle=\wt_i(\delta)$.
\end{definition}
\begin{remark}\label{r:projected-versus-full}
The roots $\bar\alpha_i$ may be dependent. Over $\mathbb Q$, this
weight space is the quotient of a full weight space by the annihilator
of the simple coroots. In the standard affine realization it is the
classical lattice $P_{\mathrm{cl}}=P/\mathbb Z\delta_{\mathrm{null}}$.
If $C_I$ is nonsingular, no directions are lost, and adaptation already
implies \eqref{eq:existwt}: a relation among the boundary classes would
give a vector in $\ker C_I$. For singular $C_I$, no independence
hypothesis is needed for coroot-coordinate weights. The next lemma
concerns only their optional lift to a \emph{full} realization.
\end{remark}

Let $n=|I|$ and $r=\operatorname{rank}(C_I)$.
A \emph{rational realization} of $C_I$ consists of a
$\mathbb Q$-vector space $\mathfrak h_{\mathbb Q}$ of dimension
$2n-r$, linearly independent elements $h_i$, $i\in I$, and
linearly independent elements $\alpha_i\in\mathfrak h_{\mathbb Q}^*$
such that $\alpha_j(h_i)=c_{i,j}$.
Its complexification is a realization over $\mathbb C$.
Extend the simple coroots to a basis
$\{h_a\}_{a\in I\sqcup K}$, and let
$\{\varpi_a\}_{a\in I\sqcup K}$ be its dual basis. We choose the
extension so that the simple roots lie in
$\Lambda=\bigoplus_{a\in I\sqcup K}\mathbb Z\varpi_a$.
Let $\Lambda_{\mathbb Q}=\Lambda\otimes_{\mathbb Z}\mathbb Q$.
The next lemma constructs the auxiliary root coordinates and the
compatible full weight map simultaneously.

\begin{lemma}\label{L:wtc}
Let $(\operatorname{wt}_i)_{i\in I}$ be a rational compatible grading
adapted to $I$. Suppose that the classes of
$\check\epsilon_i$, $i\in I$, in
$V_\Delta=\mathbb Q^{\Delta_0}/R_\Delta$,
and $R_\Delta=\operatorname{row}_{\mathbb Q}(B_\Delta)$
are linearly independent. Equivalently, writing $\check E_I$ for the
matrix with rows $\check\epsilon_i$, assume
\begin{equation}\label{eq:existwt}
\operatorname{rank}_{\mathbb Q}
\begin{pmatrix} B_\Delta\\ \check E_I\end{pmatrix}
=\operatorname{rank}_{\mathbb Q}(B_\Delta)+n.
\end{equation}
Then there exist $n-r$ integral compatible linear forms
$\operatorname{wt}_k$, indexed by $K$, and a realization of $C_I$
with coroot basis $\{h_a\}_{a\in I\sqcup K}$, such that
$\operatorname{wt}(\delta)=\sum_{a\in I\sqcup K}\operatorname{wt}_a(\delta)\varpi_a$
satisfies
$$
\operatorname{wt}(R_\Delta)=0,
\qquad
\operatorname{wt}(\check\epsilon_i)=\alpha_i,
\qquad
\langle\operatorname{wt}(\delta),h_i\rangle
=\operatorname{wt}_i(\delta).
$$
If the original forms $\operatorname{wt}_i$, $i\in I$, are integral,
then $\operatorname{wt}(\mathbb Z^{\Delta_0})\subseteq\Lambda$.
Conversely, the independence condition is necessary for a compatible
linear weight map sending every $\check\epsilon_i$ to the corresponding
simple root of a full realization.
\end{lemma}

\begin{proof}
Write $q_i$ for the class of $\check\epsilon_i$ in $V_\Delta$,
and put $U=\operatorname{span}_{\mathbb Q}\{q_i:i\in I\}$.
The given forms descend to a map
$$ p:V_\Delta\longrightarrow\mathbb Q^I, \qquad p(q_j)=(c_{i,j})_{i\in I}. $$
Thus $\dim U=n$, $\operatorname{rank}(p|_U)=r$, and
$N=\ker(p|_U)$ has dimension $n-r$.
Choose $n-r$ rational linear forms on $V_\Delta$ whose restrictions
to $N$ form a basis of $N^*$. Pull them back to
$\mathbb Q^{\Delta_0}$ and multiply them individually by positive
integers to obtain integral compatible forms
$\operatorname{wt}_k$, $k\in K$. Their restrictions still separate
points of $N$. Hence
$q\mapsto\bigl(\operatorname{wt}_a(q)\bigr)_{a\in I\sqcup K}$
is injective on $U$.

Take a vector space with basis $h_a$, $a\in I\sqcup K$, with
$|K|=n-r$, and let $\varpi_a$ be the dual basis. Define
$\alpha_j=\sum_{a\in I\sqcup K}\operatorname{wt}_a(\check\epsilon_j)\varpi_a$.
These roots are linearly independent by the preceding injectivity, and
$\alpha_j(h_i)=c_{i,j}$ for $i,j\in I$. All their coordinates are
integral: the first $n$ coordinates are Cartan entries and the others
come from integral forms. This is the required realization, and the
claims about $\operatorname{wt}$ follow immediately.

Conversely, applying a compatible full weight map to
$\sum_i a_i\check\epsilon_i\in R_\Delta$ gives
$\sum_i a_i\alpha_i=0$. Independence of the simple roots forces all
$a_i=0$.
\end{proof}

\begin{remark}
For a fixed rational realization with prescribed
$\alpha_j(h_k)$, $k\in K$, equation~\eqref{eq:existwt} gives a
rational compatible completion. When the original forms and prescribed
coordinates are integral, a sufficient condition for an integral
completion is that the classes of $\check\epsilon_i$ generate a
primitive subgroup of
$\mathbb Z^{\Delta_0}/(\mathbb Z^{\Delta_0}\cap R_\Delta)$.
The realization constructed in the lemma chooses these auxiliary
coordinates together with the grading, satisfying
$\operatorname{wt}_k(\check\epsilon_j)=\alpha_j(h_k)$.
\end{remark}

\section{The Tropical Crystal Structures} \label{S:crystal}

\subsection{The Lowering and Raising Operators attached to Boundary Representations} \label{ss:rlop}
Let $E_i$ and $E_i^\star$ be the boundary and dual boundary
representations at a rigid frozen vertex $i$. For the ordinary
operators, assume that $i$ is reachable; for the dual operators,
assume that $i$ is dually reachable, in the sense of
Definition~\ref{D:boundary-reachability}.

\begin{definition} \label{D:lrv} We define the raising and lowering operators $r_i$ and $l_i$ attached to $i$ as 
	$$ r_i := r_{\ep_i} \ \text{ and }\ l_i := l_{\ep_i}. $$
	We also define the dual raising and lowering operators $\rc_i^\star$ and $\lc_i^\star$ as 
	$$ \rc_i^\star := \rc_{\ep_i^\star} \ \text{ and }\ \lc_i^\star := \lc_{\ep_i^\star}. $$
\end{definition}

\begin{remark} 1. At this stage we have no restriction on the domain of $r_i$ and $l_i$, but later we will restrict them to $\trop(\Delta,\S)$.
	
2. There are also $r^{\ep_i}$ and $\rc^{\ep_i^\star}$ but they rarely appear in the article.
	
3. We will mostly deal with the actions of $r_i$ and $l_i$ on the $\delta$-vectors, and mention most dual statements for $\rc_i^\star$ and $\lc_i^\star$ without proof. 
Their relationship will be discussed in Section \ref{ss:dualcrystal}.
\end{remark}

\begin{lemma} \label{L:HomEinv} For any decorated representation $\mc M$ and any mutation sequence $\b u$, we have
\begin{align*} \hom(E_i, \mc M) = \hom(\mub(E_i), \mub(\mc M)) \ &\text{ and } \ \hom(\mc M,E_i^\star) = \hom(\mub(\mc M), \mub(E_i^\star)),\\
\e(\mc M, E_i) = \e(\mub(\mc M),\mub(E_i)) \ &\text{ and } \ \ec(E_i^\star, \mc M) = \ec(\mub(E_i^\star), \mub(\mc M)) .
\end{align*}
\end{lemma}
\begin{proof}
By Lemma~\ref{L:bdinv} and Proposition~\ref{P:Brep}(2), the projective
presentation of $E_i$ and the injective presentation of $E_i^\star$
have only frozen summands, before and after mutation. Thus the
corresponding multiplicities vanish at mutable vertices, and all
correction terms in Lemma~\ref{L:HEmu} are zero.
\end{proof}

\begin{lemma}\label{L:mu2S} If $i$ is reachable and $\b{u}$ is a mutation sequence with $\mub(E_i)=S_i$, then 
\begin{align*} r_i (\delta) &= \mub^{-1} (\mub(\delta) + \dtc_{S_i})  && \text{if $\e(\delta, E_i)> 0$;} \\
\shortintertext{Dually, if $i$ is dually reachable and $\b{u}$ is a mutation sequence with $\mu_{\b{u}}(E_i^\star)=S_i$, then}
\rc_i^\star(\dtc) &= \mub^{-1} (\mub(\dtc) + \delta_{S_i}	)  && \text{if $\ec(E_i^\star, \dtc)> 0$.} 
\end{align*}
\end{lemma}
\begin{proof} We only prove the statement for $r_i(\delta)$.
By Theorem \ref{T:rle} we have that 
$$r_i(\delta) = \mub^{-1} r_{S_i} (\mub(\delta)) = \mub^{-1} \left(\mub(\delta) + \delta_{S_i} + \rank(S_i, \tau\mub(\delta))B(\mub(\Delta)) \right).$$
By Lemmas \ref{L:H2E} and \ref{L:HomEinv} 
$$ \hom(S_i, \tau \mub(\delta)) = \e(\mub(\delta), S_i) = \e(\delta, E_i) >0. $$
So $\rank(S_i, \tau\mub(\delta)) = e_i$ and thus $r_i (\delta) = \mub^{-1} (\mub(\delta) + \dtc_{S_i})$.
\end{proof}

Later we will define an upper semi-normal crystal structure on the set of $\mu$-supported $\delta$-vectors.
The function $\e(-, E_i)$ will play a role of string length function for the operator $r_i$.
\begin{lemma}\label{L:er} If $i$ is reachable, then for any $\mu$-supported $\delta$, $\e(\delta, E_i)=0$ if and only if $\hom(r_i(\delta), E_i^\star)>0$.
Moreover,
$$ r_i(\delta)\in\trop(\Delta,\S) \quad\Longleftrightarrow\quad \e(\delta,E_i)>0. $$
Dually, if $i$ is dually reachable, then for any $\mu$-supported $\dtc$, $\ec(E_i^\star, \dtc)=0$ if and only if $\hom(E_i, \rc_i^\star(\dtc))>0$.
Moreover, $\rc_i^\star(\dtc)$ is $\mu$-supported exactly when $\ec(E_i^\star,\dtc)>0$.
\end{lemma}
\begin{proof}
Choose general decorated representations $\mc M,\mc R$ of weights
$\delta,r_i(\delta)$ as in Theorem~\ref{T:rlrigid}. Theorem~\ref{T:rlrigid} gives the following exact sequence of their
ordinary representations:
\begin{equation}\label{eq:inj} M\longrightarrow R\longrightarrow E_i \xrightarrow{h}\tauh\mc M,\end{equation}
where $h$ is chosen to have the generic rank. Since $M$ is $\mu$-supported,
\eqref{eq:inj} gives $R(j)=0$ at every frozen $j\ne i$, and
$0\longrightarrow R(i)\longrightarrow k \xrightarrow{h(i)}(\tauh\mc M)(i)$.
By Lemma~\ref{L:H2E},
$\hom(E_i,\tauh\mc M)=\e(\delta,E_i)$. Thus $h=0$ exactly
when this number is zero. If $h\ne0$, its kernel is a proper
submodule of $E_i$, hence is $\mu$-supported by
Proposition~\ref{P:Brep}(1). Therefore $h(i)$ is injective. Hence
$\dim R(i)=1$ if $\e(\delta,E_i)=0$, and $R(i)=0$ otherwise.
This proves the support criterion. Since $R$ is supported on
$\Delta^\mu[i]$, Definition~\ref{D:Brep} gives
$$ \Hom_J(R,E_i^\star) \cong\Hom_{J_{\mu[i]}}(R,I_{[i]})\cong R(i)^*. $$
Its dimension gives the asserted Hom criterion. The dual statements
follow by passing to the opposite QP.
\end{proof}

\begin{corollary}\label{C:pm1}
For a reachable frozen vertex $i$, consider the ambient operators $r_i,l_i$
on the full weight lattice. For
$\delta\in\trop(\Delta,\S)$, we have
$$ r_i(\delta)\in\trop(\Delta,\S) \quad\Longleftrightarrow\quad \e(\delta,E_i)>0. $$
Whenever these conditions hold,
$$ \e(r_i(\delta),E_i)=\e(\delta,E_i)-1. $$
Moreover, for every $\delta\in\trop(\Delta,\S)$,
$$ \e(l_i(\delta),E_i)=\e(\delta,E_i)+1. $$

Dually, for a dually reachable frozen vertex $i$ and $\dtc\in\check\trop(\Delta,\S)$,
$$ \rc_i^\star(\dtc)\in\check\trop(\Delta,\S) \quad\Longleftrightarrow\quad \ec(E_i^\star,\dtc)>0, $$
and, whenever this holds,
$$ \ec(E_i^\star,\rc_i^\star(\dtc))=\ec(E_i^\star,\dtc)-1. $$
For every $\dtc\in\check\trop(\Delta,\S)$, the operator satisfies
$$ \ec(E_i^\star,\lc_i^\star(\dtc))=\ec(E_i^\star,\dtc)+1, $$
but $\lc_i^\star(\dtc)$ need not be $\mu$-supported.
\end{corollary}
\begin{proof}
The support criterion is Lemma~\ref{L:er}, and the decrement
follows from Corollary~\ref{C:minexc} and
Proposition~\ref{P:Brep}(3).

For lowering, let $\mc M$ be general of weight $\delta$, and use the sequence
$L\rightarrow M\rightarrow E_i \xrightarrow{h'}\tauh\mc L$
with $\mc L$ general of weight $l_i(\delta)$. Since $M(i)=0$,
the map $h'(i)$ is injective and in particular $h'\ne0$. Thus
$\e(l_i(\delta),E_i)>0$, and the inverse identity in
Corollary~\ref{C:minexc} gives the increment without any support
assumption on $L$. The dual statements follow by passing to the
opposite QP.
\end{proof}

\begin{corollary}\label{C:lowering-support-rank}
Let $i$ be reachable, let $\mc M$ be general of weight
$\delta\in\trop(\Delta,\S)$, and let $\dtc$ be its dual weight.
Then $l_i(\delta)$ is
$\mu$-supported exactly when a general morphism
$$ g_{-1}:\tauh^{-1}\mc M\longrightarrow\tauh^{-1}E_i $$
is surjective at every frozen vertex.
If $(i,\ibar)$ is a $\tau$-exact pair with $\ibar$ dually reachable,
this criterion becomes
$$ l_i(\delta)\in\trop(\Delta,\S) \quad\Longleftrightarrow\quad \lambda_i(\delta):=\ec(\tau^{-1}E_i,\dtc)>0. $$
Whenever it holds,
$$ \e(l_i(\delta),E_i)=\e(\delta,E_i)+1, \qquad \lambda_i(l_i(\delta))=\lambda_i(\delta)-1. $$
\end{corollary}
\begin{proof}
The sequence in Theorem~\ref{T:rlrigid} contains
$\tauh^{-1}\mc M\xrightarrow{g_{-1}}\tauh^{-1}E_i \rightarrow L\rightarrow M$,
and $g_{-1}$ has general rank. At a frozen vertex $j$, exactness
and $M(j)=0$ identify $L(j)$ with $\coker g_{-1}(j)$.
If $\tau^{-1}E_i=E_{\ibar}^\star$, its only frozen space is
one-dimensional at $\ibar$, and every nonzero submodule contains
$S_{\ibar}$. Hence surjectivity at the frozen vertices is equivalent
to $g_{-1}\ne0$, or to
$$ \hom(\tau^{-1}\delta,\tau^{-1}E_i) =\ec(\tau^{-1}E_i,\dtc)>0. $$
The increment is Corollary~\ref{C:pm1}. For the decrement, the
$\dtc$-vector of $l_i(\delta)$ is
$l^{\ep_i}(\dtc)=\rc_{\tau^{-1}\ep_i}(\dtc)$, by
Remark~\ref{r:rl}. Apply the dual identity in
Corollary~\ref{C:minexc} to the dual boundary representation
$E_{\ibar}^\star$.
\end{proof}

\begin{remark}\label{r:pm1}
Let $\delta\in\trop(\Delta,\S)$, and let $\dtc$ be its corresponding
dual weight. Suppose that $\tau^{-1}E_i=E_{\ibar}^\star$ with $\ibar$ dually reachable. Write
$r^i=r^{\ep_i}$ and $l^i=l^{\ep_i}$. Remark~\ref{r:rl} gives
$$ l^i=\rc_{\tau^{-1}\ep_i}, \qquad r^i=\lc_{\tau^{-1}\ep_i}. $$
Thus the dual part of Corollary~\ref{C:pm1} gives, for the ambient
operators,
\begin{align*}
\ec(\tau^{-1}E_i,l^i(\dtc))
&=\ec(\tau^{-1}E_i,\dtc)-1
&&\text{if }\ec(\tau^{-1}E_i,\dtc)>0,\\
\ec(\tau^{-1}E_i,r^i(\dtc))
&=\ec(\tau^{-1}E_i,\dtc)+1.
\end{align*}
When the resulting weights remain $\mu$-supported, these are exactly
the changes of $\lambda_i$ under lowering and raising. The support
criterion for lowering is Corollary~\ref{C:lowering-support-rank}.
\end{remark}

\subsection{The Upper Seminormal Crystal with Weaker Weights} \label{ss:crystal}
Let $C_I=(c_{i,j})_{i,j\in I}$ be a generalized Cartan matrix, which will be assumed to be symmetric in this paper. 
We use a free weight lattice $\Lambda$, functionals $h_i$, and
root elements $\alpha_i$ with $\alpha_j(h_i)=c_{i,j}$.
Two cases are used: $\Lambda=P_{\mathrm{cor}}$, where $\alpha_i$
in the following axioms means $\bar\alpha_i$, and the weight lattice
of a full realization. The former is called a crystal with
\emph{coroot-coordinate weights}; independence of its root images is
not required.
\begin{definition} A {\em Kashiwara crystal} (or {\em crystal} for short) of Cartan type $C_I$ with the chosen weight datum is a nonempty set $\mc{B}$ together with maps \begin{align*}
		r_i, l_i&: \mc{B} \to \mc{B} \sqcup\{0\},\\
		\rho_i, \lambda_i&: \mc{B} \to \mb{Z} \sqcup\{-\infty\},\\
		{\wt}&:\mc{B} \to \Lambda,
	\end{align*}
	where $i\in I$ and $0\notin \mc{B}$ is an auxiliary element, satisfying the following conditions:
	\begin{enumerate}
		\item[A1.] If $x,y\in \mc{B}$ then $r_i(x)=y$ if and only if $l_i(y)=x$. In this case, it is assumed that
		$$ {\wt}(y)={\wt}(x)+\alpha_i,\quad \rho_i(y)=\rho_i(x)-1, \quad \lambda_i(y)=\lambda_i(x)+1. $$
		\item[A2.] We require that 
		$$ \lambda_i(x) = \innerprod{{\wt}(x), h_i} + \rho_i(x) $$
		for all $x\in\mc{B}$ and $i\in I$. In particular, if $\lambda_i(x)=-\infty$, then $\rho_i(x)=-\infty$. In this case, 
		we require that $l_i(x)=r_i(x)=0$.
	\end{enumerate}
We will tacitly assume that $r_i(x)$ or $l_i(x)$ is mapped to the auxiliary element $0$ if it is not in $\mc{B}$.
\end{definition}

Let $\mf T$ be the labeled mutation tree rooted at the
reference seed, each vertex carrying the QP obtained along its mutation
path. We keep distinct vertices of the tree even when their QPs are
isomorphic.
By abuse of language, we also call an element $t\in \mf{T}$ a seed.
A seed $t'$ obtained from $t$ by a sequence $\mub$ of mutations is denoted by $\mub:t\to t'$.
\begin{definition}\label{D:crystal-cluster} By a {\em crystal cluster} structure of $\trop(\Delta,\S)$, we mean a family of crystal structures $\{\trop(\Delta,\S)_t\}$ indexed by $t\in \mf{T}$ such that
	$(r_i,l_i; \rho_i,\lambda_i; \wt)_t$ are compatible with mutations:
\begin{align*} \mu_u(r_i(\delta)) &= r_i'(\delta') & \mu_u(l_i(\delta)) &= l_i'(\delta') \\
	\rho_i(\delta) &= \rho_i'(\delta') & \lambda_i(\delta) &= \lambda_i'(\delta') \\
	\wt(\delta) &= 	\wt'(\delta'),&
\end{align*} 
where $(r_i,l_i; \rho_i,\lambda_i; \wt)=(r_i,l_i; \rho_i,\lambda_i; \wt)_t$ and $(r_i',l_i'; \rho_i',\lambda_i'; \wt')=(r_i,l_i; \rho_i,\lambda_i; \wt)_{t'}$ with $t \edge{u} t'$.
\end{definition}

\begin{remark} Let $\mc{B}$ and $\mc{B}'$ be Kashiwara crystals.  
Whenever we have a bijective map $\sigma: \mc{B}\to \mc{B}'$ we can transfer the crystal structure on $\mc{B}$ to $\mc{B}'$ by letting
	\begin{align*}
		r_i^\sigma(\delta)  = \sigma r_i(\sigma^{-1} (\delta)) \quad &\text{and} \quad  l_i^\sigma(\delta)  = \sigma l_i(\sigma^{-1} (\delta)) \\
		\rho_i^\sigma(\delta) = \rho_{i}(\sigma^{-1} (\delta)) \quad &\text{and} \quad  \lambda_i^\sigma(\delta) = \lambda_i(\sigma^{-1} (\delta)) \\
		\wt_i^\sigma(\delta) = \wt_i(\sigma^{-1} (\delta)).
	\end{align*}
By Lemma \ref{L:musupp}, the mutation induces a bijection $\trop(\Delta,\S) \to \trop \mu_u(\Delta,\S)$.
One can think of the crystal cluster structure of $\trop(\Delta,\S)$ as a single crystal structure of $\trop(\Delta,\S)$ transferring to other seeds by mutations.
Equivalently, one can say that each mutation $\mu_u$ induces a crystal isomorphism $\trop(\Delta,\S) \to \trop \mu_u(\Delta,\S)$.

We will consider a stronger compatibility in
Section~\ref{ss:tau+}, requiring crystal structures
to be compatible with relative \(\tau_\mu\)-translations
as well as mutations.
\end{remark}

Compatibility with the exchange matrix supplies mutation-invariant
linear weight functions, but an adapted grading need not be integral.
We therefore use the following explicit relaxation of the crystal
axioms. A \emph{crystal with weaker weights} has the usual data and
satisfies axioms A1 and A2 of a Kashiwara crystal, except that its
weight and string functions may take values in
$$
\operatorname{wt}:\mathcal B\longrightarrow\Lambda_{\mathbb Q},
\qquad
\rho_i,\lambda_i:\mathcal B\longrightarrow
\mathbb Q\sqcup\{-\infty\}.
$$
Here $\Lambda_{\mathbb Q}=\Lambda\otimes_{\mathbb Z}\mathbb Q$ uses
the chosen lattice: $P_{\mathrm{cor}}$ in the general construction,
or a full weight lattice when a completion has been specified.
The conventions at $-\infty$ and the definitions of upper and lower
seminormality are unchanged. In particular, upper seminormality forces
$\rho_i$ to be a nonnegative integer, but imposes no such condition
on $\lambda_i$. Lower seminormality forces $\lambda_i$ to be a
nonnegative integer. When the weight and both string parameters take integral values,
we obtain an integral crystal for the chosen weight lattice; extending
the coroot-coordinate grading to a full weight lattice requires a
separate choice.

\begin{lemma}\label{L:full-weight-coordinates}
For coroot-coordinate weights, the weight part of A1 is equivalent to
$$ \wt_j(y)-\wt_j(x)=c_{j,i}\qquad(j\in I) $$
whenever $r_i(x)=y\ne0$. For a full completion it is equivalent to
$\wt_a(y)-\wt_a(x)=\alpha_i(h_a)$ for every $a\in I\sqcup K$;
its $I$-coordinates alone are not sufficient in singular type.
In either case, A2 is equivalent to
\begin{equation}\label{eq:A2wt} \lambda_i(x)=\rho_i(x)+\wt_i(x)\qquad(i\in I).\end{equation}
\end{lemma}
\begin{proof}
Evaluate in the basis $(h_i)_{i\in I}$ of $Q^\vee$, or in the full
basis $(h_a)_{a\in I\sqcup K}$, respectively.
\end{proof}

\begin{definition} A crystal (with weaker weights) $\mc{B}$ is called {\em seminormal} if
	$$\rho_i(x) = \max\{k\in\mb{Z}_{\geq 0} \mid r_i^k(x)\neq 0\}\ \text{ and }\ \lambda_i(x) = \max\{k\in\mb{Z}_{\geq 0} \mid l_i^k(x)\neq 0\}.$$
	If just the first (resp. second) condition is assumed, we say $\mc{B}$ is {\em upper seminormal} (resp. {\em lower seminormal}).
\end{definition}

\begin{lemma}\label{L:rholam} Suppose that $\rho_i(y)=\rho_i(x)-1$ and $\wt_i(y)=\wt_i(x)+2$. If we define $\lambda_i=\rho_i+\wt_i$, then $\lambda_i(y)=\lambda_i(x)+1$. 
\end{lemma}
\begin{proof} $\lambda_i(y) = \rho_i(y)+\wt_i(y) = (\rho_i(x)-1)+ (\wt_i(x)+c_{i,i}) = \rho_i(x)+\wt_i(x)+1 = \lambda_i(x)+1$.
\end{proof}

Fix a set $I$ of frozen vertices and its Cartan matrix $C_I$.
The default weight map takes values in $P_{\mathrm{cor}}\otimes\mathbb Q$.
Mutate every chosen compatible form together. If \eqref{eq:existwt}
holds and a full completion is chosen, include its auxiliary forms as
well. The roots, or their coroot-coordinate images, are fixed throughout
the mutation class by Lemma~\ref{L:muwt}.

\begin{theorem}\label{T:upper}
Assume that $(\Delta,\S)$ is frozen-Jacobi-finite. Let $I$ be a set of reachable frozen
vertices and let $(\wt_i)_{i\in I}$ be a compatible grading adapted to
$I$. Then $\mathcal B=\trop(\Delta,\S)$ has an upper seminormal
crystal cluster structure of Cartan type $C_I$ with
$\wt=\sum_i\wt_i\bar\varpi_i$, with $r_i,l_i$ as in
Definition~\ref{D:lrv}, and
$$ \rho_i(\delta)=\e(\delta,E_i),\qquad \lambda_i(\delta)=\rho_i(\delta)+\wt_i(\delta). $$
Outputs outside $\mathcal B$ are sent to the auxiliary $0$.
A rational grading gives the weaker rational version; an integral
grading gives an integral crystal with coroot-coordinate weights.
\end{theorem}
\begin{proof}  The crystal structure is compatible with mutations due to Theorem \ref{T:rle}, Lemmas \ref{L:HomEinv}, \ref{L:muwt}, and \ref{L:bdinv}. Then we will verify the crystal axioms for a fixed seed $t$.
	
The axiom A2 is now \eqref{eq:A2wt} which is trivially satisfied due to the definition. It remains to verify A1.
The fact that $r_i(\delta)=\eta$ if and only if $l_i(\eta)=\delta$ is the content of Lemma \ref{L:rlid}.
For the weight axiom, let $\eta=r_i(\delta)\in\mathcal B$.
Theorem~\ref{T:rlrigid} gives
$$ \eta-\delta=\check\epsilon_i-\gamma B_\Delta, \qquad \gamma=\rank(\eta,\epsilon_i), $$
with $\gamma$ supported on $\Delta_0^\mu$, since $\eta$ is $\mu$-supported.
Compatibility and adaptation give
\begin{equation}\label{eq:A1wtr} \wt_j(\eta)-\wt_j(\delta)=\wt_j(\check\epsilon_i)=c_{j,i} \qquad(j\in I).\end{equation}
Thus $\wt(\eta)-\wt(\delta)=\bar\alpha_i$.

Corollary~\ref{C:pm1} gives
$\rho_i(\eta)=\rho_i(\delta)-1$. Since \eqref{eq:A1wtr} gives
$\wt_i(\eta)=\wt_i(\delta)+2$, Lemma~\ref{L:rholam} then gives
$\lambda_i(\eta)=\lambda_i(\delta)+1$.
	
Finally, by Lemma~\ref{L:er}, $r_i(\delta)\neq0\ \Longleftrightarrow\ \rho_i(\delta)=\e(\delta,E_i)>0$,
and whenever $r_i(\delta)\neq0$, $\rho_i(r_i(\delta))=\rho_i(\delta)-1$ by Corollary~\ref{C:pm1}.
Therefore
$$ \rho_i(\delta)=\max\{m\ge0\mid r_i^m(\delta)\neq0\}. $$
Thus the crystal is upper seminormal.	
\end{proof}

\begin{corollary}[Optional full-weight completion]\label{C:full-weight-upgrade}
Under \eqref{eq:existwt}, the operators and parameters of
Theorem~\ref{T:upper} admit a full weight map and realization supplied by
Lemma~\ref{L:wtc}. The same assertion applies to the seminormal crystal
of Theorem~\ref{T:crystal}. Integrality is preserved.
\end{corollary}
\begin{proof}
Apply every completed compatible form to
$\eta-\delta=\check\epsilon_i-\gamma B_\Delta$. Its value is
$\alpha_i(h_a)$ for every $a\in I\sqcup K$, including the
auxiliary directions. The operators and their string parameters do not
change.
\end{proof}

\noindent The prototypical examples of this type in the cluster theory are the coordinate ring of the maximal unipotent subgroups of a simple simply-connected complex algebraic groups. These classical examples will be briefly reviewed in Section \ref{ss:unipotent}.

\begin{remark} At this stage it is unclear (from the proof of Theorem \ref{T:upper}) why the Cartan type of the crystal is determined by \eqref{eq:wtC}.
But we will see in Section \ref{S:lifting} that this is the correct definition for an algebraic lift of $\mc{B}$.
\end{remark}

For a set of dually reachable frozen vertices, the dual construction uses the dual
boundary representations and the set
$\check{\mathcal B}=\check\trop(\Delta,\S)$. A dually adapted
compatible grading gives the coroot-coordinate weight map
$\check\wt=\sum_i\check\wt_i\bar\varpi_i$, without an independence
hypothesis. An optional full lift uses the dual of Lemma~\ref{L:wtc},
with $\epsilon_i^\star$ in place of $\check\epsilon_i$.
The operators and parameters are

\begin{align*}
	\rc_i^\star(\dtc)&=\rc_{\ep_i^\star}(\dtc) &  \check{\rho}_i^\star(\dtc) &= \ec(E_i^\star, \dtc) \\
	\lc_i^\star(\dtc)&=\lc_{\ep_i^\star}(\dtc) &  \check{\lambda}_i^\star(\dtc) &= \check{\wt}_i(\dtc) + \check{\rho}_i^\star(\dtc).
\end{align*}
\noindent As remarked below Definition \ref{D:CartanI}, this crystal has the same Cartan type as $\trop(\Delta,\S)$ but in general the weight function $(\check{\wt}_i)_{i\in I}$ can be different from $({\wt}_i)_{i\in I}$.

\subsection{The Seminormal Crystal Structure} \label{ss:seminormal}
Next we discuss the nicest situation from a crystal perspective.

\begin{theorem}\label{T:crystal}
Assume that $(\Delta,\S)$ is frozen-Jacobi-finite. Let $\{(i,\ibar)\}_{i\in I}$ be
$\tau$-exact pairs with $i$ reachable and $\ibar$ dually reachable. The natural integral
grading $\wt_i=\wt_{\epsilon_i}$ in \eqref{eq:wti} gives
$\mathcal B$ a seminormal crystal cluster structure of Cartan type
$C_I$, with coroot-coordinate weights, $r_i,l_i$ as in
Definition~\ref{D:lrv}, and
$$ \rho_i(\delta)=\e(\delta,E_i),\qquad \lambda_i(\delta)=\ec(\tau^{-1}E_i,\check\delta). $$
The optional full-weight refinement is Corollary~\ref{C:full-weight-upgrade}.
\end{theorem}
\begin{proof}
Corollary~\ref{C:tauexact} supplies the integral compatible adapted
grading. Theorem~\ref{T:upper} gives an upper seminormal crystal with
parameters $\rho_i$ and $\rho_i+\wt_i$. Proposition~\ref{P:weight}
identifies the latter with the stated $\lambda_i$: for corresponding
$\mu$-supported weights, $\dtc-\delta$ lies in the row space of
$B_\Delta$, so $\wt_i(\dtc)=\wt_i(\delta)$.

It remains to prove lower seminormality. The dual weight of the ambient
$l_i(\delta)$ is $l^{\ep_i}(\dtc)$, and
$$ l^{\ep_i}(\dtc) =\rc_{\tau^{-1}\ep_i}(\dtc) =\rc_{\ibar}^\star(\dtc) $$
by Remark~\ref{r:rl} and $\tau$-exactness. The dual
support criterion in Lemma~\ref{L:er} gives
$$ l_i(\delta)\in\trop(\Delta,\S) \quad\Longleftrightarrow\quad \lambda_i(\delta)>0, $$
and, whenever this holds,
$\lambda_i(l_i(\delta))=\lambda_i(\delta)-1$ by the dual part
of Corollary~\ref{C:pm1}.
With the convention that a crystal operator takes the value $0$
when its output is not supported, iteration gives
$$ \lambda_i(\delta) =\max\{m\ge0:l_i^m(\delta)\ne0\}. $$
Thus the crystal is lower seminormal as well. The weight and raising
identities are already those of Theorem~\ref{T:upper}; compare
Remark~\ref{r:pm1}.
\end{proof}
\noindent The prototypical examples of this type in the cluster algebra theory is the base affine spaces and the affine coordinate ring of the Grassmannians. We will briefly review them in Section \ref{ss:G/U} and \ref{ss:Grass}.

\begin{warning} Using the dual boundary representations we can also get an upper seminormal crystal structure on $\check{\trop}(\Delta, \S)$ as in Section \ref{ss:dualcrystal}.
	However, this structure in general cannot be upgraded to a seminormal crystal as in Theorem \ref{T:crystal}.
	For this upgrading, one should require $\{(i,\ibar)\}_{i\in I}$ to be a set of {\em dual} $\tau$-exact pairs, that is, $\tau E_i^\star = E_{\ibar}$.
But this does not follow from its being a set of $\tau$-exact pairs.
\end{warning}

\section{Cluster Automorphisms and Relative $\tau$-Translation}\label{S:auto}
We compare the boundary crystals of QPs related by cluster automorphisms
and relative $\tau$-translation.

\subsection{Cluster Automorphisms and Crystal Duality}\label{ss:dualcrystal}
Throughout this section, permutations of vertices preserve the mutable and
frozen subsets, and bijections of crystal sets fix $0_{\mathrm{cr}}$.

\begin{definition}\label{D:Cauto}
A \emph{cluster automorphism} of $\Delta$ is a pair
$\sigma=(\mub,\pi)$ consisting of a sequence of mutable mutations and
a permutation of vertices such that
$$ \pi\mub(\Delta)=\Delta \qquad\text{or}\qquad
   \pi\mub(\Delta)=\Delta^{\opp}, $$
where arrows between frozen vertices are disregarded. We call $\sigma$
\emph{direct} in the first case and \emph{opposite} in the second.
\end{definition}

Put $(\Delta',\S')=\pi\mub(\Delta,\S)$, with all frozen arrows
included, and set
$\mc B=\trop(\Delta,\S)$ and $\mc B'=\trop(\Delta',\S')$.
By Lemma~\ref{L:musupp}, $\sigma$ induces a bijection
$\mc B\to\mc B'$. Let $I$ be a set of reachable frozen vertices,
with a compatible adapted grading. Writing
$r_i^\sigma=\sigma r_i\sigma^{-1}$ and
$\rho_i^\sigma=\rho_i\circ\sigma^{-1}$, mutation compatibility gives
\begin{equation}\label{eq:crystalsigma}
 r_i^\sigma=r'_{\pi(i)},\qquad
 \rho_i^\sigma=\rho'_{\pi(i)},\qquad
 \wt_i^\sigma=\pi\mub(\wt_i).
\end{equation}
The primes refer to $(\Delta',\S')$. The lowering operators and
lower string functions are determined by these data. Comparing them
with those of the original QP requires the following lemma.

\begin{lemma}\label{L:numerical-boundary-comparison}
Let two ice QPs be obtained by mutation from reference QPs satisfying
the standing assumptions. Suppose their rectangular exchange matrices
and sets of $\mu$-supported projective weights agree after identifying
the vertices. For every frozen vertex reachable in both QPs, their
operators $r_i,l_i$ and functions $\rho_i$ agree. If the compatible
gradings agree, so do the weight maps and $\lambda_i$.
The dual assertion holds for dually reachable vertices when the sets
of supported injective weights agree.
\end{lemma}
\begin{proof}
Choose a sequence $\mub$ taking the first boundary representation
$E_i$ to $S_i$. The rectangular matrices agree after each mutation,
so the second boundary representation also becomes $S_i$.
For $\eta=\mub(\delta)$, Lemma~\ref{L:HomEinv} gives
$\rho_i(\delta)=[-\eta(i)]_+$. At this seed,
Lemma~\ref{L:boundary-propagation} excludes incoming frozen arrows
at $i$, and hence
$$ \check\delta_{S_i}=e_i-
   \sum_{u\in\Delta_0^\mu}[b^{\mathrm{opt}}_{ui}]_+e_u. $$
If $\rho_i(\delta)>0$, Lemma~\ref{L:mu2S} gives
$r_i(\delta)=\mub^{-1}(\eta+\check\delta_{S_i})$; otherwise
$r_i(\delta)=0_{\mathrm{cr}}$. These formulas are the same for both
QPs. The lowering operators are inverse partial maps, and
$\lambda_i=\rho_i+\wt_i$. The dual proof is identical.
\end{proof}

\begin{corollary}\label{C:directauto}
Let $\sigma=(\mub,\pi)$ be direct. Assume that both
$(\Delta,\S)$ and $(\Delta',\S')$ arise from reference ice QPs
by mutation, and suppose that
\begin{equation}\label{eq:direct-supported-return}
 \trop(\Delta',\S')=\trop(\Delta,\S)
 \quad\text{in the common vertex coordinates}.
\end{equation}
Then $\sigma$ acts on $\mc B$ and
$$ \sigma r_i=r_{\pi(i)}\sigma,\qquad
   \sigma l_i=l_{\pi(i)}\sigma,\qquad
   \rho_{\pi(i)}\circ\sigma=\rho_i \quad(i\in I). $$
If $\pi(I)=I$ and $\pi\mub(\wt_i)=\wt_{\pi(i)}$ for every
$i\in I$, this is a crystal automorphism with coroot-coordinate
weights and indices permuted by $\pi$. For a full weight map $W$,
one also requires $W\circ\sigma=\pi_\Lambda\circ W$, where
$\pi_\Lambda$ is an automorphism of the full weight datum inducing
$\pi$ on the simple roots and coroots.
\end{corollary}
\begin{proof}
Apply Lemma~\ref{L:numerical-boundary-comparison} to
\eqref{eq:crystalsigma}.
\end{proof}

The proof uses Lemma~\ref{L:boundary-propagation}, which excludes
arrows from other frozen vertices to an optimized vertex $i$. Without
this property, $\check\delta_{S_i}$ can depend on frozen arrows not
recorded by the rectangular exchange matrix; compare
\cite[Section~9.2]{GSDT}.

Now let $\sigma$ be opposite. Assume that both $(\Delta',\S')$
and $(\Delta,\S)^{\opp}$ arise from reference ice QPs by
mutation, and suppose that
\begin{equation}\label{eq:opposite-supported-return}
 \trop(\Delta',\S')=\trop\bigl((\Delta,\S)^{\opp}\bigr)
 \quad\text{in the common vertex coordinates}.
\end{equation}
By vector-space duality, we identify
$\trop(\Delta',\S')$ with
$\check{\mc B}=\check\trop(\Delta,\S)$; the boundary representation
at $j$ corresponds to $(E_j^\star)^*$. The preceding comparison gives
\begin{equation}\label{eq:sigmaopp}
 r_i^\sigma=\rc_{\pi(i)}^\star,\qquad
 \rho_i^\sigma=\rhoc_{\pi(i)}^\star,\qquad
 \wt_i^\sigma=\pi\mub(\wt_i),
\end{equation}
with the transported grading on the injective-weight lattice.
By Theorem~\ref{T:genpi}, the projective and injective weights of
the same general decorated representation are related by
$$ \vee:\mc B\xrightarrow{\sim}\check{\mc B},\qquad
   \delta^\vee=\delta+\dv(\delta)B(\Delta). $$
Here $\dv(\delta)$ is the ordinary dimension vector, with zero
frozen coordinates. On $\mc B$, write
$r_j^\star=\vee^{-1}\rc_j^\star\vee$ and
$\rho_j^\star=\rhoc_j^\star\circ\vee$, and similarly for the other
dual crystal data.

\begin{definition}\label{D:Kmap}
The \emph{generalized Kashiwara map} associated with an opposite
cluster automorphism $\sigma$ satisfying
\eqref{eq:opposite-supported-return} is
$$ \kappa_\sigma=\vee^{-1}\circ\sigma:\mc B\xrightarrow{\sim}\mc B. $$
\end{definition}

Conjugating \eqref{eq:sigmaopp} by $\vee$ gives
$$ r_i^{\kappa_\sigma}=r_{\pi(i)}^\star,\qquad
   \rho_i^{\kappa_\sigma}=\rho_{\pi(i)}^\star,\qquad
   \wt_i^{\kappa_\sigma}=\pi\mub(\wt_i). $$
For the weight identity, use that every compatible form has the same
value on $\delta$ and $\delta^\vee$. The map $\kappa_\sigma$ need
not be involutive. The standard unipotent example gives Kashiwara's
involution on $\mc B(\infty)$; see Section~\ref{ss:unipotent}.

At a reference QP, put $\eta=\sigma(\delta)$ and
$u=\tau_\mu(-\eta_\mu)$. Then
$$ (\kappa_\sigma\delta)_\mu=u,\qquad
   \kappa_\sigma(\delta)=\eta-d_\mu(u)B_\Delta, $$
where $d_\mu(u)$ is the generic ordinary mutable dimension vector.
Indeed, Lemma~\ref{L:musupp-param} gives
$\eta_\mu=\vee_\mu((\kappa_\sigma\delta)_\mu)$, while decorated
AR-translation and Theorem~\ref{T:genpi} give
$\vee_\mu\tau_\mu=-\Id$. The second formula follows from
\eqref{eq:delta2dual}. Neither formula requires finite Cartan type.

Suppose now that Theorem~\ref{T:crystal} applies, and put
$F_i=\tau^{-1}E_i=E_{\ibar}^\star$. Define
$$ \check\wt_{\ibar}(\check\delta)
   =\check\delta(\dv F_i-\dv E_i). $$
The Cartan matrices on $I$ and $\bar I$ agree by $\tau$-invariance
of the symmetrized $\E$-pairing. Since
$\epsilon_{\jbar}^\star=-\check\epsilon_j$, these forms are dually
adapted. Compatibility gives
$\check\wt_{\ibar}(\delta^\vee)=-\wt_i(\delta)$.
Remark~\ref{r:rl} and the proof of Theorem~\ref{T:crystal} yield
$$ r_{\ibar}^\star=l_i,\qquad l_{\ibar}^\star=r_i,\qquad
   \rho_{\ibar}^\star=\lambda_i,\qquad
   \lambda_{\ibar}^\star=\rho_i. $$

Let $I_0\subseteq I$ have finite Cartan type, and suppose that
$\mc B|_{I_0}$ is a disjoint union of crystals of finite-dimensional
simple highest-weight modules; see Corollary~\ref{C:normal}.
Write $w_{0,I_0}\alpha_i=-\alpha_{i^*}$, and let $\mc H_\lambda$
be the set of highest elements of restricted weight $\lambda$.
The unique isomorphism onto each highest-weight component identifies
$$ \mc B|_{I_0}\cong
   \bigsqcup_\lambda\mc B_{I_0}(\lambda)\times\mc H_\lambda. $$
Let $\xi_\lambda$ be Lusztig's involution on
$\mc B_{I_0}(\lambda)$: it exchanges the highest and lowest elements
and satisfies $\xi_\lambda r_i=l_{i^*}\xi_\lambda$
\cite[Section~2.2]{HKc}. Its componentwise action on $\mc B$ is
denoted by $\xi_{I_0}$.

\begin{proposition}\label{P:Lusztig-Kmap}
Let $\sigma=(\mub,\pi)$ be an opposite cluster automorphism satisfying
\eqref{eq:opposite-supported-return}, with
$\pi(i)=\overline{i^*}$ for every $i\in I_0$.
There is a unique permutation $p_{\sigma,\lambda}$ of each
$\mc H_\lambda$ such that
$$ \kappa_\sigma(b,h)=
   \bigl(\xi_\lambda(b),p_{\sigma,\lambda}(h)\bigr). $$
Equivalently, $A_\sigma=\xi_{I_0}\kappa_\sigma$ is an automorphism
of the restricted crystal and
$$ \kappa_\sigma=\xi_{I_0}A_\sigma=A_\sigma\xi_{I_0}. $$
Thus $\kappa_\sigma=\xi_{I_0}$ precisely when it preserves every
connected $I_0$-component.
\end{proposition}
\begin{proof}
The preceding identities give
$\kappa_\sigma r_i=r_{\overline{i^*}}^\star\kappa_\sigma
=l_{i^*}\kappa_\sigma$ and, by inversion,
$\kappa_\sigma l_i=r_{i^*}\kappa_\sigma$.
Seminormality gives
$$ \rho_{i^*}(\kappa_\sigma\delta)=\lambda_i(\delta),\qquad
   \lambda_{i^*}(\kappa_\sigma\delta)=\rho_i(\delta), $$
so $\wt_{I_0}(\kappa_\sigma\delta)=w_{0,I_0}\wt_{I_0}(\delta)$.
Hence the image of $h\in\mc H_\lambda$ is the lowest element of a
component of highest weight $\lambda$. Denote its highest element by
$p_{\sigma,\lambda}(h)$. This defines a permutation, and the
intertwining identities determine the map on the whole component.
In these coordinates $A_\sigma(b,h)=(b,p_{\sigma,\lambda}(h))$,
which proves the remaining assertions.
\end{proof}

Moreover, $\kappa_\sigma^2=\Id$ if and only if
$p_{\sigma,\lambda}^2=\Id$ for every $\lambda$.
The permutations are trivial if every highest weight occurs at most
once. More generally, the same conclusion holds if a grading $D$
is constant on components, $D\circ\kappa_\sigma=D$, and
$(D(h),\wt_{I_0}(h))$ distinguishes highest elements.
Finally, uniqueness of Lusztig's involution on each connected
component gives $\varphi\xi_{I_0}=\xi'_{I_0}\varphi$ for every
crystal isomorphism $\varphi$ preserving the indices. In particular,
this applies to the relative translations of Section~\ref{ss:tau+}
with their transported gradings.

\subsection{Relative $\tau$-Translation}\label{ss:tau+}
Let $(\Delta,\S)$ be a reference QP, and put
$\mathbf E^\mu=(\mc E_j^\mu)_{j\in\Delta_0^{\op{fr}}}$, where
$\mc E_j^\mu$ is the mutable decorated representation associated
with the extracted presentation at $j$. We compare its boundary
crystal with that of a reference QP associated with
$\tau_\mu\mathbf E^\mu$.

\begin{definition}\label{D:taumu-shifted-QP}
A \emph{relative $\tau$-model} of $(\Delta,\S)$ is a reference
ice QP $(\Delta^\tau,\S^\tau)$ with the same mutable QP and frozen
vertex set, such that
$$ \mathbf E^{\mu,\tau}\cong\tau_\mu\mathbf E^\mu. $$
Write $J^\tau$ for its full Jacobian algebra.
\end{definition}

We assume that $J$ and $J^\tau$ are finite-dimensional.
Each $i\in I$ is reachable in both reference QPs. For the dual
assertions, let $I^\star$ be a set of frozen vertices that are
dually reachable in both QPs.

\begin{definition}\label{D:iotaV}
For a mutable weight $\delta_\mu$, we define the {\em tropical extension map}
\begin{equation}\label{eq:iotaV}
 \iota_{\mathbf E^\mu}(\delta_\mu)
 =\left(\delta_\mu,
 -\e_{J_\mu}(\mc E_j^\mu,\delta_\mu)_{j\in\Delta_0^{\op{fr}}}\right).
\end{equation}
By \eqref{eq:vertex-extraction}, this is the weight of the extension
by zero of a general decorated representation of weight
$\delta_\mu$.
\end{definition}

\begin{lemma}\label{L:iota-general}
The weight $\iota_{\mathbf E^\mu}(\delta_\mu)$ is $\mu$-supported.
\end{lemma}
\begin{proof}
Extend a general mutable decorated representation by zero. Its
minimal projective presentation has no common mutable summands and
no positive frozen term, so it belongs to
$\PHom(\iota_{\mathbf E^\mu}(\delta_\mu))$. Vanishing of the
frozen cokernel is an open condition; hence a general presentation
of this weight is also $\mu$-supported.
\end{proof}

\begin{lemma}\label{L:musupp-param}
For a $\mu$-supported weight $\delta$, the mutable restriction
of a general presentation of weight $\delta$ is general of weight
$\delta_\mu$. Every $\mu$-supported weight
has a unique expression
\begin{equation}\label{eq:musupp-param}
 \delta=\iota_{\mathbf E^\mu}(\delta_\mu)-\mathbf a,
 \qquad \delta_\mu=\delta|_{\Delta_0^\mu},\qquad
 \mathbf a\in\mb Z_{\ge0}^{\Delta_0^{\op{fr}}}.
\end{equation}
Conversely, every weight of this form is $\mu$-supported.
\end{lemma}
\begin{proof}
For a $\mu$-supported weight, the degree-zero term of the minimal projective presentation has no frozen summands. The quotient $J\twoheadrightarrow J_\mu$ gives a surjective map
$\PHom_J(\delta)\rightarrow\PHom_{J_\mu}(\delta_\mu)$:
every matrix of mutable paths lifts, and the negative frozen terms
vanish in the quotient. 
Thus a general full presentation restricts to a general mutable
presentation. Let $\mc N$ be its associated mutable decorated
representation. The ordinary part of the full decorated representation
is supported on mutable vertices. Hence the full representation is the
extension by zero of $\mc N$, together with any negative frozen
decorations.
Formula~\eqref{eq:vertex-extraction} gives
\eqref{eq:musupp-param}; the finitely many pairings can be taken at
their generic values simultaneously. The mutable coordinates determine
$\delta_\mu$, and the frozen coordinates then determine $\mathbf a$.
Conversely, add negative frozen summands to the presentation in
Lemma~\ref{L:iota-general} and use the same open condition.
\end{proof}

\begin{lemma}\label{L:boundary-comparison}
	Suppose two ice QPs have isomorphic mutable
	Jacobian algebras, preserving vertex idempotents,
	and corresponding $\mc E_j^\mu$ are isomorphic.
	Then their $\mu$-supported weight sets and
	boundary operators and upper string functions
	agree, both ordinarily and dually. Their weight
	maps and lower string functions also agree when
	their compatible gradings do.
\end{lemma}
\begin{proof}
Lemma~\ref{L:musupp-param} identifies the sets of $\mu$-supported
weights. The mutable algebra determines the principal exchange matrix,
and the weights of $\mc E_j^\mu$ determine its frozen columns.
Apply Lemma~\ref{L:numerical-boundary-comparison}.
For the dual assertion, \eqref{eq:vertex-extraction} gives the
$\mu$-supported injective weights as
$$ \left(\check\delta_\mu, -\hom(\mc E_j^\mu,N)_{j\in\Delta_0^{\op{fr}}}\right)-\mathbf a, \qquad \mathbf a\ge0, $$
where $N$ is the ordinary part of a general mutable decorated
representation of injective weight $\check\delta_\mu$.
These weights are likewise determined by the same data.
\end{proof}

\begin{remark}\label{r:Jac-comparison-criterion}
An isomorphism of the full Jacobian algebras preserving the vertex
idempotents also identifies the sets of $\mu$-supported
weights and the boundary representations, hence their crystal
operators. In particular,
right-equivalence of the full QPs is sufficient.
\end{remark}

\begin{definition}\label{D:taumu-on-trop}
For $\delta=\iota_{\mathbf E^\mu}(\delta_\mu)-\mathbf a$, define
\begin{equation}\label{eq:taumu-on-delta} \rtau\delta :=\iota_{\tau_\mu\mathbf E^\mu}(\tau_\mu\delta_\mu)-\mathbf a,\end{equation}
and set $\rtau(0_{\mathrm{cr}})=0_{\mathrm{cr}}$.
\end{definition}

\begin{lemma}\label{L:tautrop}
The map \eqref{eq:taumu-on-delta} is a bijection
$\rtau:\trop(\Delta,\S)\xrightarrow{\sim}\trop(\Delta^\tau,\S^\tau)$.
\end{lemma}
\begin{proof}
In the coordinates $(\delta_\mu,\mathbf a)$ of
Lemma~\ref{L:musupp-param}, it is
$(\delta_\mu,\mathbf a)\mapsto(\tau_\mu\delta_\mu,\mathbf a)$.
Theorem~\ref{T:genpi} gives its inverse.
\end{proof}

\begin{lemma}\label{L:res2mu}
Suppose that $i$ is reachable and $\delta\in\trop(\Delta,\S)$.
If $\rho_i(\delta)>0$, then
$$ (r_{E_i}(\delta))^\mu=r_{\mc E_i^\mu}(\delta^\mu). $$
\end{lemma}
\begin{proof}
Choose a mutation sequence $\mub$ with $\mub(E_i)=S_i$.
At this seed, $\mub(\mc E_i^\mu)$ is a negative decorated
representation. Put $\delta'=\mub(\delta)$, and let $\mc M'$
be general of weight $\delta'$. Since $\rho_i(\delta)>0$,
a general map $S_i\to\tau\mc M'$ is nonzero and hence has
image $S_i$. Thus
$$ (r_{S_i}(\delta'))^\mu =(\delta')^\mu+(e_iB(\mub\Delta))^\mu =(\delta')^\mu+\delta_{\mub(\mc E_i^\mu)}. $$
The same translation describes raising by $\mub(\mc E_i^\mu)$
on the mutable QP, since its ordinary part vanishes at this seed.
Restriction and raising commute with mutation, so applying
$\mub^{-1}$ proves the claim.
\end{proof}

Put $\eta_\mu=r_{\mc E_i^\mu}(\delta_\mu)$. Under the hypotheses
of Lemma~\ref{L:res2mu}, Theorem~\ref{T:rlrigid} gives
\begin{equation}\label{eq:local-boundary-raising}
 r_i(\delta)=\delta+\epc_i-\gamma B_\Delta,
 \qquad \gamma=\rank_{J_\mu}(\eta_\mu,\mc E_i^\mu).
\end{equation}
Indeed, the raised weight is $\mu$-supported by
Corollary~\ref{C:pm1}. Let $\mc R$ be general of this weight.
Its ordinary part is supported on mutable vertices, so maps from
that part to $E_i$ factor through $\iota_!E_i^\mu$, by
\eqref{eq:exactEi}. Lemmas~\ref{L:res2mu}
and~\ref{L:musupp-param} identify its mutable restriction with a
general decorated representation of weight $\eta_\mu$. Thus the
general rank in the full formula is this $\gamma$.

Set
$$ a_i^{\mathbf E^\mu}(\delta_\mu) =\e_{J_\mu}(\delta_\mu,\mc E_i^\mu) +\e_{J_\mu}(\mc E_i^\mu,\delta_\mu). $$
For $\delta=\iota_{\mathbf E^\mu}(\delta_\mu)-\mathbf a$,
Lemma~\ref{L:E-esym} gives
\begin{equation}\label{eq:rho-boundary-param} \rho_i(\delta)=a_i^{\mathbf E^\mu}(\delta_\mu)+\mathbf a(i).\end{equation}
The symmetrized pairing is invariant under $\tau_\mu$, by
Lemma~\ref{L:H2E}, so
\begin{equation}\label{eq:a-taumu} a_i^{\tau_\mu\mathbf E^\mu}(\tau_\mu\delta_\mu) =a_i^{\mathbf E^\mu}(\delta_\mu).\end{equation}

\begin{lemma}\label{L:tau-E-change}
Work over the mutable Jacobian algebra, and write
$\tau=\tau_\mu$, $B=B(\Delta^\mu)$. For decorated representations
$\mathcal V,\mathcal X$, write $d_V=\dv V$,
$\delta_V=\delta_{\mathcal V}$, and similarly for $\mathcal X$.
Then
\begin{align*}
 \e(\tau\mathcal V,\tau\mathcal X)-\e(\mathcal V,\mathcal X)
 &=-\delta_X(d_{\tau V})-\delta_{\tau V}(d_{\tau X}),\\
 \delta_{\tau V}+\delta_V&=-d_{\tau V}B.
\end{align*}
\end{lemma}
\begin{proof}
Lemma~\ref{L:H2E} and \eqref{eq:heform} give
\begin{align*}
 \e(\mathcal V,\mathcal X)
 &=\hom(\mathcal X,\tau\mathcal V)\\
 &=\hom(\tau\mathcal V,\tau\mathcal X)+\delta_X(d_{\tau V})\\
 &=\e(\tau\mathcal V,\tau\mathcal X)
   +\delta_{\tau V}(d_{\tau X})+\delta_X(d_{\tau V}).
\end{align*}
The second identity follows from
$\check\delta_{\tau V}=-\delta_V$ and \eqref{eq:delta2dual}.
\end{proof}

\begin{lemma}\label{L:tau-rank-pairing}
Keep the preceding notation. Let $\mathcal E$ be an
extended-reachable decorated representation of weight $\epsilon$,
and put $\eta=r_\epsilon(\delta)$. For a decorated representation
$\mathcal V$, set
$$
 \mathfrak D_{\mathcal V}(\delta,\epsilon)
 =\e(\mathcal V,\delta)+\e(\mathcal V,\mathcal E)
  -\e(\mathcal V,\eta)
  +\delta_V\bigl(\rank(\mathcal E,\tau\delta)\bigr),
$$
using generic values for arguments given by weights. Then
$$ \mathfrak D_{\tau\mathcal V}(\tau\delta,\tau\epsilon) =\mathfrak D_{\mathcal V}(\delta,\epsilon). $$
Equivalently,
$\e(\mathcal V,\delta)-\e(\mathcal V,\eta) +\hom(\mathcal V,\mathcal E) -\delta_V\bigl(\rank(\eta,\mathcal E)\bigr)$
is $\tau$-invariant.
\end{lemma}
\begin{proof}
Use the generic exact sequence of Theorem~\ref{T:rlrigid}, with
$\mathcal M,\mathcal R$ of weights $\delta,\eta$.
Extended reachability allows the pair and the ranks in this sequence
to be chosen general before and after $\tau$, by
\cite[Theorem~5.16 and Corollary~5.25]{Fg}. We also impose the
finitely many open conditions for the pairings with
$\mathcal V$ and $\tau\mathcal V$. Put
$$ \alpha=\rank(E\to\tauh\mathcal M),\qquad \alpha'=\rank(\tauh\mathcal E\to\tauh^2\mathcal M). $$
The index formula and exactness give
$$ \eta=\delta+\epsilon+\alpha B,\qquad d_{\tau M}+d_{\tau E}-d_{\tau R}=\alpha+\alpha'. $$
Write $\nu=\delta_V$, $\nu'=\delta_{\tau V}$.
By Lemma~\ref{L:tau-E-change}, the change in $\mathfrak D$ is
\begin{align*}
 &-(\delta+\epsilon-\eta)(d_{\tau V})
 -\nu'(d_{\tau M}+d_{\tau E}-d_{\tau R})
 +\nu'(\alpha')-\nu(\alpha)\\
 &\qquad=(\alpha B)(d_{\tau V})-(\nu'+\nu)(\alpha)\\
 &\qquad=(\alpha B)(d_{\tau V})+(d_{\tau V}B)(\alpha)=0,
\end{align*}
by skew-symmetry. Finally, $d_E=\rank(R,E)+\alpha$, so
\eqref{eq:heform} gives the alternative expression.
\end{proof}

\begin{definition}\label{D:Phi-i}
For $\delta=\iota_{\mathbf E^\mu}(\delta_\mu)-\mathbf a$ with
$\rho_i(\delta)>0$, define
$\Phi_i^{\mathbf E^\mu}(\delta)\in\mathbb Z_{\ge0}^{\Delta_0^{\op{fr}}}$
by
\begin{equation}\label{eq:Phi-def}
 r_i(\delta)
 =\iota_{\mathbf E^\mu}\bigl(r_{\mc E_i^\mu}(\delta_\mu)\bigr)
  -\Phi_i^{\mathbf E^\mu}(\delta).
\end{equation}
Lemmas~\ref{L:res2mu} and~\ref{L:musupp-param} give its existence
and uniqueness.
\end{definition}

\begin{lemma}\label{L:Phi-taumu}
If $\rho_i(\delta)>0$, then
$$ \Phi_i^{\tau_\mu\mathbf E^\mu}(\rtau\delta) =\Phi_i^{\mathbf E^\mu}(\delta). $$
\end{lemma}
\begin{proof}
Write $\delta=\iota_{\mathbf E^\mu}(\delta_\mu)-\mathbf a$, and
let $\mathcal M,\mathcal R$ be general of weights
$\delta_\mu,r_{\mc E_i^\mu}(\delta_\mu)$.
Substitute \eqref{eq:local-boundary-raising} into \eqref{eq:Phi-def}.
The $j$-th frozen column is $-({\ep_j^\mu})^{\mathsf T}$, so
Lemma~\ref{L:epci} gives
$$
\begin{aligned}
 \Phi_i^{\mathbf E^\mu}(\delta)(j)-\mathbf a(j)+\delta_{ij}
 &=\e(\mc E_j^\mu,\mathcal M)-\e(\mc E_j^\mu,\mathcal R)+\hom(\mc E_j^\mu,\mc E_i^\mu)
 -\ep_j^\mu\bigl(\rank(\mathcal R,\mc E_i^\mu)\bigr).
\end{aligned}
$$
The right side is invariant by Lemma~\ref{L:tau-rank-pairing}:
reachability of $i$ makes $\mc E_i^\mu$ extended-reachable.
The map $\rtau$ fixes $\mathbf a$, and
\eqref{eq:rho-boundary-param}--\eqref{eq:a-taumu} preserve positivity
of $\rho_i$. The equality follows for each $j$.
\end{proof}

\begin{lemma}\label{L:taur}
For $i\in I$ and $\delta\in\trop(\Delta,\S)$,
$$ \rho_i^\tau(\rtau\delta)=\rho_i(\delta),\qquad \rtau r_i(\delta)=r_i^\tau(\rtau\delta). $$
\end{lemma}
\begin{proof}
Equations~\eqref{eq:rho-boundary-param}--\eqref{eq:a-taumu} prove
the first identity. If $\rho_i(\delta)=0$, both raising operators
give $0_{\mathrm{cr}}$. Otherwise put
$\eta_\mu=r_{\mc E_i^\mu}(\delta_\mu)$. By \eqref{eq:Phi-def},
$$ \rtau r_i(\delta) =\iota_{\tau_\mu\mathbf E^\mu}(\tau_\mu\eta_\mu) -\Phi_i^{\mathbf E^\mu}(\delta). $$
Theorem~\ref{T:rle} gives
$\tau_\mu\eta_\mu=r_{\tau_\mu\mc E_i^\mu}(\tau_\mu\delta_\mu)$,
and Lemma~\ref{L:Phi-taumu} identifies the last term with
$\Phi_i^{\tau_\mu\mathbf E^\mu}(\rtau\delta)$.
The right side is therefore $r_i^\tau(\rtau\delta)$.
\end{proof}

\begin{lemma}\label{L:relative-grading}
Let $G_\mu$ and $H_\tau$ have rows $\ep_j^\mu$ and
$d_{\tau_\mu\mc E_j^\mu}$, respectively, indexed by the frozen
vertices. Let $G_\mu^\tau$ be the weight matrix associated with
$\tau_\mu\mathbf E^\mu$, and put
$$
 S_\tau=\begin{pmatrix}-\Id&H_\tau^{\mathsf T}\\0&\Id\end{pmatrix}.
$$
Then $S_\tau^2=\Id$ and $B_{\Delta^\tau}S_\tau=-B_\Delta$, with
$$  (\rtau\delta)S_\tau=\delta+d_{\tau_\mu M}B_\Delta,\qquad
 \check\epsilon_i^\tau S_\tau=\check\epsilon_i-d_{E_i^\mu}B_\Delta.
$$
Here $\delta=\iota_{\mathbf E^\mu}(\delta_\mu)-\mathbf a$, and
$\mathcal M$ is general of mutable weight $\delta_\mu$.
For any column $w$ of the compatible grading,
put $w^\tau=S_\tau w$. Then $w^\tau$ is
compatible and satisfies $(\rtau\delta)w^\tau=\delta w$.
This correspondence preserves adaptation, integrality, and the independence condition \eqref{eq:existwt}.
\end{lemma}
\begin{proof}
Write $B=B(\Delta^\mu)$. Lemma~\ref{L:tau-E-change} gives
$G_\mu^\tau=-G_\mu-H_\tau B$. Together with
$B_\Delta=(B,-G_\mu^{\mathsf T})$, this proves the first identity;
$S_\tau^2=\Id$ is immediate. Apply both formulas of
Lemma~\ref{L:tau-E-change} to each pair $\mc E_j^\mu,\mathcal M$
in \eqref{eq:iotaV} to obtain the second identity.

For the third identity, Lemma~\ref{L:epci} gives the common mutable
part $\ep_i^\mu$. The two $j$-th frozen coordinates agree, since
\begin{align*}
 \delta_{ij}-\hom(\tau\mc E_j^\mu,\tau\mc E_i^\mu)
 -\ep_i^\mu(d_{\tau\mc E_j^\mu})
 &=\delta_{ij}-\e(\mc E_j^\mu,\mc E_i^\mu),\\
 \delta_{ij}-\hom(\mc E_j^\mu,\mc E_i^\mu)
 +\ep_j^\mu(d_{E_i^\mu})
 &=\delta_{ij}-\e(\mc E_j^\mu,\mc E_i^\mu),
\end{align*}
by Lemma~\ref{L:H2E} and \eqref{eq:heform}; here $\tau=\tau_\mu$.
Multiplying the three identities by $w$ proves compatibility,
weight preservation, and adaptation. Integrality follows from the
integer entries of $S_\tau$. Finally, $S_\tau$ identifies the
exchange-row quotients and their boundary classes, so it preserves
\eqref{eq:existwt}.
\end{proof}

\begin{definition}\label{D:taumu-enhanced-crystal-cluster}
Let $\wtd{\mf T}$ be the rooted tree of reduced paths from a
fixed reference QP, with edges given by mutations and permitted
relative $\tau$-translations. Each vertex
$t$ carries an ice QP $(\Delta_t,\S_t)$ and its set
$\mc B_t=\trop(\Delta_t,\S_t)$ of $\mu$-supported weights.
As in $\mf T$, paths with isomorphic terminal QPs remain distinct.

Each edge $t\to t'$ induces a bijection
$\varphi_{t,t'}:\mc B_t\to\mc B_{t'}$, given by tropical
mutation or $\rtau^{\pm1}$. A
\emph{$\tau$-enhanced crystal cluster structure} is a family of
crystal structures on the $\mc B_t$ for which every
$\varphi_{t,t'}$ is a crystal isomorphism.
\end{definition}

\begin{theorem}\label{T:taumu-enhanced-upper}
Fix a compatible adapted grading at the root as in
Theorem~\ref{T:upper}, with coroot-coordinate weights or a chosen
full completion under \eqref{eq:existwt}. Transport each component
$w$ of the grading along $\wtd{\mf T}$ by mutation and by
$w^\tau=S_\tau w$ on relative edges, using the inverse map
on reversed edges.
Then the boundary crystals form a $\tau$-enhanced upper seminormal
crystal cluster structure. If the crystal at the root is seminormal,
so is the crystal at every vertex.
\end{theorem}
\begin{proof}
Theorem~\ref{T:upper} treats mutation edges. On a relative edge,
Lemmas~\ref{L:tautrop}, \ref{L:taur}, and~\ref{L:relative-grading}
give a bijection preserving $r_i$, $\rho_i$, and the chosen
weight map. It also preserves the inverse partial maps $l_i$ and
$\lambda_i=\rho_i+\wt_i$. Thus every edge map is a crystal
isomorphism. Their compositions preserve seminormality.
\end{proof}

If Theorem~\ref{T:crystal} applies at a vertex and the transported
crystal is seminormal, its natural seminormal structure agrees with
the transported one in coroot coordinates.

\begin{remark}
Suppose that a permutation $\sigma$ of the frozen vertices
preserves $I$ and satisfies
$\tau_\mu\mc E_j^\mu\cong\mc E_{\sigma(j)}^\mu$ for all frozen
$j$. If the translated grading equals the original grading after
this relabeling, Lemma~\ref{L:boundary-comparison} identifies the
sets of $\mu$-supported weights and their crystal structures. The resulting self-map
$\rtau$ satisfies
\begin{align*}
 \rtau r_i=r_{\sigma(i)}\rtau,\quad &
 \rtau l_i=l_{\sigma(i)}\rtau,\\
 \rho_{\sigma(i)}(\rtau\delta)=\rho_i(\delta),\quad &
 \lambda_{\sigma(i)}(\rtau\delta)=\lambda_i(\delta).
\end{align*}
\end{remark}

\section{Upper Cluster Algebras and their Generic Bases} \label{S:uca}
\subsection{Upper Cluster Algebras} \label{ss:uca}
In this subsection, we will briefly review the skew-symmetric upper cluster algebras.
Readers can find the definition of general upper cluster algebras in the original paper \cite{BFZ}.
\begin{definition} \label{D:seeds} 
	Let $\mc{F}$ be a field containing $\k$.
	A {\em seed} in $\mc{F}$ is a pair $(\Delta,\b{x})$ consisting of an ice quiver $\Delta$ together with a collection $\b{x}=\{x_1,x_2,\dots,x_q\}$, called an {\em extended cluster}, consisting of algebraically independent (over $\k$) elements of $\mc{F}$, one for each vertex of $\Delta$.
	The subset $\b{x}_\mu$ of $\b{x}$ associated with the mutable vertices is called {\em cluster variables}; they form a {\em cluster}.
	The subset $\b{x}_{\op{fr}}$ associated with the frozen vertices is called
	{\em frozen variables}, or {\em coefficient variables}.
	
	A {\em seed mutation} $\mu_u$ transforms $(\Delta,\b{x})$ into the seed $(\Delta',\b{x}')=\mu_u(\Delta,\b{x})$ defined as follows.
	The new quiver is $\Delta'=\mu_u(\Delta)$.
	The new extended cluster is
	$\b{x}'=\b{x}\cup\{x_{u}'\}\setminus\{x_u\}$
	where the new cluster variable $x_u'$ replacing $x_u$ is determined by the {\em exchange relation}
	\begin{equation}\label{eq:exrel} x_u\,x_u' = \prod_{v\rightarrow u} x_v + \prod_{u\rightarrow w} x_w.\end{equation}
\end{definition}

\noindent 
Two seeds $(\Delta,\b{x})$ and $(\Delta',\b{x}')$ that can be obtained from each other by a sequence of mutations are called {\em mutation-equivalent}, denoted by $(\Delta,\b{x})\sim (\Delta',\b{x}')$.

Let $\mc{L}(\b{x}):=\k[\b{x}_\mu^{\pm 1}, \b{x}_{\op{fr}}]$ be the subalgebra of the Laurent polynomial algebra $\k[\b{x}^{\pm 1}]$, which is polynomial in the frozen variables.
\begin{definition}[{\em Upper Cluster Algebra}]
	The upper cluster algebra with seed $(\Delta,\b{x})$ is
	$$ \br{\mc{C}}(\Delta,\b{x}):=\bigcap_{(\Delta',\b{x}') \sim (\Delta,\b{x})}\mc{L}(\b{x}'). $$
\end{definition}
\noindent Note that by the Laurent Phenomenon \cite{FZ1} the upper cluster algebra $\br{\mc{C}}(\Delta,\b{x})$ contains the corresponding cluster algebra $\mc{C}(\Delta,\b{x})$, generated by all mutable cluster variables over $\k[\b{x}_{\op{fr}}]$.

An \emph{initial-seed mutation} $\mu_u(f)$ denotes the
expression of $f$ in the coordinates of the adjacent seed
$\b{x}'$.
Although $\uca(\Delta',\b{x}')$ is equal to $\uca(\Delta,\b{x})$, we will write $\mu_u(f)\in \uca(\Delta',\b{x}')$ to indicate our choice of the cluster to express $\mu_u(f)$.

In general, there may be infinitely many seeds mutation equivalent to $(\Delta,\b{x})$.
So the following theorem is useful to test the membership in an upper cluster algebra.
\begin{definition}[\cite{BFZ}]
Let $\b{x}_u$, for $u\in\Delta_0^\mu$, be the \emph{adjacent} cluster
obtained from $\b{x}$ by mutation at $u$. The \emph{upper bound} is
$$ \mc U(\Delta,\b x) :=\mc L(\b x)\cap\bigcap_{u\in\Delta_0^\mu}\mc L(\b x_u). $$
Thus the intersection includes the original seed and all its adjacent
seeds. If there are no mutable vertices, it is $\mc L(\b x)$.
\end{definition}

\begin{theorem}[{\cite{BFZ}, \cite[Theorem~3.11]{GSV}}]\label{T:bounds}
Suppose that $B_\Delta$ has full rank, namely
$\operatorname{rank}B_\Delta=|\Delta_0^\mu|$. If
$(\Delta,\b x)\sim(\Delta',\b x')$, then
$$ \mc U(\Delta,\b x)=\mc U(\Delta',\b x') =\uca(\Delta,\b x). $$
Here every Laurent ring is polynomial in the frozen variables, as in
the definition of $\mc L$.
\end{theorem}

\subsection{Generic Bases} \label{ss:genericB}
Let $\b{y}=(y_v)_{v\in\Delta_0}$ be independent variables, and
write $\check{\b{y}}=(\check y_v)_{v\in\Delta_0}$, where
$\check y_v=y_v^{-1}$. Recall from \cite{DWZ2} that the
{\em $F$-polynomial} of a representation $M$ and its dual are
\begin{align*}
 F_M(\b{y}) &= \sum_{\gamma}\chi(\op{Gr}_\gamma(M))\b{y}^{\gamma},\\*
 \check F_M(\check{\b{y}}) &= \sum_{\gamma}\chi(\op{Gr}^{\gamma}(M))\check{\b{y}}^{\gamma},
\end{align*}
where $\op{Gr}_\gamma(M)$ and $\op{Gr}^\gamma(M)$ parametrize
subrepresentations and quotient representations of dimension $\gamma$,
respectively. In particular, both generating functions are polynomials,
and
$$ \check F_M(\check{\b{y}})=\check{\b{y}}^{\dv M}F_M(\b{y}). $$
For $\mu$-supported $M$, we substitute
$$ y_u=\b{x}^{b_u},\qquad \check y_u=\b{x}^{-b_u} \qquad(u\in\Delta_0^\mu), $$
where $b_u$ is the $u$-th row of $B_\Delta$.

We can choose a nonempty open subset $U\subset \PHom(\delta)$ such that $\coker(d)$ has a constant $F$-polynomial for any $d\in U$.
We denote by $\coker(\delta)$ the cokernel of a general presentation in $\PHom(\delta)$,
and by $\ker(\dtc)$ the kernel of a general presentation in $\IHom(\dtc)$.
\begin{definition}[\cite{P}]
	We define the {\em generic character} $C_{\gen}:\trop(\Delta,\mc{S})\to \mb{Z}(\b{x})$~by
	\begin{equation}\label{eq:genCC} C_{\gen}(\delta)=\b{x}^{-\delta} \check{F}_{\coker(\delta)}(\check{\b{y}}).\end{equation}
By Theorem \ref{T:genpi}.(2) and \eqref{eq:delta2dual}, this is the same as 
	$$ \check{C}_{\gen}(\dtc)=\b{x}^{-\dtc} {F}_{\ker(\dtc)}(\b{y}). $$
The vector $-\dtc$ is the representation-theoretic
$\g$-vector in $\mathbb Z^{\Delta_0}$, including the frozen coordinates. 
For reachable cluster variables, its restriction to mutable part agrees with the $\g$-vector of \cite{FZ4};
the vector $-\delta$ is also called the dual $\g$-vector or the degree of $C_{\gen}(\delta)$.
\end{definition}
\noindent If $\delta$ is obtained from a negative
simple by mutations, $C_{\gen}(\delta)$ is the corresponding
variable in an extended cluster \cite{DWZ2}. We call
$C_{\gen}(\delta)$ a \emph{generalized cluster variable} if
$\delta\in\trop(\Delta,\S)$ is indecomposable and rigid, and a
\emph{generalized cluster monomial} if $\delta$ is rigid.

\begin{lemma}[{\cite{P}}, see also \cite{Fs1}] \label{L:Cmu} The generic character commutes with the initial-seed mutations: $$ \mu_u (C_{\gen}(\delta)) = C_{\gen}(\mu_u(\delta)). $$	
In particular, $C_{\gen}(\delta)\in \uca(\Delta, \b{x})$.
\end{lemma}

\noindent If the generic character maps the set of $\mu$-supported $\delta$-vectors to a basis of $\uca(\Delta,\b{x})$, then such a basis is called the {\em generic basis} of $\uca(\Delta,\b{x})$.
It was proved \cite{FW, Qb} that the generic bases exist for a large class of upper cluster algebras.

\begin{definition}[\cite{Qt}] 
An un-iced quiver $\Delta$ is \emph{reachable} if a single
sequence of the tropical mutations \eqref{eq:mug}, applied to every row
of $-I_{\Delta_0}$, gives $I_{\Delta_0}$ up to permutation.
An ice quiver is reachable if its mutable part is.
\end{definition}

For full-rank reachable seeds,
\cite[Theorem~5.2.2]{Qb} gives a basis by generic characters after
inverting the frozen variables. The theta-function regularity results
of \cite[Theorem~3.8]{CMMM} apply under their stated valuative
hypotheses. If the regular theta basis is indexed precisely by the
$\mu$-supported presentation weights, and if the comparison in
\cite[Lemma~5.1.1]{Qb} restricts to the algebra with polynomial frozen
variables, it yields the corresponding generic basis there. The
restriction requires control of frozen-variable denominators (compare
\cite[Proposition~4.20]{Fh}). 
The following remark compares the generic-character span with the
middle cluster algebra.

\begin{remark}[Generic characters and the middle cluster algebra]
\label{r:middle}
Set
$$
\mathcal V_{\gen}(\Delta,\S)
=\operatorname{span}_{\mathbb C}
\{C_{\gen}(\delta):\delta\in\trop(\Delta,\S)\}
\subseteq\uca(\Delta)_{\mathbb C}.
$$
For full-rank $B_\Delta$, the generic characters are linearly
independent \cite{P}. The middle cluster algebra
$\operatorname{mid}(\mathcal A)$ of \cite{GHKK} is spanned by regular
theta functions. In many familiar classes these and the generic
characters have the same linear span (cf.~\cite{Qb}); more generally,
we expect
$$ \mathcal V_{\gen}(\Delta,\S) =\operatorname{mid}(\mathcal A)_{\mathbb C}. $$
The individual generic characters and theta functions need not agree.
The further equality
$$ \operatorname{mid}(\mathcal A)_{\mathbb C} =\uca(\Delta)_{\mathbb C} $$
is a separate question. 
The generic-basis hypothesis of Theorem~\ref{T:genericCR} concerns this latter equality, 
once the identification of the two spans is established.
\end{remark}

\subsection{Tropical $x$-polynomials and their Pairing}
Recall the following mutation rule for a vector $d\in \mb{Z}^{\Delta_0}$, which is obtained by tropicalizing the mutation of $x$-variables \eqref{eq:exrel}
\begin{equation}\label{eq:muatrop}
 d'(v)=\begin{cases}
 -d(u)+\max\{d\cdot [b_u]_+, d\cdot [-b_u]_+\},&v=u,\\
 d(v),&v\ne u.
 \end{cases}
\end{equation}
A compatible family $(d_t)$ is called a \emph{tropical $\mc A$-point}, represented by
its vector in any chosen seed.

\begin{lemma}\label{L:dvtropx} The dimension vector of a boundary representation is a tropical $\mc{A}$-point.
\end{lemma}
\begin{proof}
For each mutable vertex $u$, Lemma~\ref{L:gdmu} gives
$$
\begin{aligned}
 d'(u)&=d\cdot [-b_u]_+-d(u)+\betac_-(u)\\
      &=d\cdot [b_u]_+-d(u)+\betac_+(u).
\end{aligned}
$$
Boundary representations are $\E$-rigid by Lemma~\ref{L:HEres}, so
their minimal injective terms have no common summand. Hence at least
one of $\betac_-(u),\betac_+(u)$ is zero. These identities
then give \eqref{eq:muatrop}; all other coordinates are unchanged.
\end{proof}

\begin{lemma}\label{L:mudB}
Let $d$ be a tropical $\mc{A}$-point and put
$\delta=dB_\Delta^{\mathsf T}\in\mathbb Z^{\Delta_0^\mu}$.
Then for every $t\in\mf T$,
$$ d_tB_{\Delta_t}^{\mathsf T}=\delta_t $$
\end{lemma}

\begin{proof}
Since $\check y_u=\mathbf x^{-b_u}$, its tropical value at $d$ is
$-\delta(u)$. Tropicalizing the dual $\check y$-seed mutation rule
gives \eqref{eq:mug}
\cite{FGc}, \cite[(2.3)]{FZ4}. Thus the two families agree at every seed.
\end{proof}

\begin{definition} For $f(\b{x}) = \sum_{\eta} c_{\eta}\b{x}^{-\eta} \in \uca(\Delta)$, we define its tropical $x$-polynomial as 
	$$ f^{\trop} (d) = \max_{\eta:\,c_\eta\ne0} (-\eta(d)) \quad \text{for $d\in\mb{Z}^{\Delta_0}$.} $$
\end{definition}
We set $0^{\trop}(d)=-\infty$. The tropicalization is the
support function of the Newton polytope of $f$, with arbitrary
nonzero coefficients.

\begin{lemma}\label{L:tropxmu} Let $d$ be a tropical $\mc{A}$-point and $f\in \uca(\Delta)$. Then the pairing $f^{\trop} (d)$ is mutation-invariant.
\end{lemma}
\begin{proof}
It suffices to consider one mutation and $f\ne0$. Write
$x_u'=P/x_u$, where $P$ is independent of $x_u$, and
$f=\sum_r a_r x_u^r$, with each $a_r$ Laurent in the remaining
variables. Put $D=\deg_d P$. Then $d'(u)=D-d(u)$.

Since $f$ is Laurent in both clusters, each $a_rP^r$ is Laurent.
Weighted degree is additive on nonzero products, and distinct values
of $r$ yield distinct powers of $x_u'$. Consequently,
$$
\begin{aligned}
(\mu_u f)^{\trop}(d')
&=\max_{r:a_r\ne0}\{\deg_d a_r+rD-rd'(u)\}\\
&=\max_{r:a_r\ne0}\{\deg_d a_r+rd(u)\}
=f^{\trop}(d).
\end{aligned}
$$
The zero case follows from the convention above.
\end{proof}

\section{Algebraic Lifting of Crystal Structures} \label{S:lifting}
We now lift the tropical crystal operators to derivations of the
upper cluster algebra and study the Lie-theoretic structures they
produce. Throughout this section, $(\Delta,\S)$ is
frozen-Jacobi-finite, $I$ consists of reachable frozen vertices,
and $I^\star$ of dually reachable ones. As before, the frozen
variables are not inverted, and the representation-theoretic
conventions of Section~\ref{ss:conventions} apply. We work over
$\k=\mathbb Z$, extending scalars when indicated.

We first construct locally nilpotent root derivations and establish
their Serre and mixed relations. For the $\tau$-exact pairs of
Theorem~\ref{T:crystal}, these derivations combine into an integrable
action of the derived Kac--Moody algebra. This forces the associated
Cartan matrix to be positive semidefinite. Neither full rank nor
the existence of a generic basis is needed for the action or this
Cartan restriction.
We then turn to perfect bases and normality, and finally to string
reflections and algebraic Tits lifts.

\subsection{Root Derivations}\label{ss:der}
At an optimized seed, a root derivation has a simple
differential form. We hope to extend it to the upper
cluster algebra without imposing a rank assumption.
A rational derivation is determined by its values on the variables of any seed,
by the chain rule. We call a seed $i$-optimized if $E_i=S_i$,
and dually $i$-optimized if $E_i^\star=S_i$. In the respective
optimized coordinates, define
\begin{equation}\label{eq:Ri}
 R_i=\left(\prod_{u\to i}x_u\right)\partial_{x_i},\qquad
 R_j^\star=\left(\prod_{j\to u}x_u\right)\partial_{x_j}.
\end{equation}
Lemma~\ref{L:boundary-propagation} shows that the coefficients are
monomials in the mutable variables. Equivalently,
$$ R_i(x_i)=C_{\gen}(r_i(-e_i)),\qquad
 R_j^\star(x_j)=\check C_{\gen}(\rc_j^\star(-e_j)).
$$

\begin{proposition}\label{P:root-derivations}
The rational derivations $R_i$, $i\in I$, and $R_j^\star$,
$j\in I^\star$, are independent of the chosen optimized seeds.
They preserve $\uca(\Delta)$ and are locally nilpotent on it.
\end{proposition}
\begin{proof}
At two $i$-optimized seeds, the tropical point $\dv E_i$
is $e_i$. By Lemmas~\ref{L:dvtropx} and~\ref{L:tropxmu},
mutation invariance of tropical pairing shows that every mutable
variable of either seed has $x_i$-degree zero when expressed in the
other seed. These variables belong to $\uca(\Delta)$, so their
Laurent expansions are polynomial in the frozen coordinate $x_i$.
Having degree zero in $x_i$, they are therefore independent of $x_i$.
The identity $R_i(x_i)=C_{\gen}(r_i(-e_i))$, together with
Theorem~\ref{T:rle} and Lemma~\ref{L:Cmu}, identifies the
coefficients of $\partial_{x_i}$ in the two seeds.
Thus the derivations agree on the other seed's extended cluster. The starred argument is dual.

First suppose that $B_\Delta$ has full rank. The derivation
$R_i$ preserves the Laurent ring of an $i$-optimized seed.
It also preserves every adjacent Laurent ring: if a mutable $u$
is not adjacent to $i$, it kills $x_u'$. Otherwise there are
$m>0$ arrows $u\to i$, and we may write
$$ x_ux_u'=A x_i^m+B,\qquad
 R_i(x_u')=mA x_i^{m-1}\frac{R_i(x_i)}{x_u},
$$
where $A,B$ are monomials independent of $x_i,x_u$.
Since $x_u^m$ divides $R_i(x_i)$, substituting
$x_u=(A x_i^m+B)/x_u'$ puts this expression and $R_i(x_i)$
in the adjacent Laurent ring, with frozen variables still polynomial.
All other unchanged coordinates are
killed, and $R_i(z^{-1})=-z^{-2}R_i(z)$ handles mutable
inverses. Theorem~\ref{T:bounds} identifies the intersection of
these rings with $\uca(\Delta)$. The proof for $R_j^\star$
is dual.

To remove the rank assumption, we adjoin coefficients and specialize the images of lifted seed variables.
For arbitrary rank, append a frozen source $z_u\to u$ for each
mutable vertex $u$ at the distinguished reference seed, which has no
arrows between frozen vertices, leaving the potential
unchanged. The resulting exchange matrix has a signed identity
block and hence full rank. Its mutable and old one-frozen
restrictions are unchanged; this is the simultaneous projective
extension of Construction~\ref{cons:iQP}. The original full
Jacobian algebra is finite-dimensional by Lemma~\ref{L:ice-finite}.
The old projectives are unchanged, while the projective at $z_u$
has radical $P_u$, so the extended full Jacobian algebra is also
finite-dimensional.

Recall that mutation commutes with these restrictions. The extracted mutable
presentations of the old boundary modules, and therefore their
rigidity and optimizing sequences, are unchanged. At an old
optimized seed, Lemma~\ref{L:boundary-propagation} again makes
the coefficient of $\widehat R_i$ or $\widehat R_j^\star$ a
monomial in the mutable variables; both derivations kill the new frozen variables.
Consequently specialization $z_u=1$ intertwines the extended
derivations with the original ones. For any seed $t$, the
full-rank case gives
$$
 \widehat R_i(\widehat x_k^t)
 \in\uca(\widehat\Delta)\subseteq\widehat{\mc L}_t
 \quad\Longrightarrow\quad
 R_i(x_k^t)\in\mc L_t.
$$
Specialization is defined here because the expression is Laurent
in mutable variables and polynomial in frozen variables, and each
specialized mutable variable is nonzero. The chain rule now proves
that $R_i$ preserves every $\mc L_t$, including mutable inverses,
and hence their intersection $\uca(\Delta)$. The starred proof is identical.

Finally, elements of $\uca(\Delta)$ are polynomial in the
optimized frozen coordinate, whereas the corresponding coefficient $R_i(x_i)$ or $R_j^\star(x_j)$
is independent of that coordinate. Repeated
differentiation proves local nilpotence.
\end{proof}

For a $\tau$-exact pair $(i,\ibar)$ as in Theorem~\ref{T:crystal},
we put $L_i=R_{\ibar}^\star$,and define the simple Cartan derivations by
$$ H_i(x_k)=\wt_i(-e_k)x_k.$$
Compatibility of the grading makes them independent of the seed.
They act on the integral algebra for an integral grading, and on
$\uca(\Delta)\otimes_{\mathbb Z}\mathbb Q$ for a rational grading.
If \eqref{eq:existwt} holds, choose the integral auxiliary forms of
Lemma~\ref{L:wtc} and define $H_a$ by the same formula for $a\in K$.

\subsection{Serre and Mixed Relations}
In an optimized seed, a root derivation differentiates in one
frozen coordinate. The next degree estimates control how much a
second derivation can increase that coordinate's degree. This is
exactly the information needed to bound iterated commutators and
obtain the Serre relations.

For distinct reachable vertices we put
$c_{ij}=-(a_{ij}+a_{ji})$. For the mixed relations we use the
pairing over the full Jacobian algebra
$$ c_{ij}^\star=-\e(E_j^\star,E_i)=-\ec(E_j^\star,E_i).$$
The second equality follows from \eqref{eq:heform},
\eqref{eq:hecform}, skew-symmetry of $B$, and
$\epc_j^\star(\dv E_i)=\ep_i(\dv E_j^\star)$ from
Lemma~\ref{L:E^mu}. In particular, the same mixed pairing occurs
in the opposite quiver.

\begin{lemma}\label{L:degi}
Suppose that $E_i=S_i$, let $f\in\uca(\Delta)$, and use
$\deg_{x_i}0=-\infty$.
\begin{enumerate}
\item For $j\in I\setminus\{i\}$ and $a=-c_{ij}$, we have
$$ \deg_{x_i}R_j(f)\le\deg_{x_i}f+a,\qquad \deg_{x_i}R_j(x_j)=a.$$
Moreover, $R_j(x_k)=0$ if and only if $E_j(k)=0$, and
$\deg_{x_i}R_j(x_u)\le a-1$ whenever $u$ is mutable and $u\to i$.
\item For $j\in I^\star$ and $b=-c_{ij}^\star$, we have
$$ \deg_{x_i}R_j^\star(f)\le\deg_{x_i}f+b,\qquad
 R_j^\star(x_k)=0\ \Longleftrightarrow\ E_j^\star(k)=0.
$$
If $j\ne i$, then $\deg_{x_i}R_j^\star(x_j)=b$.
If $j=i$ and $E_i$ is nonsimple at a dually $i$-optimized
seed, then $\deg_{x_i}R_i^\star(x_i)=b+1$.
\end{enumerate}
\end{lemma}
\begin{proof}
We first establish the general degree bounds and vanishing
criteria by tropical pairing, followed by the strict unstarred
bound needed for the Serre relations.

In coordinates optimized for the indicated derivation, differentiation
of monomials gives
$$ R_j(f)^{\trop}\le f^{\trop}-\epc_j,\qquad
 R_j^\star(f)^{\trop}\le f^{\trop}-\ep_j^\star.
$$
Choose $\mub:t\to t'$ optimizing $j$ and put
$d=\dv E_i'=\mub(e_i)$. For the unstarred case,
$\Hom(E_i',E_j')=0$. Since $(\mc E_j^\mu)'$ is negative,
\eqref{eq:boundary-weights} gives $\epc_i'(j)=0$, and hence
$\ec(E_j',E_i')=0$. Symmetrizing Lemma~\ref{L:H2E} yields
$$ -\epc_j'(d)=\ec(E_i',E_j')=a_{ij}+a_{ji}=a.$$
For the starred case, choose instead $(E_j^\star)'=S_j$.
Lemma~\ref{L:HomEinv} gives
$$ -(\ep_j^\star)'(d) =b-\hom(S_j,E_i')\le b.$$
Here the Hom term vanishes if $j\ne i$, and also if $j=i$
and $E_i'\ne S_i$, by Proposition~\ref{P:Brep}.
Mutation invariance of tropical pairing proves both degree bounds.
The frozen variable $x_j$ is unchanged under $\mub$, and its
derivative is a single monomial. Its $d$-degree is $a$ in the
first case and $d(j)-(\ep_j^\star)'(d)$ in the second.
Since $d(j)=\updelta_{ij}$, these are the asserted equalities.
Also
$$ \deg_{x_j'}\mub(x_k)=E_j(k) \quad\text{or}\quad E_j^\star(k),$$
respectively. Polynomiality in the frozen variable $x_j'$ and
characteristic zero prove the vanishing criteria.

For the strict unstarred bound, the key observation is that
the derivation kills the highest-degree Laurent monomial; all
surviving terms therefore have smaller degree.
Let $u\to i$ be mutable, choose $t'$ with $E_j'=S_j$, and write
$$ \mub(x_u)=\sum_\alpha c_\alpha
       \mathbf x^{-\dtc_u+\alpha B'},\qquad B'=B(\Delta_{t'}).
$$
For each mutable $v\to i$, write $\mathcal T_v=(T_v,T_v^-)$
for the decorated representation of $\mub(x_v)$, with injective weight
$\dtc_v$. We apply rigidity to the whole direct sum
$\mathcal T=\bigoplus_{v\to i}\mathcal T_v$, counting arrows
with multiplicity. This sum is rigid because it is obtained by
mutating the corresponding rigid direct sum of negative simple
representations. Write $T$ for its ordinary part.
The mutable restriction of Lemma~\ref{L:mudB} in dual coordinates gives
$(dB')|_{\Delta_0^\mu}=\sum_{v\to i}\dtc_v|_{\Delta_0^\mu}$.
For a nonzero subrepresentation $N\subseteq T_u$ of dimension
$\alpha$, rigidity gives $\ec(N,\mathcal T)=0$: restriction to
$N$ is surjective on the Hom spaces into the injective terms of
$\mathcal T$, hence on their cokernel. The inclusion into
$T_u\subseteq T$ gives $\hom(N,\mathcal T)>0$.
By \eqref{eq:hecform}, this is the subrepresentation form of
\cite[Lemma~5.15]{Ft}. Since $\alpha$ has zero frozen coordinates,
$$
 -\alpha B'd^{\mathsf T}
   =\sum_{v\to i}\dtc_v(\alpha)>0\qquad(\alpha\ne0).
$$
Thus $\mathbf x^{-\dtc_u}$ is the unique monomial of maximal
$d$-degree, which is zero. The ordinary part of $\mathcal T_u$
is $\mu$-supported, and $E_j'=S_j$ excludes arrows from $j$
to mutable vertices. Equation~\eqref{eq:Bettidual} consequently
gives $\dtc_u(j)=0$, so $R_j$ kills this leading monomial.
Every surviving monomial has $d$-degree at most $-1$ before
the differentiation shift, which is $a$. Tropical pairing now
gives the strict bound.
\end{proof}

We use the enveloping algebras of the integral Chevalley--Serre Lie algebras, 
with the usual Kac--Moody interpretation after scalar extension to characteristic zero.
\begin{theorem}\label{T:crystalUCA}
The derivations $R_i$, $i\in I$, and $R_j^\star$,
$j\in I^\star$, define $U(\mathfrak n_+(C_I))$- and
$U(\mathfrak n_+(C_{I^\star}))$-module algebra structures on
$\uca(\Delta)$, respectively. Here $\mathfrak n_+(C)$ is
the positive part of the Kac--Moody Lie algebra
associated with $C$.
\end{theorem}
\begin{proof}
At an $i$-optimized seed, $R_i$ differentiates only $x_i$,
so $R_i^2(x_k)=0$ for every seed variable. Thus, for $n\ge1$,
\begin{equation}\label{eq:adjoint-on-cluster}
 (\ad R_i)^n(R_j)(x_k)
 =R_i^nR_j(x_k)-nR_i^{n-1}R_j(R_i(x_k)).
\end{equation}
The same identity holds with $R_j$ replaced by any other
derivation. For $i\ne j$ and $a=-c_{ij}$, Lemma~\ref{L:degi}
gives $\deg_{x_i}R_j(x_k)\le a$ for $k\ne i$, while
$R_j(x_i)=0$ because $E_j(i)=0$. Since every factor of
$R_i(x_i)$ comes from a mutable $u\to i$, its strict bound gives
$\deg_{x_i}R_j(R_i(x_i))\le a-1$. Taking $n=a+1$ proves
$$
 (\ad R_i)^{1-c_{ij}}(R_j)=0.
$$
Checking seed variables suffices by the chain rule. Passing to
$\Delta^{\opp}$ gives the starred action.
\end{proof}
The relation is sharp: evaluation at $x_j$ gives
$(\ad R_i)^a(R_j)(x_j)=R_i^aR_j(x_j)\ne0$ by
Lemma~\ref{L:degi}.
For an integral compatible adapted grading, adjoining the simple
Cartan derivations gives the corresponding Borel action over
$\mathbb Z$; a rational grading gives it after extension to
$\mathbb Q$. The optional auxiliary Cartan derivations give
the Borel action for the full realization. The verification is
included in the proof of Theorem~\ref{T:normalUCA} below.

\begin{proposition}\label{P:mixed}\label{C:UxU}
For $i\in I$ and $j\in I^\star$, put
$a_{ij}^\star=-c_{ij}^\star=\e(E_j^\star,E_i)$.
The case $a_{ij}^\star=0$ is the commuting case; when
$a_{ij}^\star>0$, the adjoint exponent below is $a_{ij}^\star+2$. 
Uniformly, we have
$$ (\ad R_i)^{1+a_{ij}^\star+\min(a_{ij}^\star,1)}(R_j^\star)=0,\qquad
 (\ad R_j^\star)^{1+a_{ij}^\star+\min(a_{ij}^\star,1)}(R_i)=0.
$$
The two actions commute if and only if every $c_{ij}^\star$ vanishes.
Equivalently, this characterizes when they define a
$U(\mathfrak n_+(C_I))\otimes U(\mathfrak n_+(C_{I^\star}))$-module
algebra structure on $\uca(\Delta)$.
\end{proposition}
\begin{proof}
Choose an $i$-optimized seed $t$ and write $a=a_{ij}^\star$.
Here $\deg_{x_i,t}$ denotes degree in the frozen coordinate
$x_i$ in this seed. Suppose first that $i\ne j$. Lemma~\ref{L:degi} gives
$\deg_{x_i}R_j^\star(x_k)\le a$ for $k\ne i$, while
$R_j^\star(x_i)=0$ because $E_j^\star(i)=0$. Moreover,
$\deg_{x_i}R_j^\star(R_i(x_i))\le a$.
For $a>0$, \eqref{eq:adjoint-on-cluster}, with
$R_j$ replaced by $R_j^\star$ and $n=a+2$, proves the first
mixed relation.

If $a=0$, the degree bound gives $R_iR_j^\star(x_k)=0$
for every $k$. We also have $R_j^\star(R_i(x_i))=0$:
$E_i=S_i$ excludes arrows from $i$ to mutable vertices,
whereas $E_j^\star$ has frozen support only at $j$, with
simple socle $S_j$ annihilated by paths of positive length.
Thus $\gamma_i(E_j^\star)=0$, and \eqref{eq:Betti} gives
$$
 0=a=\e(E_j^\star,S_i)=\sum_{u\to i}\dim E_j^\star(u).
$$
The vanishing criterion in Lemma~\ref{L:degi} shows that
$R_j^\star$ kills every factor of $R_i(x_i)$, proving
$[R_i,R_j^\star]=0$.
Conversely, commutativity gives $R_iR_j^\star(x_j)=0$,
while $\deg_{x_i}R_j^\star(x_j)=a$. Since the coefficient
of $\partial_{x_i}$ in $R_i$ is nonzero, this forces $a=0$.

Now suppose $i=j$, and choose a dually $i$-optimized seed
$t'$. If $E_i'=S_i$, both derivations have a mutable-monomial
coefficient times $\partial_{x_i}$ in that seed, so they
commute and $a=0$. Otherwise Lemma~\ref{L:degi} gives
$$
 \deg_{x_i,t}R_i^\star(x_i)=a+1,\qquad
 \deg_{x_i,t}R_i^\star(x_k)\le a+\updelta_{ik},\qquad
 \deg_{x_i,t}R_i^\star(R_i(x_i))\le a.
$$
For $a>0$, \eqref{eq:adjoint-on-cluster} with $n=a+2$
again proves the relation. For $a=0$, the degree bounds
show that $[R_i,R_i^\star]$ kills all seed variables of $t$
except possibly $x_i$. At $t'$, $R_i^\star(x_i)$ has
$x_i$-degree zero, and the dual degree bound gives
$R_i^\star R_i(R_i^\star(x_i))=0$. Since
$(R_i^\star)^2(x_i)=0$, we obtain
$[R_i,R_i^\star](R_i^\star(x_i))=0$.
But $\deg_{x_i,t}R_i^\star(x_i)=1$, so $R_i^\star(x_i)$ is
affine-linear in $x_i$ with nonzero leading coefficient
independent of $x_i$. As the commutator kills the other
seed variables, it must also kill $x_i$. Conversely, if
$[R_i,R_i^\star]=0$, then
$R_i^2(R_i^\star(x_i))=R_i^\star(R_i^2(x_i))=0$,
so $\deg_{x_i,t}R_i^\star(x_i)\le1$ and $a=0$.

Finally, passage to the opposite quiver gives the second mixed
relation with the same $a$, by the equality of the mixed pairings.
The commuting-action assertion follows from the pairwise
equivalences just proved.
\end{proof}

\subsection{\texorpdfstring{$\tau$}{tau}-Exact Actions and Cartan Type}\label{ss:Chevalley}
The positive and starred actions have so far been constructed
separately. Under $\tau$-exactness they become the raising and
lowering halves of one Kac--Moody action. The remaining challenge is
the diagonal relation $[R_i,L_i]=H_i$.
The following character estimate will prove the diagonal Chevalley relation and, later, perfectness. 
\begin{lemma}\label{L:character-strings}
For $i\in I$ and $\delta\in\mc B$, put $n=\rho_i(\delta)$.
If $n=0$, then $R_iC_{\gen}(\delta)=0$; if $n>0$, then
\begin{equation}\label{eq:character-remainder}
 R_iC_{\gen}(\delta)-nC_{\gen}(r_i\delta)\in\ker R_i^{n-1}.
\end{equation}
In all cases,
\begin{equation}\label{eq:character-maximal}
 R_i^nC_{\gen}(\delta)=n!C_{\gen}(r_i^n\delta),\qquad
 R_i^{n+1}C_{\gen}(\delta)=0.
\end{equation}
The same assertions hold for the starred crystal, with
$(R_i,r_i,\rho_i)$ replaced by
$(R_j^\star,r_j^\star,\rho_j^\star)$, $j\in I^\star$.
Under the hypotheses of Theorem~\ref{T:crystal}, they also
hold for $(L_i,l_i,\lambda_i)$.
\end{lemma}

\begin{proof} We work in an $i$-optimized seed; the starred case is dual.
Here $n=-\delta(i)$ and
$$ C_{\gen}(\delta)=\sum_\gamma
 \chi\bigl(\Gr^\gamma(M)\bigr)
 \mathbf x^{-\delta-\gamma B_\Delta},\qquad
 q_\gamma:=(\gamma B_\Delta)(i)\ge0.
$$
Indeed, there is no arrow from $i$ to a mutable vertex.
The summand indexed by $\gamma$ has $x_i$-degree $n-q_\gamma$.
Thus the character has degree at most $n$; when $n=0$,
polynomiality in $x_i$ gives $R_iC_{\gen}(\delta)=0$.

Suppose $n>0$. Theorem~\ref{T:rlrigid} and
Lemma~\ref{L:er} give general representations $M,M_r$ of weights
$\delta,r_i\delta$ in an exact sequence
$$ \tauh^{-1}E_i\longrightarrow M\longrightarrow M_r\longrightarrow 0.$$
Choose them in the open subsets computing the respective generic
characters. The surjectivity follows because $M_r$ is
$\mu$-supported and $E_i=S_i$.
The projective presentation $P_i\to\bigoplus_{u\to i}P_u$ of $\tau^{-1}S_i$
shows that the kernel of $M\twoheadrightarrow M_r$ is generated by its components
at the vertices $u\to i$. Since
$q_\gamma=\sum_{u\to i}\gamma(u)$, with arrows counted with
multiplicity and all summands nonnegative, $q_\gamma=0$ forces
$\gamma(u)=0$ at every such vertex. A quotient map
$M\twoheadrightarrow N$ with $\dv N=\gamma$ therefore
annihilates these generators, hence the kernel, and factors uniquely
through $M_r$. Conversely, every quotient of $M_r$ gives a
quotient of $M$ by composition. This correspondence is compatible
with families and identifies the quotient Grassmannians:
$\Gr^\gamma(M)\cong\Gr^\gamma(M_r)$.

Since $r_i\delta=\delta+\epc_i$ and
$R_i(\mathbf x^{-\delta})=n\mathbf x^{-r_i\delta}$,
the terms with $q_\gamma=0$ cancel in
$R_iC_{\gen}(\delta)-nC_{\gen}(r_i\delta)$.
All remaining terms have $x_i$-degree at most $n-2$;
combining coincident Laurent monomials cannot increase this bound.
Their sum belongs to $\uca(\Delta)$ and is polynomial in $x_i$;
it is therefore killed by $R_i^{n-1}$.
This proves \eqref{eq:character-remainder}.
Apply $R_i^{n-1}$ to the remainder estimate. The remainder
vanishes, and $\rho_i(r_i\delta)=n-1$. Induction gives the
first identity in \eqref{eq:character-maximal}; one further
derivative vanishes at the end of the string.

The opposite QP gives the starred statements. Under $\tau$-exactness,
$L_i=R_{\ibar}^\star$ and its corresponding crystal data are
$(l_i,\lambda_i)$, proving the last assertion.
\end{proof}

\begin{theorem}\label{T:normalUCA}
In the situation of Theorem~\ref{T:crystal}, the assignments
$$
 e_i\mapsto R_i,\qquad f_i\mapsto L_i,\qquad h_i\mapsto H_i
 \quad(i\in I)
$$
define a $U(\mf g'(C_I))$-module algebra structure on
$\uca(\Delta)$, integrable after scalar extension to a field of
characteristic zero. Under \eqref{eq:existwt}, the auxiliary
$H_a$, $a\in K$, extend this action to the full realization
$\mf g(C_I)$.
\end{theorem}

\begin{proof}
Work first over $\mathbb C$. The positive and negative Serre
relations follow from Theorem~\ref{T:crystalUCA}, since
$\tau$-exactness identifies the Cartan matrix on the partners with
$C_I$. For $i\ne j$, Lemma~\ref{L:H2E} and
Corollary~\ref{C:HomEij} give
$$
 \e(E_{\bar j}^\star,E_i)
 =\e(\tau^{-1}E_j,E_i)=\hom(E_i,E_j)=0,
$$
so Proposition~\ref{P:mixed} gives $[R_i,L_j]=0$.
The Cartan relations follow from homogeneity: the degrees of $R_j$
and $L_j$ are respectively $\wt(\epc_j)=\alpha_j$ and
$\wt(\ep_{\bar j}^\star)=-\alpha_j$, and the $H_a$ commute
because they are diagonal on Laurent monomials.
It remains to prove $[R_i,L_i]=H_i$.
The key point is that $\tau$-exactness identifies the
difference of the two nilpotence degrees with the natural weight.
We use this identity to verify the diagonal relation on seed variables.

Fix $i$. For $f\ne0$, write
$\deg_{R_i}f=\max\{n\ge0:R_i^nf\ne0\}$, and similarly
$\deg_{L_i}f$; both degrees are $-\infty$ at $f=0$.
Recall from the $\tau$-exact construction that
$F_i=\tau^{-1}E_i=E_{\ibar}^\star$.
The vectors $\dv E_i$ and $\dv F_i$ are tropical $\mc A$-points
by Lemma~\ref{L:dvtropx} and its opposite-QP version.
They become $e_i$ and $e_{\ibar}$ in the respective optimizing
seeds. Thus Lemma~\ref{L:tropxmu} identifies the two
nilpotence degrees with the tropical pairings at these vectors.
If $f\ne0$ is $H_i$-homogeneous of weight $m$, combine coincident
monomials in its Laurent expansion. Every exponent $a$ with nonzero
coefficient satisfies $a(\dv F_i-\dv E_i)=m$.
The constant difference persists when we take the maxima, so
\begin{equation}\label{eq:tau-degree-balance}
 \deg_{R_i} f=f^{\trop}(\dv E_i),\qquad
 \deg_{L_i} f=f^{\trop}(\dv F_i)=\deg_{R_i} f+m.
\end{equation}

Choose an $i$-optimized seed. For $k\ne i$, put
$\delta=-e_k$ and $m=\lambda_i(\delta)=\wt_i(\delta)$.
Then $C_{\gen}(\delta)=x_k$, $R_ix_k=0$, and $m\ge0$.
If $m=0$, Lemma~\ref{L:character-strings} gives $L_ix_k=0$.
If $m>0$, the lemma gives
$$
 L_ix_k=mC_{\gen}(l_i\delta)+v,\qquad
 \deg_{L_i}v\le m-2,\qquad
 R_iC_{\gen}(l_i\delta)=x_k.
$$
Here the last equality uses $\rho_i(l_i\delta)=1$ and
$r_il_i\delta=\delta$.
Both $L_ix_k$ and $C_{\gen}(l_i\delta)$ are homogeneous
of weight $m-2$. Thus, if $v\ne0$, its weight is $m-2$, and
\eqref{eq:tau-degree-balance} gives $\deg_{R_i}v\le0$.
Hence $R_iv=0$, and $([R_i,L_i]-H_i)(x_k)=0$ for every $k\ne i$.
As $R_i$ is a nonzero monomial multiple of $\partial_{x_i}$,
there is a rational function $c$ such that
$[R_i,L_i]-H_i=cR_i$ on $\mathbb C(\mathbf x)$.

Apply the same argument with $R_i,L_i,H_i$ replaced by
$L_i,R_i,-H_i$ in a dually $\ibar$-optimized seed
$\mathbf y$. Then $[R_i,L_i]-H_i$ kills every $y_k$
with $k\ne\ibar$.
The nonzero monomial
$L_i(y_{\ibar})=\prod_{\ibar\to u}y_u$ is killed by $L_i$
and by $[R_i,L_i]-H_i$. Its $H_i$-weight is
$-1-\updelta_{i,\ibar}$, because
$H_iy_{\ibar}=(1-\updelta_{i,\ibar})y_{\ibar}$ and
$L_i$ has degree $-2$.
Equation~\eqref{eq:tau-degree-balance} therefore gives
$\deg_{R_i}L_i(y_{\ibar})=1+\updelta_{i,\ibar}>0$.
Thus
$$0=([R_i,L_i]-H_i)L_i(y_{\ibar})=cR_iL_i(y_{\ibar})$$
forces $c=0$, proving the diagonal relation also for self-pairs.

All operators preserve the integral upper cluster algebra, so their
relations descend to it. Proposition~\ref{P:root-derivations}
gives local nilpotence. The integral compatible gradings give a
simultaneous weight decomposition of every seed Laurent ring and
their intersection. These facts prove integrability after scalar
extension to any characteristic-zero field, including the additional
Cartan action when \eqref{eq:existwt} holds.
\end{proof}

\paragraph{\em Restriction on the Cartan type.}
\begin{proposition}\label{P:Cartan-growth}
Let $C=(c_{ij})_{i,j\in I}$ be a symmetric generalized Cartan matrix,
and let
$$ \mathbb C[x_1,\ldots,x_N]\subseteq\mathcal U\subseteq
 \mathcal L=\mathbb C[x_1^{\pm1},\ldots,x_N^{\pm1}].
$$
If derivations $R_i,L_i,H_i$ of $\mathcal U$ satisfy the
Chevalley--Serre relations for $C$, with $R_i\ne0$ for every $i$,
then $C$ is positive semidefinite.
\end{proposition}
\begin{proof}
We compare two growth properties. Lie words in finitely many
Laurent derivations have polynomially bounded growth, whereas an
indefinite Kac--Moody algebra contains a free Lie algebra on two
generators. An injective realization of the latter by Laurent
derivations would contradict the former.

Localizing at the seed variables extends the derivations and their relations to $\mathcal L$. 
The dimension of the space spanned by iterated commutators of length at most $n$ grows at most polynomially in $n$.
Indeed, for $\theta_a=x_a\partial_{x_a}$,
$$ [x^\alpha\theta_a,x^\beta\theta_b]
 =x^{\alpha+\beta}(\beta_a\theta_b-\alpha_b\theta_a).
$$
If the Laurent supports of finitely many generators lie in
$[-M,M]^N$, their Lie words of length at most $m$ span a space
$\mathfrak d_{\le m}$ satisfying
\begin{equation}\label{eq:Laurent-Lie-growth}
 \dim\mathfrak d_{\le m}\le N(2mM+1)^N.
\end{equation}
Applying this bound to any two Laurent derivations shows that they
cannot freely generate a Lie algebra on two generators
\cite[Proposition~4.3]{CCMTZ}.

Suppose, for a contradiction, that $C$ is not positive
semidefinite. We first restrict to a nonsingular indefinite Cartan
matrix. The principal-minor criterion gives a principal submatrix
of negative determinant. One of its connected components has negative
determinant and is therefore a nonsingular indecomposable indefinite
principal submatrix $A=C_J$.
The restricted relations define
$$
 \phi:\mathfrak g(A)\longrightarrow
 \operatorname{Der}_{\mathbb C}(\mathcal L),\qquad
 (e_i,f_i,h_i)\longmapsto(R_i,L_i,H_i)
 \quad(i\in J).
$$
For $\sum_i a_iH_i=0$, taking brackets with the nonzero $R_j$
gives $A^{\mathsf T}a=0$, hence $a=0$. Since $A$ is nonsingular,
the $h_i$ span its full Cartan subalgebra, so
$\ker\phi\cap\mathfrak h(A)=0$.
By the ideal-intersection property for symmetrizable
Kac--Moody Lie algebras, an ideal whose intersection with the full
Cartan subalgebra is zero must be zero
\cite[Chapters~1 and~9]{Kac}; see also
\cite[Section~2.1]{CCMTZ}. Hence $\phi$ is injective.

We now obtain the contradiction by finding a free Lie subalgebra
on two generators in $\mathfrak g(A)$.
Perron--Frobenius applied to $2\operatorname{Id}-A$, followed by
rational approximation, gives $v\in\mathbb Z_{>0}^J$ with
$Av<0$. Thus $\beta=\sum_i v_i\alpha_i$ is an imaginary root
of negative norm \cite[Theorem~5.4]{Kac}. By
\cite[Proposition~4.4 and Lemma~4.5]{CCMTZ}, some
$\mathfrak g_{s\beta}$ contains two independent elements freely
generating a Lie subalgebra. Their images under $\phi$ contradict
\eqref{eq:Laurent-Lie-growth}.
\end{proof}

\begin{corollary}\label{C:Cartan-positive}
Under the hypotheses of Theorem~\ref{T:crystal}, the matrix $C_I$
is positive semidefinite. Hence every connected component is of finite
or affine type. In particular,
$$
 \e(\mc E_i^\mu,\mc E_j^\mu)+\e(\mc E_j^\mu,\mc E_i^\mu)
 =-c_{ij}\le2\qquad(i\ne j).
$$
If $c_{ij}=-2$, then $\{i,j\}$ is an entire connected component.
\end{corollary}
\begin{proof}
Theorem~\ref{T:normalUCA} supplies the required relations, and
$R_i(x_i)=\prod_{u\to i}x_u\ne0$ in an $i$-optimized seed.
Proposition~\ref{P:Cartan-growth} therefore applies to the
complexified upper cluster algebra. The type assertion follows from
the classification of symmetric generalized Cartan matrices
\cite[Chapter~4]{Kac}. Its $\{i,j\}$ principal minor
gives $4-c_{ij}^2\ge0$.
If $c_{ij}=-2$, then $w=e_i+e_j$ satisfies
$w^{\mathsf T}C_Iw=0$, hence $C_Iw=0$. For $k\notin\{i,j\}$,
this gives $c_{ki}+c_{kj}=0$; both nonpositive entries must vanish.
\end{proof}

\subsection{Perfect Bases and Normality}\label{ss:generic-perfect}
The character-string estimate of Lemma~\ref{L:character-strings}
shows that, whenever the generic characters form a basis, the root
derivations follow the crystal strings up to terms of lower string
length. We express this property through perfectness and then use it
to establish normality in the $\tau$-exact setting.

\begin{definition}\label{D:BK-perfect}
A homogeneous basis
$\Bup=\{\Bup(\delta):\delta\in\mc B\}$, normalized by
$\Bup(\mathbf0)=1$, is \emph{BK-perfect} for the upper crystal if
$$
 R_i\Bup(\delta)=\rho_i(\delta)\Bup(r_i\delta)+v_i,
 \qquad
 v_i\in\operatorname{span}_{\k}
 \{\Bup(\eta):\rho_i(\eta)<\rho_i(\delta)-1\}.
$$
Here the auxiliary crystal element is assigned
$\Bup(0_{\mathrm{cr}})=0$; it is distinct from the zero lattice
vector $\mathbf0$. The basis is \emph{BK-biperfect} if it is
homogeneous for both gradings and satisfies the same condition for
$(R_j^\star,r_j^\star,\rho_j^\star)$, $j\in I^\star$.
In the situation of Theorem~\ref{T:crystal}, it is BK-perfect for
the seminormal crystal if the condition holds for both
$(R_i,r_i,\rho_i)$ and $(L_i,l_i,\lambda_i)$.
These are the conditions of \cite{BK,BKK} in the present setting.
Here homogeneity is understood with respect to the compatible
grading of each crystal structure, including any additional Cartan
coordinates.
\end{definition}
Mutation reindexes these bases and preserves perfectness, since the
root derivations and crystal operators commute with the change of seed.

\begin{theorem}\label{T:genericCR}
Assume that the generic characters form a $\k$-basis of
$\uca(\Delta)$. Then this basis is BK-biperfect with respect to
$(R_i)_{i\in I}$ and $(R_j^\star)_{j\in I^\star}$.
\end{theorem}

\begin{proof}
The generic characters are homogeneous by compatibility of the grading.
Fix $i\in I$. For a nonzero finite basis expansion $f=\sum_\delta a_\delta
C_{\gen}(\delta)$, let $N$ be the largest $\rho_i(\delta)$
with $a_\delta\ne0$. Lemma~\ref{L:character-strings} gives
$$
 R_i^Nf=N!\sum_{\rho_i(\delta)=N}a_\delta
 C_{\gen}(r_i^N\delta)\ne0.
$$
The inequality follows because $r_i^N$ is injective on
$\{\delta:\rho_i(\delta)=N\}$ and the characters are independent.
Together with \eqref{eq:character-maximal}, this proves
$$
 \ker R_i^{n+1}=\operatorname{span}_{\k}
 \{C_{\gen}(\delta):\rho_i(\delta)\le n\}\qquad(n\ge0).
$$
Thus the remainder in \eqref{eq:character-remainder} is a linear
combination of characters with strictly smaller string parameter
than $\rho_i(\delta)-1$, as required for BK-perfectness.
The starred argument is identical.
\end{proof}

After scalar extension to characteristic zero, write
$R_i^{(n)}=R_i^n/n!$ for the divided powers.
\begin{corollary}\label{C:string-kernels-hw}
For a BK-perfect basis $\Bup$, the identities
\eqref{eq:character-maximal} hold with $C_{\gen}$ replaced by
$\Bup$. For $\mathbf n\in\mathbb Z_{\ge0}^{I}$, the simultaneous kernel
$\bigcap_{i\in I}\ker R_i^{n_i+1}$ has basis
$$
 \{\Bup(\delta):\rho_i(\delta)\le n_i\text{ for all }i\in I\}.
$$
In particular, putting
$\mc B^{\mathrm{hw}}=\{\delta:\rho_i(\delta)=0\text{ for all }i\in I\}$,
$$
 \uca(\Delta)^{\mathfrak n_+(C_I)}
 =\bigoplus_{\delta\in\mc B^{\mathrm{hw}}}\k\,\Bup(\delta).
$$
The analogous statements hold for the starred family, or for the
lowering family when the basis is perfect for it. They apply to the
generic basis whenever it exists.
\end{corollary}
\begin{proof}
Iteration of the perfectness relation gives the maximal-derivative
identities. The argument in Theorem~\ref{T:genericCR} then gives
$\ker R_i^{n+1}=\operatorname{span}_{\k}\{\Bup(\delta):\rho_i(\delta)\le n\}$.
Intersect these submodules, which are spanned by subsets of a common
basis. Upper seminormality identifies $\rho_i=0$ with being killed
by the crystal raising operator $r_i$.
\end{proof}

Perfectness also connects the algebraic action to the usual finite-type crystals. 
This is needed for normality and, later, the Coxeter relations for string reflections.
\begin{definition}\label{D:normal}
We call a seminormal crystal \emph{normal} if its
restriction to every finite-type subset $J\subseteq I$, with weights
restricted to $\bigoplus_{j\in J}\mathbb Z h_j$, is a disjoint
union of highest-weight crystals of finite-dimensional simple
modules for $\mathfrak g_J$. Thus normality is used here in the
finite-type-restriction sense, also called \emph{regularity}.
\end{definition}

\begin{corollary}\label{C:normal}
In the situation of Theorem~\ref{T:crystal}, suppose that the generic
characters form a basis of $\uca(\Delta)$. Then the seminormal
crystal is normal.
\end{corollary}
\begin{proof}
Work over $\mathbb C$, and let $J\subseteq I$ have finite Cartan
type. The integrable action of Theorem~\ref{T:normalUCA}, restricted
to $\mathfrak g_J$, is locally finite and is a direct sum of
finite-dimensional simple modules \cite[Chapter~3]{Kac}; their
multiplicities need not be finite. By Theorem~\ref{T:genericCR} with
$I^\star=\bar I$, the generic basis is homogeneous and perfect for
all raising and lowering operators indexed by $J$.
The perfect-basis comparison theorem
\cite[Lemma~5.50 and Main Theorem~5.55]{BK} identifies its crystal
with that of a direct sum of canonical bases of the simple summands.
\end{proof}

Consequently, the Lusztig involutions of Section~\ref{ss:dualcrystal}
are defined on every finite-type restriction and are compatible with
mutation and every permitted relative $\tau$-translation.

\begin{conjecture}
Suppose that the theta basis \cite{GHKK} or the classical
specialization of the triangular basis \cite{Qt} is a basis of the
unlocalized upper cluster algebra $\uca(\Delta)$, indexed by
$\trop(\Delta,\S)$. Then this basis is BK-biperfect.
\end{conjecture}

\subsection{Weyl and Tits Actions}
There are two related constructions. The seminormal crystal
provides individual string reflections; normality ensures that they
satisfy the Coxeter relations. Independently, the integrable
Kac--Moody action produces algebra automorphisms satisfying the
Tits relations.

Let $W$ be the Weyl group of $C_I$. For the seminormal crystal of
Theorem~\ref{T:crystal}, define the string reflection by
\begin{equation}\label{eq:si}
 s_i(\delta)=
 \begin{cases}
 l_i^n(\delta),&n=\wt_i(\delta)\ge0,\\
 r_i^{-n}(\delta),&n=\wt_i(\delta)<0.
 \end{cases}
\end{equation}

\begin{corollary}\label{C:Weyl}
Under the hypotheses of Theorem~\ref{T:crystal}, each string
reflection $s_i$ is a mutation-compatible involution, and
$$
 \wt(s_i\delta)=\wt(\delta)-\wt_i(\delta)\alpha_i.
$$
If $\mub(E_i)=S_i$ and $n=\wt_i(\delta)$, then
\begin{align*}
s_i(\delta)=\mub^{-1}\bigl(\mub(\delta)-n\check\delta_{S_i}\bigr) 
	=(\mu_i\mub)^{-1}\bigl(\mu_i\mub(\delta)+ne_i\bigr),
\end{align*}
where $\mu_i$ in the second expression is an auxiliary mutation at
the frozen vertex $i$.
If, moreover, the crystal is normal---in particular, under the
additional basis hypothesis of Corollary~\ref{C:normal}---these
involutions satisfy the Coxeter relations and define an action of $W$.
\end{corollary}
\begin{proof}
On an $i$-string of length $N$, write $q=\rho_i(\delta)$.
Seminormality gives $\lambda_i(\delta)=N-q$ and
$\wt_i(\delta)=N-2q$; hence \eqref{eq:si} replaces $q$ by
$N-q$. This proves involutivity. The formula for $\wt(s_i\delta)$
follows from the crystal weight axiom. Theorem~\ref{T:rle} and mutation
compatibility of the grading give mutation compatibility.
At an $i$-optimized seed, Lemma~\ref{L:mu2S} makes $r_i$
add $\check\delta_{S_i}$; its inverse $l_i$ subtracts that
vector whenever the output is supported. Seminormality keeps all
intermediate points in \eqref{eq:si} supported, proving the first
formula. Theorem~\ref{T:rle}, applied to $\mu_i\mub$,
gives the second.
Finally, under normality every finite braid relation is checked in
a finite-type rank-two restriction \cite{K2}; together with the
square relations, these are all Coxeter relations.
\end{proof}
\noindent In optimizing coordinates, the individual reflections are linear,
but need not be cluster automorphisms.

We now pass from crystal operators to algebra automorphisms.
The signs in the Tits relations explain why their lifts do not, in
general, give a literal Weyl-group action.
For the algebraic lifts below, assume only the hypotheses of
Theorem~\ref{T:crystal}. Let $\mb F$ be a characteristic-zero
field with the specified $\k$-algebra structure, and put
$$
 A_{\mb F}=\uca(\Delta)\otimes_{\k}\mb F,
 \qquad Q_I^\vee=\bigoplus_{i\in I}\mathbb Z h_i.
$$
For $h\in Q_I^\vee$, define $\Theta_h(f)=(-1)^{\nu(h)}f$
on the weight-$\nu$ subspace of $A_{\mb F}$.

\begin{theorem}[Tits lift]\label{T:Tits-lift}
In the situation of Theorem~\ref{T:crystal}, the operators
$$
 \dot s_i=\exp(R_i)\exp(-L_i)\exp(R_i)
 \qquad(i\in I)
$$
are algebra automorphisms of $A_{\mb F}$. They satisfy the braid
relations and
\begin{align*}
 \Theta_h\Theta_{h'}&=\Theta_{h+h'},& \Theta_{2h}&=1,\\
 \dot s_i^2&=\Theta_{h_i},&
 \dot s_i\Theta_h\dot s_i^{-1}&=\Theta_{s_i h},
\end{align*}
where $s_i h=h-\alpha_i(h)h_i$.
Consequently the assignments
$$
 \widetilde s_i\longmapsto\dot s_i,\qquad
 h(-1)\longmapsto\Theta_h
$$
define a homomorphism
$\widetilde W\to\operatorname{Aut}_{\mb F\text{-alg}}(A_{\mb F})$,
where $\widetilde W$ is the Tits extension with these generators and
relations. Modulo the subgroup of sign automorphisms, the images of
the $\dot s_i$ define a possibly non-faithful homomorphism from $W$.
\end{theorem}
\begin{proof}
Theorem~\ref{T:normalUCA} makes $A_{\mb F}$ an integrable
$U(\mathfrak g')$-module algebra. The exponential of a locally
nilpotent derivation is an algebra automorphism, with inverse the
exponential of its negative \cite[Section~2.1]{vandenEssen}.
The maps $\Theta_h$ are algebra automorphisms because weights add
under multiplication.
The relations among $\dot s_i$ and $\Theta_h$ follow from the
Tits construction on integrable modules \cite[Section~2.4]{AT}; see also
\cite[Definition~8.2 and Proposition~8.3]{T} for the sign and square
relations. These can be checked in rank-one and finite-type
rank-two restrictions, which are locally finite by integrability.
The conjugation relation makes the sign subgroup normal;
after quotienting by it, the squares and braid relations are the
Coxeter relations.
\end{proof}
\noindent The square relation need not be trivial. 
For the standard rank-one action $R_i=X\partial_Y$ and $L_i=Y\partial_X$ on $\mb F[X,Y]$,
the lift sends $X\mapsto-Y$ and $Y\mapsto X$, so its square negates both variables.

When a perfect basis is available, the algebraic lift recovers
the corresponding string reflection at the leading level of the
string-length filtration.

\begin{corollary}\label{C:Tits-perfect-basis}
In the situation of Theorem~\ref{T:crystal}, let
$\Bup=\{\Bup(\delta):\delta\in\trop(\Delta,\S)\}$ be a
BK-perfect basis for the seminormal crystal, extended to $A_{\mb F}$.
Under the hypotheses of Theorem~\ref{T:genericCR}, one may take the
generic basis. Fix $i\in I$, and put
$$
 N_i=\rho_i+\lambda_i,\qquad
 F_m^{(i)}=\operatorname{span}_{\mb F}
 \{\Bup(\eta):N_i(\eta)\le m\},\qquad F_{-1}^{(i)}=0.
$$
Then $F_m^{(i)}$ is stable under $R_i,L_i,\dot s_i$, and
$$
 \dot s_i\Bup(\delta)
 \equiv(-1)^{\lambda_i(\delta)}\Bup(s_i\delta)
 \pmod{F_{N_i(\delta)-1}^{(i)}}.
$$
Thus the associated graded action recovers the crystal reflection
with its sign; compare \cite{Bws}.
\end{corollary}
\begin{proof}
Homogeneity and $N_i=2\rho_i+\wt_i=2\lambda_i-\wt_i$ show
that every remainder in either perfectness formula has strictly
smaller $N_i$, while its leading term stays in the same string.
This proves stability under $R_i,L_i$, and hence under $\dot s_i$.
On $F_N^{(i)}/F_{N-1}^{(i)}$, an $i$-string becomes the standard
degree-$N$ $\mathfrak{sl}_2$-module, with
$$
 \overline{\Bup(\delta^{(q)})}\longleftrightarrow X^{N-q}Y^q,
 \qquad \rho_i(\delta^{(q)})=q,
 \qquad R_i=X\partial_Y,\quad L_i=Y\partial_X.
$$
The substitution $X\mapsto-Y$, $Y\mapsto X$ sends this monomial
to $(-1)^{N-q}X^qY^{N-q}$. Since
$N-q=\lambda_i(\delta^{(q)})$ and
$s_i\delta^{(q)}=\delta^{(N-q)}$, the claimed congruence follows.
\end{proof}

\section{Biperfect Bases}\label{S:biperfect}
BK-biperfectness controls the leading term along each crystal string, but still leaves considerable
freedom in the choice of a basis. We first illustrate this freedom, then impose compatible
pointedness to obtain a classification suited to the cluster structure. The resulting changes of
basis are unitriangular, and their coefficients satisfy simple relations. A coefficient graph records
the remaining freedom.

\subsection{BK-biperfect Bases}\label{ss:perfect-transitions}
Fix a BK-biperfect basis $B=\{B(\delta):\delta\in\mc B\}$ of
$\uca(\Delta)$, together with its two crystal structures and weight
maps. We work over $\k=\mathbb Z$ or a field of characteristic zero,
extending scalars when necessary. 
The auxiliary crystal element is $0_{\mathrm{cr}}$, distinct from the zero
lattice vector $\mathbf0$; thus $B(0_{\mathrm{cr}})=0$ and $B(\mathbf0)=1$.

To treat the two families together, let $\mathcal I=I\sqcup I^\star$
be the disjoint union of operator labels (even when the same vertex
belongs to both index sets), and write $(r_a,l_a,\rho_a,D_a)$ for either
$(r_i,l_i,\rho_i,R_i)$ or
$(r_j^\star,l_j^\star,\rho_j^\star,R_j^\star)$.
We write $\eta\preceq_\rho\delta$ if
$\rho_a(\eta)\le\rho_a(\delta)$ for every $a\in\mathcal I$,
and $\eta\llcurly_\rho\delta$ if all these inequalities are strict.
By Corollary~\ref{C:string-kernels-hw}, the string filtration is
$$ K_{a,n}:=\ker D_a^{n+1} =\operatorname{span}_{\k}\{B(\delta):\rho_a(\delta)\le n\} \qquad(n\ge0). $$
Set $K_{a,n}=0$ for $n<0$. The map $D_a$ induces maps
$K_{a,n}/K_{a,n-1}\to K_{a,n-1}/K_{a,n-2}$.

\begin{proposition}\label{T:allcr}
Let $\Psi$ be a $\k$-linear automorphism of $\uca(\Delta)$
that fixes $1$ and preserves both weight gradings, and write
\begin{equation}\label{eq:crtriangular} B'(\delta)=\Psi(B(\delta)) =\sum_{\eta\in\mc B}T_{\delta,\eta}B(\eta).\end{equation}
Then $B'$ is BK-biperfect for the same crystal structures if and
only if $\Psi(K_{a,n})=K_{a,n}$ for every $a$ and $n$, and
the induced maps on the associated graded spaces commute with those
induced by $D_a$.
\end{proposition}
The proof, including the equivalent relations on the coefficients
$T_{\delta,\eta}$, is given in
Appendix~\ref{ss:perfect-transition-proof}.

In particular, terms with strictly smaller string lengths may be
added freely.
\begin{corollary}\label{C:strict-biperfect}
Assume $I\ne\varnothing$. For each $\delta$, choose finitely
many coefficients $T_{\delta,\eta}\in\k$, supported on
$\eta\llcurly_\rho\delta$ with the same two weights. Then
$$ B'(\delta)=B(\delta)+ \sum_{\eta\llcurly_\rho\delta}T_{\delta,\eta}B(\eta) $$
is a BK-biperfect basis for the same crystal structures.
\end{corollary}
\begin{proof}
Write $\Psi=1+N$ for this change of basis. For any fixed
$a\in\mathcal I$, each application of $N$ strictly decreases
$\rho_a$, so $N^{\rho_a(\delta)+1}B(\delta)=0$.
Thus $\Psi^{-1}=\sum_{m\ge0}(-N)^m$ is finite on each element.
The map $\Psi$ fixes $1$, preserves both gradings and every string
filtration, and induces the identity on the associated graded spaces.
Proposition~\ref{T:allcr} applies.
\end{proof}

The following example, found with the help of Normaliz \cite{BI},
shows that such a change can remove a cluster monomial from the basis.
\begin{example}\label{ex:SL5}
Consider $\k[U]$ for $G=\SL_5$, with the seed
$$ \UAfour $$
The blue and red path modules are the boundary and dual boundary
representations. The dashed arrows record
$\e(\mc E_i^\mu,\mc E_j^\mu)$, and hence the Cartan type.
Theorems~\ref{T:musupp} and~\ref{T:HomE} give the supported weights
\begin{align*}
\delta&=(2,0,-2,-2,2,0,-1,-2,-1,0),\\
\delta'&=(0,0,0,0,0,0,0,-3,0,-1),
\end{align*}
with $\wt(\delta)=\wt(\delta')=(-1,-3,-3,-1)$
in fundamental-weight coordinates and
\begin{align*}
(\rho(\delta);\rho^\star(\delta))&=(1,4,1,2;1,4,1,2),\\
(\rho(\delta');\rho^\star(\delta'))&=(0,3,0,1;0,3,0,1).
\end{align*}
Here $\rho=(\rho_7,\rho_8,\rho_9,\rho_{10})$, with the same
ordering for $\rho^\star$. In particular,
$\delta'\llcurly_\rho\delta$. The torus gradings for the two
translation actions, with the prescribed relabeling, give the same
weights to $\delta$ and $\delta'$ for both structures.

By Corollary~\ref{C:strict-biperfect}, replacing $B(\delta)$ by
$B(\delta)+cB(\delta')$ and leaving the other elements unchanged
gives a BK-biperfect basis for every $c\in\k$. Start with the
cluster-monomial generic basis, which exists in this finite cluster
type. For $c\ne0$, the new basis does not contain the cluster
monomial $B(\delta)$. A similar example appears in
\cite[Example~2.7(ii)]{B}.
\end{example}

Thus BK-biperfectness alone does not ensure that a basis contains all
cluster monomials. We now require each $B(\delta)$ to be compatibly
pointed at $\delta$.

\subsection{Compatibly Pointed Bases}\label{ss:biperfect-parameters}
Throughout this subsection and the next, assume that $B_\Delta$
has full rank and the seed is reachable, and work over
$\k=\mathbb Z$ or $\mathbb C$.
For a seed $t$, write $\eta\prec_t\delta$ if
$$ \eta=\delta+\gamma B_{\Delta_t},\qquad 0\ne\gamma\in\mathbb Z_{\ge0}^{\Delta_0^\mu}. $$
We write $\eta\prec_{\mf T}\delta$ if this holds at every seed,
using the corresponding tropical representatives. Let ${\sf M}$
denote the set of tropical degree classes, with representatives in
${\sf M}_t\cong\mathbb Z^{\Delta_0}$.

\begin{definition}[\cite{Qt}]\label{D:pointed-normalized}
An element $z\in\k[\mathbf x_t^{\pm1}]$ is \emph{pointed at}
$\delta\in{\sf M}_t$ if
$$ z=\mathbf x_t^{-\delta}F(\check{\mathbf y}_t),\qquad F\in\k[\check y_u:u\in\Delta_0^\mu],\qquad F(0)=1. $$
It is \emph{compatibly pointed at} $[\delta]$ if it is pointed at
the corresponding representative in every seed.
\end{definition}
A basis $B=\{B(\delta):\delta\in\Theta\}$, with
$\Theta\subseteq{\sf M}$, is called \emph{good} if every $B(\delta)$
is compatibly pointed at $\delta$. We consider bases of
$\uca(\Delta)$, where frozen
variables remain polynomial, except in the following recalled theorem.

\begin{theorem}[\cite{Qb}, Theorem~1.2.1]\label{T:good}
Over $\k=\mathbb Z$, a collection in
$\uca(\Delta)_{\mathrm{loc}}=\bigcap_t\k[\mathbf x_t^{\pm1}]$
compatibly pointed at all tropical degrees is a basis containing all
cluster monomials. Such bases exist. Relative to one such basis, every
other is obtained by freely choosing the coefficients at the finite
sets $\{\eta:\eta\prec_{\mf T}\delta\}$.
\end{theorem}

For the algebra with polynomial frozen variables, we use the following
relative form.
\begin{proposition}\label{T:relative-good}
Let $A$ be a free $\k$-module of Laurent polynomials with a basis
$\{B(\delta):\delta\in\Theta\}$ compatibly pointed at a subset
$\Theta$ of tropical degrees. Set
$D_\delta=\{\eta\in\Theta:\eta\prec_{\mf T}\delta\}$.
Each $D_\delta$ is finite. The compatibly pointed bases
$B'=\{B'(\delta):\delta\in\Theta\}$ of $A$ are precisely
$$ B'(\delta)=B(\delta)+\sum_{\eta\in D_\delta}T_{\delta,\eta}B(\eta), \qquad T_{\delta,\eta}\in\k. $$
Here $T$ is the transition matrix, with $T_{\delta,\delta}=1$
and all other entries zero outside $D_\delta$.
\end{proposition}
\begin{proof}
Finiteness follows from the two-seed lower sets in
\cite[Lemma~5.1.1]{Qb}. In a finite expansion in $B$, the leading
monomial of a maximal degree cannot cancel with terms of another
maximal degree. Thus an element pointed at $\delta$ has coefficient
one at $B(\delta)$, and every other degree lies strictly below
$\delta$. Applying this at every seed gives the stated form.
Conversely, these sums remain pointed in every seed. The correction
operator $N$ strictly lowers dominance and its iterates stay in the
finite set $D_\delta$; hence $(1+N)^{-1}=\sum_{m\ge0}(-N)^m$
is finite on each element. The sums therefore form a basis.
\end{proof}

\begin{definition}\label{D:biperfect}
A BK-biperfect basis $B=\{B(\delta):\delta\in\mc B\}$ is
\emph{biperfect} if $B(\delta)$ is compatibly pointed at
$\delta$ for every $\delta\in\mc B$.
\end{definition}
A good basis that is BK-biperfect in one seed is biperfect by
Definition~\ref{D:BK-perfect}. Under the hypotheses of
Theorem~\ref{T:genericCR}, the generic basis is biperfect.

\begin{lemma}\label{L:T2rho}
For $\delta,\eta\in\trop(\Delta,\S)$, the relation
$\eta\prec_{\mf T}\delta$ implies
$\rho_i(\eta)\le\rho_i(\delta)$ for every $i\in I$.
The analogous assertion for $\rhoc_j^\star$, $j\in I^\star$,
holds on $\check{\trop}(\Delta,\S)$ with dominance in the dual charts.
\end{lemma}
\begin{proof}
Choose a seed with $E_i=S_i$. Its frozen column $b_i$ is
nonnegative. Writing $\eta=\delta+\gamma B_\Delta$, with
$\gamma\ge0$, gives
$\rho_i(\eta)=[-\eta(i)]_+\le[-\delta(i)]_+=\rho_i(\delta)$.
Apply the same argument to the opposite QP for the dual assertion.
\end{proof}

Fix a biperfect basis $B$ of $\uca(\Delta)$, indexed by
$\Theta=\trop(\Delta,\S)$, with the two weight maps as before.
We now impose perfectness on the pointed changes of basis in
Proposition~\ref{T:relative-good}.
\begin{definition}\label{D:admissible-coefficients}
A pair $(\delta,\eta)$ is \emph{admissible} if
$$
\delta\ne0,\qquad \eta\in D_\delta,\qquad
\eta\preceq_\rho\delta,\qquad
\wt(\eta)=\wt(\delta),\quad
\wt^\star(\eta)=\wt^\star(\delta).
$$
\end{definition}

\begin{theorem}\label{C:goodcr}
With the crystal structures and weight maps fixed, the biperfect
bases $B'=\{B'(\delta):\delta\in\Theta\}$ are precisely those
whose transition matrices $T$ satisfy
$$
T_{\delta,\delta}=1,\qquad
T_{\delta,\eta}=0
\quad\text{if }\eta\ne\delta\text{ and }(\delta,\eta)
\text{ is inadmissible},
$$
and, for every $a\in\mathcal I$,
\begin{equation}\label{eq:pointed-perfect-recursion}
T_{r_a\delta,\theta}=T_{\delta,l_a\theta}
\quad\text{whenever }\rho_a(\delta)>0,\quad
\rho_a(\theta)=\rho_a(\delta)-1.
\end{equation}
Entries involving $0_{\mathrm{cr}}$ are interpreted as zero.
\end{theorem}
\begin{proof}
Proposition~\ref{T:relative-good} gives the pointed transitions and
finite unitriangular inverses. Perfectness requires preservation of the
unit, both gradings, and every string filtration; these give the
admissibility conditions. Conversely, admissible coefficients ensure
that the transition and its inverse preserve these data, since the
inverse is a finite sum of compositions that decrease dominance and
do not increase any string length.

It remains to compare the leading string terms. Fix $a$ and put
$n=\rho_a(\delta)>0$. For every $\theta$ with
$\rho_a(\theta)=n-1$, the coefficients of $B(\theta)$ in
$D_aB'(\delta)$ and $nB'(r_a\delta)$, modulo $K_{a,n-2}$,
are respectively
$nT_{\delta,l_a\theta}$ and $nT_{r_a\delta,\theta}$.
Thus \eqref{eq:pointed-perfect-recursion} is exactly the remaining
perfectness condition. In particular, for $\theta=r_a\delta$,
both coefficients equal $n$. If $n=0$, filtration preservation
gives $D_aB'(\delta)=0$.
\end{proof}

\subsection{Parameters and Uniqueness}\label{ss:biperfect-rigidity}
The relations in Theorem~\ref{C:goodcr} either identify two
coefficients or force a coefficient to vanish. We record these
relations in a graph.

\begin{definition}\label{D:coefficient-graph}
The vertices are ordered pairs $(\delta,\eta)\in\Theta^2$ with
$\delta\ne\eta$. Join
$$ (\delta,\eta)\;---\;(r_a\delta,r_a\eta) \quad\text{if }\rho_a(\delta)=\rho_a(\eta)>0. $$
We mark a pair $(\delta,\eta)$ if it is inadmissible, or if
there is an $a\in\mathcal I$ such that
\begin{equation}\label{eq:coefficient-vanishing}
\rho_a(\delta)=\rho_a(\eta),\qquad
l_a\delta\ne 0_{\mathrm{cr}},\qquad l_a\eta=0_{\mathrm{cr}}.
\end{equation}
In either case, the off-diagonal entry $T_{\delta,\eta}$ must vanish.
A connected component is \emph{marked} if it contains a marked pair,
and \emph{unmarked} otherwise. Let $\mathfrak C$ be the set of
unmarked components.
\end{definition}
For the second vanishing condition, apply
\eqref{eq:pointed-perfect-recursion} to $l_a\delta$ and
$\theta=\eta$. Since the coefficients are equal along edges, one
marked pair forces every coefficient in its component to vanish.
Each unmarked component contributes one free parameter.

\begin{theorem}\label{T:biperfect-parameters}
Relative to $B$, all biperfect bases indexed by $\Theta$
for the fixed crystal structures and weight maps are parametrized by
$\prod_{C\in\mathfrak C}\k$.
For $\mathbf t=(t_C)$, the transition matrix is given by
$$
T_{\delta,\eta}=\begin{cases}
1,&\eta=\delta,\\
t_C,&(\delta,\eta)\in C\in\mathfrak C,\\
0,&\text{otherwise}.
\end{cases}
$$
The corresponding basis is
\begin{equation}\label{eq:parameter-basis} B_{\mathbf t}(\delta)= \sum_{\eta\in\Theta}T_{\delta,\eta}B(\eta).\end{equation}
Each sum is finite. The set $\mathfrak C$ depends only on the two crystal structures,
their weight maps, and the order $\prec_{\mf T}$. The coordinates
$t_C$ are relative to the chosen basis $B$.
\end{theorem}
\begin{proof}
Taking $\theta=r_a\eta$ in
\eqref{eq:pointed-perfect-recursion} shows that coefficients are
constant along every edge. They vanish at each marked pair, and
hence on every component containing a marked pair. Each remaining
component therefore has at most one free coefficient.

Conversely, assign an arbitrary scalar to each component in
$\mathfrak C$. The sums in \eqref{eq:parameter-basis} are finite,
and all nonzero off-diagonal entries are admissible. To verify the
recursion, first suppose $\theta=r_a\delta$; then both entries
are diagonal and equal to one. Otherwise, if $l_a\theta\ne 0_{\mathrm{cr}}$,
the two off-diagonal pairs in the recursion are joined by an edge.
If $l_a\theta=0_{\mathrm{cr}}$, the pair $(r_a\delta,\theta)$ is marked
by \eqref{eq:coefficient-vanishing}, so both sides vanish.
Theorem~\ref{C:goodcr} now proves biperfectness. Every unmarked
component contains a pair, so its parameter is recovered from the
corresponding transition entry.
\end{proof}

\begin{corollary}\label{C:biperfect-rigidity}
Under the hypotheses of Theorem~\ref{T:biperfect-parameters},
the biperfect basis is unique if and only if $\mathfrak C=\varnothing$.
For $\delta\in\Theta$, the equality $B'(\delta)=B(\delta)$
holds in every biperfect basis $B'$ if and only if no component of
$\mathfrak C$ contains a pair with first entry $\delta$.

For a finite set $F\subset\Theta$, let $\mathfrak C_F$ be the
components in $\mathfrak C$ containing a pair $(\delta,\eta)$
with $\delta\in F$. Then $\mathfrak C_F$ is finite, and the
possible restrictions of biperfect bases to $F$ are parametrized
by $\k^{\mathfrak C_F}$. Over $\mathbb C$, this is an affine
space of dimension $\lvert\mathfrak C_F\rvert$.
\end{corollary}
\begin{proof}
Each set $D_\delta$ is finite. A parameter affects
$B'(\delta)$ exactly when its component contains a pair
with first entry $\delta$. Any assignment on $\mathfrak C_F$
extends to one on $\mathfrak C$, so the assertions follow from
Theorem~\ref{T:biperfect-parameters}.
\end{proof}

An admissible pair with $\eta\llcurly_\rho\delta$ is an isolated
unmarked vertex and gives a free coefficient by itself. Equal string
lengths can instead tie together coefficients belonging to different
basis elements. Symmetries impose further equalities among the
parameters.

\begin{corollary}\label{C:biperfect-parameter-symmetry}
Suppose a group $\Gamma$ acts by algebra automorphisms, preserves
$B$, and induces an action on $\Theta$ preserving the order
$\prec_{\mf T}$, both weight gradings, and the triples $(r_a,l_a,\rho_a)$ up to
permutation of $a\in\mathcal I$. Then the $\Gamma$-invariant
biperfect bases are parametrized by
$\prod_{\mathcal O\in\mathfrak C/\Gamma}\k$.
\end{corollary}
\begin{proof}
The action preserves the edges and marked pairs of the coefficient
graph. Since $B$ is invariant, a change of basis is equivariant
exactly when $t_{\gamma C}=t_C$. Thus the parameters are constant
on $\Gamma$-orbits.
\end{proof}

\begin{remark}\label{r:coefficient-certificates}
To determine a coefficient, one follows its component in the graph.
A path to a marked pair proves that it vanishes. On a finite set of
pairs closed under all edges, this is a finite calculation. The
finiteness of $D_\delta$ bounds the possible coefficients in a
single basis element, while deciding that a parameter is free requires
control of its whole component.
\end{remark}

For $\uca(\Delta)=\k[U]$, Proposition~\ref{T:allcr} gives a
transition criterion for the BK-biperfect bases considered in
\cite[Question~1.14]{Ka2}. Theorem~\ref{T:biperfect-parameters}
classifies the compatibly pointed bases by independent parameters,
giving the pointed refinement related to
\cite[Question~6.8(1)]{Ka2}.

\section{Examples} \label{S:example}
For the classical examples we use the QPs and potentials specified
in \cite{Fart,FW}; the rank-one base-affine case is checked directly.
The reference-model results apply when the corresponding full QPs are
related to valid reference QPs by QP mutation or an
idempotent-preserving Jacobian-algebra isomorphism.
Section~\ref{ss:G/U} instead compares the local boundary
derivatives on the full base-affine QP directly with the classical group
action. The canonical
family in Section~\ref{ss:exact} uses the explicit relation potential
of Definition~\ref{D:canonical-QP}.
\subsection{Unipotent Groups} \label{ss:unipotent} 
Let $Q$ be a Dynkin quiver, and $\br{Q}$ be its underlying graph.
Let $G$ be a simple, simply-connected, algebraic group over $\mb{C}$ of type $\br{Q}$, 
and $U$ be a maximal unipotent subgroup of $G$.
Recall from \cite{BFZ,GLSa} that $\mathbb C[U]$ is the upper cluster algebra of $\Delta_Q$,
where $\Delta_Q$ is the Auslander-Reiten quiver of $Q$ with translation.
The dual canonical basis of $\mathbb C[U]$ is indexed by the crystal $\mc B(\infty)$.
For a quiver $Q$, let $A_Q(u,v)$ count the arrows $u\to v$.

\begin{proposition} \label{P:unipotent} Let $I$ be the set of all frozen vertices of $\Delta_Q$. 
	\begin{enumerate}
		\item There is a unique integral compatible grading adapted to $I$ given by 
		$$ \wt_i = \dv E_i - \dv \tau^{-1} E_i + \dv \tau^{-2} E_i. $$
		\item $C_I(i,j) = -A_{\br{Q}}(i,j)$ for $i\neq j$ so the Cartan type of $I$ is given by the Cartan matrix of $\br{Q}$.
		\item There is an opposite cluster automorphism whose associated
Kashiwara map is the classical Kashiwara involution on
$\mc B(\infty)$.
	\end{enumerate}
\end{proposition}
\begin{proof}
Part~(1) is proved in \cite{Fart}, and (2) follows by direct
calculation. For (3), let $\Pi$ be the Dynkin preprojective algebra
and let $T=\bigoplus_v T_v$ be the standard cluster-tilting module
realizing $\Delta_Q$. Write $D$ for the contravariant duality of
$\op{mod}\Pi$ that interchanges opposite arrows. The GLS
reduced-word construction shows that $DT$ is reachable from $T$ by
mutations \cite[Proposition~6.1]{GLSp},
\cite[proof of Theorem~17.4]{GLSk}.
Since $\End_\Pi(DT)\cong\End_\Pi(T)^{\op{op}}$, such a mutation
sequence, followed by the identification $DT_v\leftrightarrow T_v$,
gives an opposite cluster automorphism $\sigma$.

Let $Z$ be an irreducible component of a preprojective representation
variety. Write $\delta_T(Z)$ and $\check\delta_T(Z)$ for its
projective and injective presentation weights under the generic
cluster-character correspondence \cite[Theorem~4]{GLSg}, \cite{P}.
For a general $M\in Z$, let
$0\to M\to T^0\to T^1\to0$ and
$0\to T_1\to T_0\to M\to0$ be minimal left and right
$\operatorname{add}T$-approximation sequences. In our convention,
$\delta_T(Z)=[T^1]-[T^0]$ and
$\check\delta_T(Z)=[T_1]-[T_0]$, the negatives of the
coindex and index. The brackets include the frozen summands of $T$.
Duality interchanges the two approximation sequences, whence
$$ \sigma\bigl(\delta_T(Z)\bigr) =\delta_{DT}(Z)=\check\delta_T(DZ). $$
By the definition of $\vee$, we also have
$\vee(\delta_T(Z))=\check\delta_T(Z)$.
Consequently,
$$ \kappa\bigl(\delta_T(Z)\bigr) =\vee^{-1}\sigma\bigl(\delta_T(Z)\bigr) =\delta_T(DZ). $$
Under the Kashiwara--Saito realization, preprojective duality
$Z\mapsto DZ$ is Kashiwara's involution
\cite[Theorem~5.3.2]{KS}, \cite[Section~5.1]{BKmv}.
This proves (3).
\end{proof}

By Theorem~\ref{T:upper}, $\trop(\Delta_Q,\S)$ carries an upper
seminormal crystal structure of Cartan type $C_{\br Q}$.
The grading in (1) comes from conjugation by a maximal torus of $G$
\cite{Fart}. The derivations $R_i$ act by left simple-root
translations:
\begin{equation}\label{eq:unipotent-root-action}
 R_i f(g)=\left.\frac{d}{ds}\right|_{s=0}f(x_i(s)g),
 \qquad x_i(s)=\exp(se_i).
\end{equation}
Write $\Delta_{\lambda,\mu}(g)=(v_\lambda,g v_\mu)$ for the
normalized generalized minors, using the standard contravariant
form on each fundamental representation. Choose a reduced word
$\mathbf j$ of $w_0$ beginning with $i$. After Chevalley transpose, the dual reduced-word seed of
\cite[Theorems~6.12 and~6.15]{FOstring} has frozen variables
$X_b=\Delta_{\varpi_b,w_0\varpi_b}|_U$; its other variables are
$\Delta_{w_p\varpi_b,w_0\varpi_b}|_U$, where
$w_p=s_{j_1}\cdots s_{j_p}$ is a nonempty prefix.
The $i$-th frozen column has a single entry $1$, at
$Y_i=\Delta_{s_i\varpi_i,w_0\varpi_i}|_U$, and zeros elsewhere.
If $i$ occurs only once, this column is zero and $Y_i=1$.
Thus $E_i=S_i$, and \eqref{eq:Ri}, together with
Lemma~\ref{L:boundary-propagation}, gives $R_i=Y_i\partial_{X_i}$.

Let $\mathcal D_i$ denote the derivative on the right of
\eqref{eq:unipotent-root-action}. Under the standard contravariant
form, $\mathcal D_i$ applies $f_i$ to the extremal vector in
the first entry of each minor. Since every nonempty prefix begins
with $i$, we have $w_p^{-1}\alpha_i<0$, and hence
$f_i v_{w_p\varpi_b}=0$. Also $f_i v_{\varpi_b}=0$ for
$b\ne i$, while $f_i v_{\varpi_i}=v_{s_i\varpi_i}$.
Therefore $\mathcal D_i$ kills every variable of this seed except
$X_i$, and $\mathcal D_i(X_i)=Y_i$. This proves
\eqref{eq:unipotent-root-action} on a transcendence basis, hence on
$\mathbb C[U]$. Together with Theorem~\ref{T:genericCR} and the uniqueness of the
crystal of a perfect basis \cite[Proposition~3.6]{FOperfect}, this
identifies the boundary crystal with $\mathcal B(\infty)$.

\subsection{Simple Canonical Models} \label{ss:simple}
We will construct for any acyclic quiver $Q$ a quiver with potential such that
its tropical points carry an upper seminormal crystal structure of type $C_{\br{Q}}$.

\begin{definition}	The {\em simple canonical QP}
	$(\Omega_Q,\Omega_\S)$ of a quiver with	potential $(Q,\S)$ is the extension of
	$(Q,\S)$ by simple representations $S_i$ for $i\in Q_0$.
\end{definition}
\indent We now consider the case where $Q$ is acyclic, so that $\S=0$.
By definition the $B$-matrix of $\Omega_Q$ is equal to $(B_Q, -E_Q^\t)$ where $E_Q=I_{n}-A_Q$ is the Euler matrix of $Q$ and $n=|Q_0|$.
Let $I$ be the set of all frozen vertices of $\Omega_Q$, which is just a copy of $Q_0$.
By construction, the Cartan type of $I$ is clearly $C_{\br{Q}}$.
We put the weight function $\wt_i$ given by the columns of the following matrix:
$$\left(\begin{matrix}I_n \\ {E_Q^\t}^{-1}B_Q  \end{matrix}\right) (I_n + {E_Q^\t}^{-1} E_Q).$$

\begin{proposition} With the above weight function $(\wt_i)_{i\in I}$,
	the set $\trop(\Omega_Q, \Omega_{\S})$ has an upper seminormal crystal cluster structure of type $C_{\br{Q}}$.
\end{proposition}
\begin{proof}Since $Q$ is acyclic, the Jacobian algebra of $(\Omega_Q,\Omega_\S)$ is finite-dimensional:
	the Jacobian relations kill all paths passing internally through frozen vertices. 
	Each frozen	vertex is reachable, since mutation at $i$ sends the extracted $S_i$ to $S_i^-$.
	Also, $E_Q$ is unimodular, so the above grading is integral.
	By Theorem \ref{T:upper} we only need to verify that $\wt$ is a compatible grading and adapted to $I$.
	Note that $B_Q = -E_Q + E_Q^\t$. It is immediate that
	$$(B_Q, -E_Q^\t) \left(\begin{matrix}I_n \\ {E_Q^\t}^{-1}B_Q  \end{matrix}\right) (I_n + {E_Q^\t}^{-1} E_Q) = {\rm O}.$$
By Corollary~\ref{C:HomEij}, $(\ep_i)_{i\in I}=({\rm O},E_Q^\t)$,
with the vectors arranged as rows. Equation~\eqref{eq:delta2dual} then gives
$$(\epc_i)_{i\in I} = ({\rm O}, E_Q^\t) + (I_n, I_n)\left(\begin{matrix}B_Q & -E_Q^\t \\ E_Q & {\rm O} \end{matrix}\right) = (E_Q^\t , {\rm O}).
$$
Finally we verify that $(\wt_i)_{i\in I}$ is adapted to $I$, namely,
$$(E_Q^\t , {\rm O})\left(\begin{matrix}I_n \\ {E_Q^\t}^{-1}B_Q  \end{matrix}\right) (I_n + {E_Q^\t}^{-1} E_Q) = E_Q^\t +E_Q = C_{\br{Q}}.$$
\end{proof}
\noindent This construction is closely related to the acyclic cluster algebras arising from
Kac--Moody unipotent subgroups in \cite[Section~16]{GLSk}.
In particular, their construction for the square of a Coxeter element gives an acyclic
cluster algebra with $n$ noninvertible coefficients. 
Readers can easily generalize this example to construct other crystals of Kac-Moody type.

\subsection{Base Affine Spaces} \label{ss:G/U}
Let $G$ be the simply-connected group of simply-laced Dynkin type $\br Q$, with
opposite maximal unipotent subgroups $U^-,U$ and maximal torus $T$.
We use the full base-affine QP of \cite[Section~9.1]{Fart}, including
its frozen arrows, and the convention
$$ \mathcal A=U^-\backslash G,\qquad A=\mathbb C[\mathcal A]=\bigoplus_{\lambda\in P_+}L(\lambda). $$
The action is right translation, and the summand of highest-weight
degree $\lambda$ consists of functions satisfying $f(tg)=\lambda(t)f(g)$.
Fix a pinning and normalized generalized minors. Let $D_i^+,D_i^-$
be the derivations induced by right multiplication by
$\exp(te_i),\exp(tf_i)$, respectively.
By \cite[Theorem~9.2]{Fart}, $A$ is the graded upper cluster algebra
of $\Delta_Q^\sharp$, and its generic basis is indexed by the integral
points of $\mr G_Q^\sharp$.

Assume first that $|Q_0|\ge2$. A vertex $v$ represents a minimal
presentation $X_v:P(\beta_1(v))\to P(\beta_0(v))$ over $\mathbb C Q$.
Put $p_i=(P_i\to0)$ and $q_i=(0\to P_i)$.
Writing $\varpi(a)=\sum_i a(i)\varpi_i$, for a vertex $v$ on the
translation chain from $p_j$ to $q_{j^*}$ the seed degrees are
$$\sigma_{\mathrm{hw}}(v)=\varpi_j,\qquad \sigma_{\mathrm{geom}}(v)=\varpi\bigl(\beta_1(v)-\beta_0(v)\bigr).
$$
They define the compatible gradings
$$\deg_{\mathrm{hw}}(\delta)=-\sum_v\delta(v)\sigma_{\mathrm{hw}}(v),\qquad
\wt_{\mathrm{geom}}(\delta)=-\sum_v\delta(v)\sigma_{\mathrm{geom}}(v).
$$
Thus $A_{\lambda,\gamma}=L(\lambda)_\gamma$.
We index the boundary representations by projectives, writing
$E_i=E_{q_i}$ and $F_i=E_{p_i}^\star$.

\begin{lemma}[\cite{Fart}]\label{L:base-boundary-pair}
We have $\tau^{-1}E_i\cong F_i$. Equivalently, in the original
indexing, $E_{i^*}\cong\tau E_{\ibar}^\star$.
\end{lemma}

The boundary quotients and the path relations of
\cite[Section~5.1]{Fart} give
$$ \dim E_i(v)=\beta_0(v)(i),\qquad \dim F_i(v)=\beta_1(v)(i), $$
and the minimal injective copresentation
$0\longrightarrow E_i\longrightarrow I_{v_i}\xrightarrow{c_i}I_{w_i}$,
where $v_i$ represents $S_i$ and $c_i:w_i\to v_i$ is its translation
arrow. Hence $\check\epsilon_i=e_{v_i}-e_{w_i}$. Since $v_i,w_i$
belong to the same translation chain, the boundary grading satisfies
\begin{equation}\label{eq:base-grading-comparison}
 \deg_{\mathrm{hw}}(\check\epsilon_i)=0,\qquad
 \wt_{\mathrm{geom}}(-e_v)=\sigma_{\mathrm{geom}}(v).
\end{equation}
The degree-zero assertion persists under mutation. Indeed, for a
compatible seed grading $\zeta$, the injective-weight formula in
Lemma~\ref{L:gdmu} gives
$$ \sum_v\check\delta'(v)\zeta'(v)-\sum_v\check\delta(v)\zeta(v)=-\betac_+(k)\sum_v b_{kv}\zeta(v)=0. $$
Apply this to $E_i$, using Lemma~\ref{L:bdinv}. Since
$\delta_{F_i}=-\check\epsilon_i$ by Lemma~\ref{L:base-boundary-pair},
its highest-weight degree is also zero in every seed.

\begin{lemma}\label{L:base-grading-comparison}
For the standard base-affine QP of connected simply-laced Dynkin
type of rank at least two, each simple root $i$ has reduced-word seeds
in which $E_i=S_{q_i}$ and $F_i=S_{p_i}$, respectively, and
$$ D_i^+=\left(\prod_{u\to q_i}x_u\right)\frac{\partial}{\partial x_{q_i}},\qquad D_i^-=\left(\prod_{p_i\to u}x_u\right)\frac{\partial}{\partial x_{p_i}}. $$
Each product is a cluster variable. The products
include all arrows of the mutated full QP.
\end{lemma}
\begin{proof}
For a reduced word $\mathbf j=(j_1,\ldots,j_N)$ of $w_0$, put
$w_k=s_{j_1}\cdots s_{j_k}$. The seed associated with $\mathbf j$ has variables
$$ x_{-a}=\Delta_{\varpi_a,w_0\varpi_a},\qquad x_k=\Delta_{\varpi_{j_k},w_0w_k\varpi_{j_k}}\quad(1\le k\le N). $$
These seeds are mutation-related to the standard seed
\cite[Sections~6.3 and~9.1]{Fart},
\cite[Theorem~2.10 and Remark~2.14]{BFZ}; we perform the mutations
on the full QP.

Choose $j_1=i^*$, and set
$$ X_i=x_{-i^*}=\Delta_{\varpi_{i^*},w_0\varpi_{i^*}},\qquad Y_i=x_1=\Delta_{\varpi_{i^*},s_iw_0\varpi_{i^*}}. $$
The arrow rule of \cite[Definition~2.2]{BFZ} gives exactly one arrow
between $q_i$ and a mutable vertex, namely $Y_i\to X_i$; thus
$E_i=S_{q_i}$. For every nonempty prefix,
$(w_0w_k)^{-1}\alpha_i>0$, so $e_i$ kills the corresponding
extremal vector. It also kills the lowest fundamental vectors of
types $a\ne i^*$. The normalized rank-one identity gives
$D_i^+X_i=Y_i$, and therefore $D_i^+=Y_i\partial_{X_i}$.

For the negative root, choose $j_N=i$ and let $k$ be the preceding
occurrence of $i$. Set
$$ X_i^-=x_N=\Delta_{\varpi_i,\varpi_i},\qquad Y_i^-=x_k=\Delta_{\varpi_i,s_i\varpi_i}. $$
The only arrow between $p_i$ and a mutable vertex is
$X_i^-\to Y_i^-$, so $F_i=S_{p_i}$. For $l<N$, the reduced
suffix $w_0w_l$ begins with $s_i$, whence
$(w_0w_l)^{-1}\alpha_i<0$. Thus $f_i$ kills these extremal vectors
and all the lowest fundamental vectors, while $D_i^-X_i^-=Y_i^-$.
Hence $D_i^-=Y_i^-\partial_{X_i^-}$.

To account for the frozen arrows, the simple injective and projective
presentations give
$$ \deg_{\mathrm{hw}}\prod_{u\to q_i}x_u=\deg_{\mathrm{hw}}X_i=\varpi_{i^*},\qquad \deg_{\mathrm{hw}}\prod_{p_i\to u}x_u=\deg_{\mathrm{hw}}X_i^-=\varpi_i, $$
by the degree-zero identities above. The known factors $Y_i$ and
$Y_i^-$ already have these degrees, respectively. Every additional
factor would be frozen and have nonzero dominant highest-weight
degree. Thus the two products are exactly $Y_i$ and $Y_i^-$.
\end{proof}

\begin{remark}\label{r:base-rank-one}
For type $A_1$, the boundary derivations on $\mathbb C[p,q]$ are
$R=p\partial_q,\ L=q\partial_p$, and $H=p\partial_p-q\partial_q$. 
They are the classical $\mathfrak{sl}_2$-action. Both variables have
highest-weight degree one, with internal weights $1,-1$, so the
degree-$m$ component has crystal $\mathcal B(m)$.
\end{remark}

\begin{corollary}\label{C:G/U}
For the standard full base-affine QP of simply-laced Dynkin type,
the local boundary rules define derivations of $\mathbb C[\mathcal A]$
independently of the optimized seeds. They coincide with the
infinitesimal right $G$-action:
$$ R_i=D_i^+,\qquad L_i=R_{p_i}^\star=D_i^-. $$
The generic basis is perfect for this action. Consequently, the
highest-weight-degree part
$$ \trop(\Delta_Q^\sharp,\S)^{\mathrm{hw}}_\lambda:=\{\delta\in\trop(\Delta_Q^\sharp,\S):\deg_{\mathrm{hw}}(\delta)=\lambda\} $$
is isomorphic to $\mc B(\lambda)$, and its internal-weight-$\gamma$
fiber has cardinality $\dim L(\lambda)_\gamma$. 
\end{corollary}
\begin{proof}
Rank one is Remark~\ref{r:base-rank-one}. Otherwise, use the positive reduced-word seed of Lemma~\ref{L:base-grading-comparison}. The generic
characters form a basis by \cite[Theorem~9.2]{Fart}, and the exchange
matrix has full rank by \cite[Proposition~2.6]{BFZ}.
The generic characters are pointed in this seed, with coefficient one
at $\mathbf x^{-\eta}$. Since no mutable arrow leaves $q_i$,
every term has $X_i$-degree at most $-\eta(q_i)=\rho_i(\eta)$,
and the pointed monomial attains this degree. Full rank makes
dominance a partial order. In a finite nonzero basis expansion,
choose a dominance-maximal index among those with largest $\rho_i$.
Its pointed monomial cannot cancel: an index of smaller $\rho_i$
cannot contribute at that $X_i$-degree, and another index of the
same $\rho_i$ could contribute only from above in dominance.
Consequently,
$$ \{f\in A:\deg_{X_i}f\le m\}=\operatorname{span}_{\mathbb C}\{C_{\gen}(\eta):\rho_i(\eta)\le m\},\qquad m\ge0. $$
Here $\rho_i(\eta)=-\eta(q_i)$; the space is zero for $m<0$.

We apply the local character argument of Lemma~\ref{L:character-strings} to
$D_i^+=Y_i\partial_{X_i}$. For $n=\rho_i(\delta)>0$,
Theorem~\ref{T:rlrigid} gives a general module $N=M/K$ of weight
$r_i\delta=\delta+\check\epsilon_i$, where
$K=\operatorname{im}(\tau^{-1}S_{q_i}\to M)$. The theorem applies
because an auxiliary mutation makes $S_{q_i}$ negative.
Since $\tau^{-1}S_{q_i}$ is generated at vertices entering $q_i$,
quotients vanishing there kill $K$. Their quotient Grassmannians
therefore agree for $M$ and $N$. The leading terms cancel, while
all remaining terms have $X_i$-degree at most $n-2$. Thus
$$ D_i^+C_{\gen}(\delta)-nC_{\gen}(r_i\delta)\in\operatorname{span}_{\mathbb C}\{C_{\gen}(\eta):\rho_i(\eta)<n-1\}. $$
At $n=0$ the derivative vanishes. The negative reduced-word seed gives
the dual argument for $D_i^-$, whose partial operator is $l_i$
because $F_i=\tau^{-1}E_i$. Mutation covariance of the boundary
operators and generic characters identifies the same perfect-basis
operators in every seed.

In any seed with $E_i=S_{q_i}$, perfectness and
$\rho_i(-e_v)=\updelta_{v,q_i}$ now give
$$ D_i^+x_v=0\quad(v\ne q_i),\qquad D_i^+x_{q_i}=C_{\gen}(-e_{q_i}+\check\delta_{S_{q_i}})=\prod_{u\to q_i}x_u. $$
No lower term is possible. These are the boundary rules on a
transcendence basis, so $R_i=D_i^+$ independently of the seed.
Dually, $L_i=D_i^-$. The degree-$\lambda$ generic basis is therefore
a perfect basis of the classical module $A_\lambda=L(\lambda)$.
The crystal and weight-multiplicity assertions follow from
\cite[Main Theorem~5.55]{BK}. 
\end{proof}

The isomorphism $\psi(U^-g)=g^{-1}\dot w_0U$ identifies this with
the usual left action on $\mathbb C[G/U]$: for
$\partial_Xf(hU)=\left.\frac{d}{dt}\right|_{t=0}f(\exp(-tX)hU)$,
$$ D_X\psi^*f=\psi^*(\partial_Xf). $$
Highest-weight degree $\lambda$ then corresponds to right-torus degree
$\lambda^*=-w_0\lambda$. The Weyl group preserves each
highest-weight-degree component and acts on it as on $\mc B(\lambda)$.

Write $\mathsf G_Q^\sharp(\mu,\gamma)$ for the fiber with
highest-weight degree $\mu$ and internal weight $\gamma$. For dominant
integral weights $\lambda,\mu,\nu$, let $\lrcoef$ be the multiplicity of $L(\lambda)$ in
$L(\mu)\otimes L(\nu)$. The following corollary recovers the
tensor-product multiplicity formula of Berenstein--Zelevinsky
\cite{BZ} in cluster coordinates.
\begin{corollary}
For the standard base-affine seed, $\lrcoef$ is the number of
lattice points in
$$ \mathsf G_Q^\sharp(\mu,\lambda-\nu)\cap \{\delta:\rho_i(\delta)\le\nu(h_i)\text{ for all }i\in I\}. $$
\end{corollary}
\begin{proof}
The PRV description \cite{PRV} gives
$$ \lrcoef=\dim\{v\in L(\mu)_{\lambda-\nu}:R_i^{\nu(h_i)+1}v=0\text{ for all }i\in I\}. $$
Corollary~\ref{C:G/U} supplies a homogeneous perfect basis for this
classical action. Its simultaneous kernels are indexed by the stated
inequalities, by the perfect-basis kernel argument of
Corollary~\ref{C:string-kernels-hw}.
\end{proof}
\noindent The functions $\rho_i$ also occur in \cite[Appendix~B]{GS}.

\subsection{Grassmannians} \label{ss:Grass} Let $\Gr\sm{n \\k}$ be the Grassmannian of $k$-planes in $\mb{C}^n$.
Recall that the coordinate ring of the affine cone over $\Gr\sm{n \\k}$ is the cluster algebra of the following quiver $\Square:=\Square_{k,l}$, where $l=n-k$ and $k,l\ge2$.
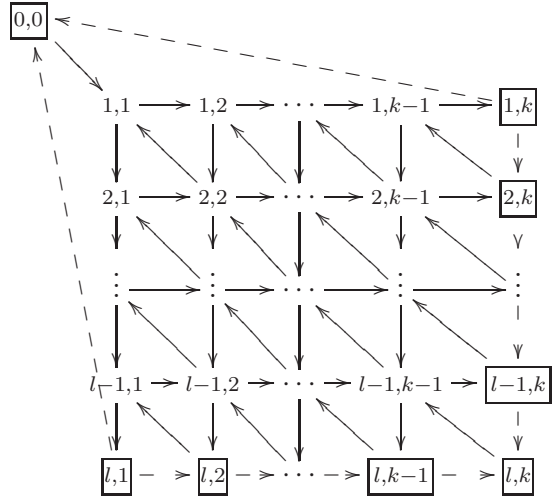
$$ \grass $$	
We also give another labelling on the frozen vertices:
$0=(0,0)$, $i=(i,k)$ for $1\le i\le l$, and
$l+j=(l,k-j)$ for $1\le j<k$, with labels in
$\mathbb Z/n\mathbb Z$.
The following proposition records the boundary calculations.
\begin{proposition} We have the following equalities:
\begin{enumerate}
	\item  $E_{m} = \tau E_{m+l \mod n}^\star$.
	\item  $\e(\mc{E}_{i,k}^\mu, \mc{E}_{i+1,k}^\mu)=\e(\mc{E}_{l,j}^\mu, \mc{E}_{l,j+1}^\mu)=1$ and $\e(\mc{E}_{1,k}^\mu, \mc{E}_{0,0}^\mu)=\e(\mc{E}_{l,1}^\mu, \mc{E}_{0,0}^\mu)=1$.
Here $1\le i<l$, $1\le j<k$, and all other ordered off-diagonal pairings vanish.
	\item  $\displaystyle \dv E_{i,k}=\sum_{r=0}^{\min(i,k)-1}e_{i-r,k-r},\quad
\dv E_{l,j}=\sum_{r=0}^{\min(l,j)-1}e_{l-r,j-r},\quad
\dv E_{0,0}=e_{0,0}+\sum_{(i,j)\in\Square_0^\mu}e_{i,j}.$
\end{enumerate}
In particular, the coefficient pattern of $\Square_{k,l}$ is exact and the Cartan type is the affine type $\tilde{A}_{n-1}$.
\end{proposition}
The coroot-coordinate functions
$\wt_m=\dv E_m-\dv\tau^{-1}E_m$ can be read off from (1) and (3).
They are the simple-coroot evaluations of the natural torus grading;
the central scalar degree is a separate grading.
\begin{remark}\label{r:affine-classical-weights}
The affine examples here use classical weights
$P_{\mathrm{cl}}=P/\mathbb Z\delta_{\mathrm{null}}$, not a full affine
weight map on the same tropical set. For example, for $\Gr(2,4)$,
label the frozen vertices $a,b,c,d$ and the mutable vertex $u$ so that
$a,c\to u\to b,d$. Then
$\check\epsilon_a+\check\epsilon_b+\check\epsilon_c+\check\epsilon_d=B_\Delta$.
Thus \eqref{eq:existwt} fails, while the classical relation
$\bar\alpha_a+\bar\alpha_b+\bar\alpha_c+\bar\alpha_d=0$ is
compatible with the same equality. No auxiliary compatible linear form
can recover this null-root direction on the original set. Affinization,
if desired, adds an extra integer coordinate and changes the underlying
crystal; it is not the completion in Lemma~\ref{L:wtc}.
\end{remark}
For an affine type $A$ geometric crystal on the Grassmannian and its tropicalization, see \cite{Fr}.
There are similar results for isotropic Grassmannians. Cluster constructions for partial flag varieties of types $A,D,E$ are given in \cite{GLSp}.
We will give below an example of an exceptional Grassmannian. 
In the notation of \cite{GLSp}, the group is of type $E_6$ and
the complementary index set for the parabolic is $\{1\}$.
\begin{example} Consider the following quiver $\Delta_6^\mu$ as the mutable part
	$$ \hivesix $$
	We extend it by the indecomposable rigid representations of the following weights:
	$$\xymatrix@R=3ex@C=5ex{  &\scriptstyle e_{1,1,4} - e_{2,1,3} \ar@{-->}[r] &\scriptstyle -e_{1,1,4}  \\
			\scriptstyle e_{2,2,2} \ar@{-->}[r] \ar@{-->}[dr] \ar@{-->}[ur]	 &\scriptstyle e_{1,4,1} - e_{1,3,2} \ar@{-->}[r] &\scriptstyle -e_{1,4,1}  \\
			&\scriptstyle e_{4,1,1} - e_{3,2,1} \ar@{-->}[r] &\scriptstyle -e_{4,1,1}  
	} $$
	with dashed arrows indicating the Cartan type.
	We can easily check that those representations satisfy \eqref{eq:tauexact}.
	In particular, the coefficient pattern of $\Delta$ is exact and the Cartan type is the affine type $\tilde{E}_{6}$.
\end{example}

\subsection{The Exact Coefficient Patterns} \label{ss:exact}
\begin{definition} A coefficient pattern with rigid frozen
vertices is \emph{exact} if, for every frozen vertex $i$, there is a
frozen vertex $\ibar$ such that $\tau^{-1}E_i=E_{\ibar}^\star$.
\end{definition}

\begin{proposition}\label{P:exact} The coefficient pattern of $(\Delta,\S)$ is exact if and only if
there is a permutation $i\mapsto\ibar$ of $\Delta_0^{\op{fr}}$ such that
$\mc E_i^\mu\cong\tau_\mu^2\mc E_{\ibar}^\mu$ for $i\in\Delta_0^{\op{fr}}$
and \eqref{eq:tauexact} holds for every $i,j\in\Delta_0^{\op{fr}}$.
\end{proposition}
\begin{proof}
The partner map is unique and injective, since distinct
boundary representations have different frozen supports and $\tau$
is invertible on decorated representations. It is therefore a
permutation of $\Delta_0^{\op{fr}}$. Apply Lemma~\ref{L:tauexact} to each pair.
\end{proof}
\noindent Thus exactness of a coefficient pattern can be checked using its mutable presentations.
We also call a tuple of rigid mutable presentations \emph{exact} if it satisfies the conditions of Proposition~\ref{P:exact}.

\begin{definition}\label{D:canonical-QP}
Let $r\ge2$ and $a_1,\ldots,a_r\ge2$. For $r\ge3$, choose
pairwise distinct nonzero parameters $\lambda_3,\ldots,\lambda_r$,
normalized by $\lambda_3=1$. Start with the quiver consisting of the
$r$ arms
$$
0\xrightarrow{a_{k,1}}1_k\xrightarrow{a_{k,2}}\cdots
\xrightarrow{a_{k,a_k-1}}(a_k-1)_k
\xrightarrow{a_{k,a_k}}\infty\qquad(1\le k\le r),
$$
and write $p_k=a_{k,1}\cdots a_{k,a_k}$. The canonical algebra has
minimal relations $p_s-p_2+\lambda_s p_1=0$ for $3\le s\le r$.
Adjoin one returning arrow
$c_s:\infty\to0$ for each relation, and set
\begin{equation}\label{eq:canonical-potential} W_{\mathrm{can}}=\sum_{s=3}^r(p_s-p_2+\lambda_s p_1)c_s.\end{equation}
We denote the resulting Jacobian algebra by
$J_{\mathrm{can}}(\mathbf a,\boldsymbol\lambda)$, and call it the
\emph{$2$-Calabi--Yau tilted algebra of canonical type}
$(\mathbf a,\boldsymbol\lambda)$. For $r=2$ the relation set and
the returning-arrow set are empty, and the potential is zero.
\end{definition}
\noindent This is the relation-potential construction of
\cite[Section~6.9 and Theorem~6.12]{KellerCY}; its quiver has one
returning arrow for each minimal relation, as in
\cite[Theorem~2.6]{ABS}. It is the endomorphism algebra of the
canonical tilting bundle in the cluster category of the corresponding
weighted projective line \cite[Sections~2--3]{BKL}; equivalently, it is
the relation extension of the canonical algebra
\cite[Theorem~3.1]{ASSqt}. In particular, it is finite-dimensional.
There are $r-2$ returning arrows. For $r=3$, rescaling arrows turns
\eqref{eq:canonical-potential} into the sum of the three arm cycles;
for more arms the parameters and the additional returning arrows are
part of the construction.

\begin{proposition}\label{P:can}
For $J_{\mathrm{can}}(\mathbf a,\boldsymbol\lambda)$, the simple
modules $S_{j_k}$, $1\le j<a_k$, together with
$$ T_k:=\coker\bigl(P_{1_k}\xrightarrow{a_{k,1}}P_0\bigr) \qquad(1\le k\le r) $$
form an exact tuple of mutable presentations of Cartan type
$$ \widetilde A_{a_1-1}\times\cdots\times\widetilde A_{a_r-1}. $$
Each $T_k$ is indecomposable rigid of weight $e_0-e_{1_k}$.
\end{proposition}
\begin{proof}
Put $V_{k,0}=T_k$ and $V_{k,s}=S_{s_k}$ for $1\le s<a_k$,
with the second index read modulo $a_k$. The canonical tilting-bundle
realization identifies these modules with the quasi-simples in the
exceptional torsion tubes \cite{R,BKL}; the defining presentation of $T_k$
follows from the corresponding line-bundle sequence. For torsion sheaves $X$,
$\Ext^1(\mathcal T,X)=\Ext^1(\mathcal T,\tau^{-1}X)=0$, so the
cluster functor restricts to the fully faithful canonical-algebra
functor. The torsion-tube Auslander--Reiten sequences survive the
cluster quotient: $\tau\mathcal T$ has no torsion summands
\cite[Propositions~2.1 and~2.4]{BKL}. Thus
$$ \tau_\mu V_{k,s}=V_{k,s+1},\qquad \hom(V_{k,s},V_{\ell,t})=\updelta_{k,\ell}\updelta_{s,t}. $$
Lemma~\ref{L:H2E} gives
$\e(V_{k,s},V_{\ell,t})=\updelta_{k,\ell}\updelta_{t,s+1}$.
In particular the modules are rigid and the Cartan matrix is the stated
product (a two-element cycle gives $\widetilde A_1$). Taking
$\overline{(k,s)}=(k,s-2)$, the two Hom-terms in
\eqref{eq:tauexact} are $\updelta_{i,j}$ and
$\updelta_{\ibar,j}$. These are precisely the conditions of Proposition~\ref{P:exact}.
\end{proof}

\noindent The quiver of canonical type $(2,3,6)$, as well as the
mutable quiver of the Grassmannian $\Gr\sm{9\\3}$, is
mutation-equivalent to $\Delta_6^\mu$. The three examples
thus have underlying quivers in the same mutation class and
exact tuples of Cartan types
$\widetilde A_8$, $\widetilde E_6$, and $\widetilde A_1\times\widetilde A_2\times\widetilde A_5$.
A tuple is \emph{maximal} if no further indecomposable rigid mutable
presentations can be adjoined, over its fixed mutable QP, while
preserving exactness.
Conjecturally these are the only three maximal exact coefficient
patterns on $\Delta_6^\mu$, up to mutation and frozen relabeling. 
These exact coefficient patterns give coroot-coordinate crystals.
A compatible full-weight realization additionally requires
\eqref{eq:existwt}.

\appendix
\section{Kashiwara's Data}\label{S:KD}
We first describe maximal crystal operators in terms of
$\E$-orthogonalizations. We then use them to compute Kashiwara's
data and study the adjunction formulas for $\E$-pairings.
For the standard unipotent models, we finally show how a single
cluster seed makes all string coordinates linear.

\subsection{Maximal Operators and Orthogonalizations}
For a basic rigid presentation $e$, denote by
$\widetilde{\sqcup}_e^\pm$ its \emph{unreduced}
$\E$-orthogonalizations \cite[Definition~8.2]{Fop}.
Removing the direct summands belonging to $\operatorname{add}(e)$
gives the reduced orthogonalizations $\sqcup_e^\pm$.
The unreduced maps keep these summands and commute with extended
mutation \cite[Proposition~8.1 and (8.3)]{Fop}.

Suppose that $e=d_{\mc E}$ is indecomposable and extended-reachable,
of weight $\epsilon$. For a general presentation of projective
weight $\delta$, and one of injective weight $\check\eta$, put
$$ s_e(\delta)=\e(\epsilon,\delta)+\e(\delta,\epsilon),\qquad \check s_e(\check\eta)=\e(\mc E,\mc N)+\e(\mc N,\mc E), $$
where $\mc N$ is general of injective weight $\check\eta$.
Define the maximal ambient operators by
$$
 r_\epsilon^{\max}(\delta)=r_\epsilon^{s_e(\delta)}(\delta),
 \qquad
 \check r_\epsilon^{\max}(\check\eta)
 =\check r_\epsilon^{\check s_e(\check\eta)}(\check\eta).
$$
For a supported weight and an upper boundary $E_i$, the first
exponent is $\rho_i(\delta)$; the dual boundary has the
analogous injective-string exponent.

\begin{proposition}\label{P:max-approximation}
Let $e=d_{\mc E}$ be indecomposable and extended-reachable.
For general presentations $d$ and $d_{\mc N}$ of projective
weight $\delta$ and injective weight $\check\eta$, respectively,
the unreduced orthogonalizations are general and have weights
$$
 \delta_{\widetilde{\sqcup}_e^-(d)}
 =r_\epsilon^{\max}(\delta),\qquad
 \check\delta_{\widetilde{\sqcup}_e^+(d_{\mc N})}
 =\check r_\epsilon^{\max}(\check\eta).
$$
\end{proposition}
\begin{proof}
Choose an extended mutation taking $e$ to $P_k[1]$, and let
$d'$ be the transformed presentation, of weight $\delta'$.
By \cite[Lemma~8.3]{Fop}, its negative unreduced orthogonalization
is the universal extension
$$
 0\longrightarrow d'\longrightarrow h^-
 \longrightarrow P_k[1]^{\oplus s}\longrightarrow0,
 \qquad s=\e(P_k[1],d')=s_e(\delta).
$$
Thus $h^-$ has weight $\delta'-s e_k$. At this seed,
$r_{-e_k}$ is translation by $-e_k$, so the weight equals
$r_{-e_k}^{s}(\delta')$. Genericity follows from
\cite[Remark~1.13 and Lemma~8.3]{Fop}; mutation covariance
(Theorems~\ref{T:genpi} and~\ref{T:rle}) gives the first formula.
The second follows by the dual construction with $e$ taken to
$P_k$.
\end{proof}

\begin{lemma}\label{L:LRrigid}
If $\mc E$ is indecomposable and extended-reachable and $\mc M$
is rigid, then both
$\mc E\oplus\widetilde{\sqcup}_{d_{\mc E}}^\pm(d_{\mc M})$
and $\mc E\oplus\sqcup_{d_{\mc E}}^\pm(d_{\mc M})$
are rigid.
\end{lemma}
\begin{proof}
At a seed where $e$ is projective or shifted projective, this is
rigidity preservation for the universal extension
\cite[Theorem~5.5 and Lemma~8.3]{Fop}. The assertion follows at
the original seed by mutation invariance. Removing summands of $\mc E$
preserves rigidity.
\end{proof}

For the following definition and Corollary~\ref{C:Rmaxcv}, assume
the hypotheses of Theorem~\ref{T:genericCR}. For a generic-basis
element, define
$$ R_i^{\max}C_{\gen}(\delta) =R_i^{(\rho_i(\delta))}C_{\gen}(\delta), $$
and define $(R_j^\star)^{\max}$ similarly, extending both maps
$\k$-linearly.
\begin{corollary}\label{C:Rmaxcv}
If $f$ is a generalized cluster monomial, then so are
$R_i^{\max}(f)$ and $(R_j^\star)^{\max}(f)$, for
$i\in I$ and $j\in I^\star$, respectively.
\end{corollary}
\begin{proof}
By BK-perfectness, maximal raising sends $C_{\gen}(\delta)$
to $C_{\gen}(r_i^{\max}\delta)$. Proposition~\ref{P:max-approximation}
and Lemma~\ref{L:LRrigid} show that the new weight is rigid.
The starred case is identical.
\end{proof}

\subsection{Kashiwara's Data and Adjoint Identities}
Let $I$ be a set of reachable frozen vertices. For a word
$\mathbf i=(i_1,\ldots,i_n)\in I^n$ and
$\delta\in\trop(\Delta,\S)$, define
$$ \delta^{(0)}=\delta,\qquad n_k=\rho_{i_k}(\delta^{(k-1)}),\qquad \delta^{(k)}=r_{i_k}^{n_k}(\delta^{(k-1)}). $$
\begin{definition}
The $\mathbf i$-\emph{Kashiwara data} of $\delta$ are
$\mathbf s_{\mathbf i}(\delta)=(n_1,\ldots,n_n)$.
\end{definition}
The data together with the terminal tropical point $\delta^{(n)}$
determine $\delta$. To compute them, choose mutations
$\boldsymbol\mu_k$ taking $E_{i_k}$ to $S_{i_k}$, and put
$\xi_k=\boldsymbol\mu_k(\delta^{(k-1)})$. Then
Lemma~\ref{L:mu2S} gives
$$ n_k=-\xi_k(i_k),\qquad \delta^{(k)}=\boldsymbol\mu_k^{-1} \bigl(\xi_k+n_k\check\delta_{S_{i_k}}\bigr). $$
By Proposition~\ref{P:max-approximation}, this is also the
unreduced orthogonalization at $E_{i_k}$.

A different question is whether a string coordinate can be expressed
as an $\E$-pairing with a fixed representation. The following
identity explains the restriction needed to transfer maximal raising
between the two arguments.

\begin{lemma}\label{L:adjoint}
Let $\delta\in\trop(\Delta,\S)$, let $E=E_i=S_i=P_i$, and
put $\epsilon=\delta_E$. For an injective weight $\check\eta$,
let $\eta$ be its corresponding projective weight. Then
$$ \hom(r_\epsilon^{\max}\delta,\check\eta) =\hom(\delta,\check r_\epsilon^{\max}\check\eta) $$
and
$$
 \e(r_\epsilon^{\max}\delta,\check\eta)
 -\e(\delta,\check r_\epsilon^{\max}\check\eta)
 =-\rho_i(\delta)\hom(\eta,\check\epsilon).
$$
In particular, if $e_i$ is not a summand of $\eta$, then
$$ \e(r_\epsilon^{\max}\delta,\check\eta) =\e(\delta,\check r_\epsilon^{\max}\check\eta). $$
\end{lemma}
\begin{proof}
Choose general decorated representations with ordinary parts $M,N$.
Since $P_i=S_i$, no arrow starts at $i$. Write
$$ A_N=\bigoplus_{a:u\to i}N(u),\qquad \alpha_N:A_N\to N(i),\qquad k_N=\dim\ker\alpha_N,\quad c_N=\dim\coker\alpha_N. $$
Let $g=(P_i\to\bigoplus_{a:u\to i}P_u)$ present
$\tau^{-1}E$. The ordinary parts $R,\check R$ of the two
unreduced orthogonalizations are, by \cite[Section~8.1]{Fop},
$$ R\cong L_gM,\qquad \check R\cong R_{\nu g}N\oplus S_i^{\oplus c_N}, $$
where $L_g,R_{\nu g}$ are the left and right perpendicular
projections \cite[Theorem~2.2]{Fop}. Their adjunction, together
with $M(i)=0$, gives the stated equality of Hom dimensions.
Put $\rho=\rho_i(\delta)=-\delta(i)$. The same formulas give
$$ \delta_R=\delta+\rho\check\epsilon,\qquad \dim\check R=\dim N+k_Ne_i,\qquad \check\epsilon(\dim N)=c_N-k_N. $$
The Hom terms cancel in \eqref{eq:heform}; consequently,
$$
\begin{aligned}
 \e(r_\epsilon^{\max}\delta,\check\eta)
 -\e(\delta,\check r_\epsilon^{\max}\check\eta)
 &=k_N\delta(i)-\rho(c_N-k_N)\\
 &=-\rho c_N.
\end{aligned}
$$
Finally, $c_N=\hom(N,S_i)=\hom(\eta,\check\epsilon)$.
A quotient onto the projective simple $S_i$ splits, so
$c_N=0$ precisely when $e_i$ is not a summand of $\eta$.
\end{proof}

\begin{corollary}\label{C:adjointe}
Let $i$ be reachable with $E_i=P_i$, put
$\epsilon=\epsilon_i$, and let
$\delta\in\trop(\Delta,\S)$. Suppose the projective weights
$\eta,\zeta$ corresponding to $\check\eta$ and
$\check r_\epsilon^{\max}(\check\eta)$ are frozen-supported,
and $\epsilon$ is not a summand of $\eta$. Then
$$ \e(r_\epsilon^{\max}\delta,\check\eta) =\e(\delta,\check r_\epsilon^{\max}\check\eta). $$
\end{corollary}
\begin{proof}
Mutate to a seed where $E_i$ is simple projective.
The frozen-supported presentation weights give mutation invariance
of both pairings by Lemma~\ref{L:HEmu}. The nonsummand condition is
also preserved under mutation. Apply Lemma~\ref{L:adjoint} and
Theorem~\ref{T:rle}.
\end{proof}

For a fixed word and position, write
$$ F_{\mathbf i,k}(\delta)=\rho_{i_k}(\delta^{(k-1)}). $$
One may ask for a representation $N_{\mathbf i,k}$, independent of
$\delta$, with
$F_{\mathbf i,k}(\delta)=\e(\delta,N_{\mathbf i,k})$ for every
supported $\delta$. Corollary~\ref{C:adjointe} supplies
sufficient conditions at each transfer. Any such $N_{\mathbf i,k}$
necessarily satisfies
$$ \dim N_{\mathbf i,k}(v)=F_{\mathbf i,k}(-e_v) \qquad(v\in\Delta_0). $$

A related question asks whether one mutation sequence can make all
these pairings simultaneously visible. Put
$$
 r_{k>}^{\max}=r_{i_{k-1}}^{\max}\cdots r_{i_1}^{\max},\qquad
 \check r_{<k}^{\max}=\check r_{\epsilon_{i_1}}^{\max}\cdots
   \check r_{\epsilon_{i_{k-1}}}^{\max},
$$
with empty products equal to the identity, and consider
\begin{equation}\label{eq:muadj}
 \e(r_{k>}^{\max}(\delta),\check\epsilon_{i_k})
 =\e\bigl(\mub(\delta),
       \mub(\check r_{<k}^{\max}(\check\epsilon_{i_k}))\bigr)
 \qquad(1\le k\le n).
\end{equation}
The formula fails for the word $(i,i)$: with $\delta=-e_i$
and $k=2$, the left side is zero, whereas rigidity and
Lemma~\ref{L:HomEinv} make the right side equal to
$\e(-e_i,E_i)=1$, for every mutation sequence $\mub$.
For a reduced word of $w_0$ in a standard unipotent model, it
remains open whether a single $\mub$, independent of $k$ and
$\delta$, makes \eqref{eq:muadj} hold for all supported weights.

\subsection{Linear String Coordinates}\label{ss:linear-string-coordinates}
For a standard unipotent model, the successive orthogonalizations
have a simultaneous linear expression in a suitable seed. We use
\eqref{eq:unipotent-root-action} to regard the boundary derivations
as left simple-root translations on $\mathbb C[U]$.

\begin{theorem}[Linear string coordinates]\label{T:string-linear}
Let $(\Delta,\S)$ be the standard unipotent model of a group of
finite simply-laced Dynkin type. For every reduced word
$\mathbf i=(i_1,\ldots,i_N)$ of $w_0$, where
$N=\ell(w_0)=|\Delta_0|$, there are a seed $t_{\mathbf i}$
in its mutation class and a matrix
$$ D_{\mathbf i}\in\operatorname{GL}_N(\mathbb Z) \cap\operatorname{Mat}_{N\times N}(\mathbb Z_{\ge0}) $$
such that
\begin{equation}\label{eq:string-linear}
 \mathbf s_{\mathbf i}^{\,t_{\mathbf i}}(\delta)
 =-\delta D_{\mathbf i}
 \qquad
 \bigl(\delta\in\trop(\Delta_{t_{\mathbf i}},\S_{t_{\mathbf i}})\bigr).
\end{equation}
The row of $D_{\mathbf i}$ indexed by $v$ is
$\mathbf s_{\mathbf i}^{\,t_{\mathbf i}}(-e_v)$.
\end{theorem}
\begin{proof}
Put $\Phi_{\mathbf i}(\mathbf z)=x_{i_1}(z_1)\cdots x_{i_N}(z_N)$,
where $x_i(z)=\exp(ze_i)$, and compare exponents lexicographically,
starting with the first coordinate. Apply the mutation construction
of \cite[Corollary~6.22, Theorem~6.25, and the proof of
Theorem~6.27]{FOstring} to the reversed word on $U^-$, and pull
back by the Chevalley transpose. It gives a seed $t_{\mathbf i}$
with
\begin{equation}\label{eq:string-monomial-seed}
 \Phi_{\mathbf i}^*x_v=\mathbf z^{d_v},\qquad
 \Phi_{\mathbf i}^*\check y_u=\frac{z_{q(u)}}{z_{p(u)}},
 \qquad p(u)<q(u).
\end{equation}
The matrix $D_{\mathbf i}$ with rows $d_v$ is unimodular, and
$t_{\mathbf i}$ lies in the standard mutation class
\cite[Theorem~5.4]{FOstring}. Here $B_\Delta=-\varepsilon$ in the notation of \cite{FOstring};
reversing the word reverses the order of its factorization parameters.
The cluster variables are regular on $U$, so their monomial
pullbacks under the polynomial map $\Phi_{\mathbf i}$ have
nonnegative exponents.

Fix a supported weight $\delta$ at $t_{\mathbf i}$.
Let $d^{(k)}$ be the successive unreduced negative
$\E$-orthogonalizations of a general presentation $d^{(0)}$
at $E_{i_1},\ldots,E_{i_k}$. Write $\delta^{(k)}$ for their
weights, put $f_k=C_{\gen}(\delta^{(k)})$, and set
$n_k=\rho_{i_k}(\delta^{(k-1)})$.
Proposition~\ref{P:max-approximation} and
Theorem~\ref{T:genericCR} give
$$ \delta^{(k)}=r_{i_k}^{n_k}(\delta^{(k-1)}),\qquad R_{i_k}^{(n_k)}f_{k-1}=f_k,\qquad R_{i_k}^{n_k+1}f_{k-1}=0. $$
Taylor expansion therefore gives
$$ [z_1^{n_1}\cdots z_k^{n_k}]\Phi_{\mathbf i}^*f_0 =\Phi_{(i_{k+1},\ldots,i_N)}^*f_k. $$
For the empty suffix the right side means evaluation at $1_U$.
Since $\mathbf i$ is reduced for $w_0$, the
$\mathcal B(\infty)$ identification in
Section~\ref{ss:unipotent} gives $\delta^{(N)}=0$ and $f_N=1$.
The coefficient at $(n_1,\ldots,n_N)$ is consequently one;
any larger exponent vanishes at its first differing coordinate.
Thus
$$ \max_{\mathrm{lex}}\operatorname{supp} \bigl(\Phi_{\mathbf i}^*C_{\gen}(\delta)\bigr) =\mathbf s_{\mathbf i}(\delta). $$
On the other hand, \eqref{eq:genCC} and
\eqref{eq:string-monomial-seed} give
$$
 \Phi_{\mathbf i}^*C_{\gen}(\delta)
 =\mathbf z^{-\delta D_{\mathbf i}}
 \check F_{\coker(\delta)}\!\left(
 \left(\frac{z_{q(u)}}{z_{p(u)}}\right)_u\right).
$$
Every nonconstant monomial in the last factor has a negative
lexicographic exponent, whereas its constant term is one.
Its greatest exponent is therefore $-\delta D_{\mathbf i}$,
which proves \eqref{eq:string-linear}. The assertion about the
rows follows by taking $\delta=-e_v$.
\end{proof}

If $\boldsymbol\mu_{\mathbf i}$ takes $t_0$ to
$t_{\mathbf i}$, the formula in the original seed is
$$ \mathbf s_{\mathbf i}^{\,t_0}(\delta) =-\boldsymbol\mu_{\mathbf i}(\delta)D_{\mathbf i}. $$

\begin{example}[Type $D_4$]\label{ex:string-linear-D4}
Number the simple roots so that $2$ is central, and take
$\mathbf i=(4,3,2,1)^3$. From the reduced-word seed for
$\mathbf j=(1,2,3,4)^3$, indexed by its twelve positions, mutate
successively at $(8,7,6,5,4,8,3,7,2,6,1,5)$
and transpose to $U$. Write $x_1,\ldots,x_{12}$ for the resulting
variables; the frozen vertices $9,10,11,12$ correspond to roots
$1,2,3,4$, respectively. Mutation at vertex $2$ gives
$$ x_2x'_2=x_1x_7x_8+x_3x_4x_6,\qquad \delta_{x'_2}=e_2-e_3-e_4-e_6. $$
Direct substitution gives
$$
 \Phi_{\mathbf i}^*\check y_2=\frac{z_7}{z_3},\qquad
 \Phi_{\mathbf i}^*\!\left(\frac{x_3x_4x_6}{x_2}\right)
 =z_3z_4z_7z_9z_{10},
$$
so $\Phi_{\mathbf i}^*x'_2=z_4z_7z_9z_{10}(z_3+z_7)$.

One can also compute its string directly by Lemma~\ref{L:mu2S}.
From this seed, optimizing sequences for vertices $9,10$ are
$(5,6,7,8,6,2)$ and $(1,2,3,4,2)$, respectively; vertices
$11,12$ are already optimized. In the coordinates of this seed,
the nonzero maximal raising steps are
$$
\begin{array}{c|c|c}
 k&n_k&\delta^{(k)}\\ \hline
 3&1&-e_6\\
 4&1&-e_2+e_7+e_8\\
 7&1&e_3+e_4-e_{11}-e_{12}\\
 9&1&e_3-e_{11}\\
 10&1&0
\end{array}
$$
All other $n_k$ vanish. Hence
$$ \mathbf s_{\mathbf i}(\delta_{x'_2}) =(0,0,1,1,0,0,1,0,1,1,0,0), $$
as the greatest exponent of $\Phi_{\mathbf i}^*x'_2$ also shows.
Proposition~\ref{P:max-approximation} identifies the same table
with the successive unreduced $\E$-orthogonalizations.
\end{example}

\section{Transitions and String Profiles}\label{S:string-profiles}
We prove the change-of-basis criterion used in Section~\ref{S:biperfect}
and compare the string data of arbitrary BK-biperfect bases. Use the
notation $(r_a,l_a,\rho_a,D_a)$, $a\in\mathcal I=I\sqcup I^\star$,
of Section~\ref{ss:perfect-transitions}.

\subsection{Changes of BK-biperfect Bases}\label{ss:perfect-transition-proof}
Let $\Psi$ and $T$ be as in \eqref{eq:crtriangular}.
Since $B$ is adapted to each string filtration,
$\Psi(K_{a,n})=K_{a,n}$ for every $a,n$ is equivalent to both
$T$ and $T^{-1}$ having nonzero entries only at pairs with
$\eta\preceq_\rho\delta$. Both matrices are row-finite, since each basis element has a
finite expansion in the other basis.

\begin{proof}[Proof of Proposition~\ref{T:allcr}]
Two BK-biperfect bases are adapted to the same kernel filtrations by
Corollary~\ref{C:string-kernels-hw}; hence their transition preserves each
filtration with equality. Assuming this filtration condition, fix
$a$, put $D=D_a$, $K_n=K_{a,n}$, and let
$n=\rho_a(\delta)>0$. Perfectness of $B$ gives
$$ DB'(\delta)\equiv n\sum_{\rho_a(\eta)=n}T_{\delta,\eta}B(r_a\eta) \pmod{K_{n-2}}. $$
For $\rho_a(\theta)=n-1$, the coefficient of $B(\theta)$
in $DB'(\delta)$ is $nT_{\delta,l_a\theta}$, since $r_a$ and
$l_a$ are partial inverses. Thus
$DB'(\delta)-nB'(r_a\delta)\in K_{n-2}$ is equivalent to
\begin{equation}\label{eq:all-perfect-recursion}
T_{r_a\delta,\theta}=T_{\delta,l_a\theta}
\quad\text{if }\rho_a(\delta)>0,\qquad
\rho_a(\theta)=\rho_a(\delta)-1,
\end{equation}
where entries involving $0_{\mathrm{cr}}$ are zero. These equations express
commutation of the maps induced by $\Psi$ and $D_a$ on the
associated graded spaces. Since
$\Psi(K_{n-2})=K_{n-2}$, the remainder has an expansion in $B'$
using only smaller string lengths. For $n=0$, filtration
preservation gives $DB'(\delta)=0$. This proves both directions.
\end{proof}

\begin{remark}\label{r:transition-graded}
Taking $\theta=r_a\eta$ in \eqref{eq:all-perfect-recursion} gives
$$ T_{\delta,\eta}=T_{r_a\delta,r_a\eta} \quad\text{if }\rho_a(\delta)=\rho_a(\eta)>0. $$
If instead $l_a\theta=0_{\mathrm{cr}}$, the same recursion forces
$T_{r_a\delta,\theta}=0$. These are the two kinds of relations
used in the coefficient graph. The commutation condition concerns the
associated graded maps; it allows $\Psi D_a-D_a\Psi$ to have
terms of smaller string length.
\end{remark}

\begin{remark}\label{r:multiplicity-normalization}
Adjoin an isolated frozen variable $z$, of weight zero for both
gradings and outside $I\cup I^\star$. The elements
$B(\delta)z^m$ form a BK-biperfect basis. A linear automorphism of
$\k[z]$ fixing $1$, tensored with the identity, gives another
such basis. For example, one may replace $z^m$ by $(-1)^mz^m$,
or replace $z$ by $z+z^2$ while fixing the other monomials.
This illustrates the freedom in normalizing components with identical
string data. Compatible pointedness fixes their leading monomials.
\end{remark}

The transition criterion also describes bases preserved by a symmetry.
Let $\sigma=(\mub,\pi)$ be a cluster automorphism. Expressing a
function in the mutated cluster $\mathbf x'$ and substituting
$x'_v\mapsto x_{\pi(v)}$ defines its action on $\uca(\Delta)$.
To obtain a self-map of $\mc B$, assume
\eqref{eq:direct-supported-return} in the direct case or
\eqref{eq:opposite-supported-return} in the opposite case. Corollary~\ref{C:directauto} applies in the direct case; in the
opposite case, we compose the tropical map with $\vee^{-1}$ as in
Definition~\ref{D:Kmap}. We also write
$\sigma:\mc B\to\mc B$ for the resulting map.
Invariance of a labeled basis below means equivariance for this
specified map on labels, not merely preservation of the basis as an
unlabeled set.

\begin{corollary}\label{C:sigmacr}
Suppose $B$ is $\sigma$-invariant. Another BK-biperfect basis
$B'=\{B'(\delta):\delta\in\mc B\}$ is $\sigma$-invariant
if and only if its
transition matrix satisfies
$$ T_{\sigma(\delta),\sigma(\eta)}=T_{\delta,\eta}. $$
Equivalently, its linear change of basis commutes with $\sigma$.
\end{corollary}
\begin{proof}
Expand $B'(\sigma\delta)=\sigma(B'(\delta))$ in the
$\sigma$-invariant basis $B$ and compare coefficients.
\end{proof}

\subsection{String Profiles}

For a finite word $\mathbf a=(a_1,\ldots,a_m)$ in $\mathcal I$, set
$$ \delta_0=\delta,\qquad n_k=\rho_{a_k}(\delta_{k-1}),\qquad \delta_k=r_{a_k}^{\,n_k}(\delta_{k-1}), $$
and define
$$ \mathbf s_{\mathbf a}(\delta)=(n_1,\ldots,n_m), \qquad r_{\mathbf a}^{\max}(\delta)=\delta_m. $$
Upper seminormality ensures that each $\delta_k$ belongs to $\mc B$.
For the empty word, the string vector is empty and
$r_{\varnothing}^{\max}$ is the identity.
We call the two weights of $\delta$, together with its string
vectors $\mathbf s_{\mathbf a}(\delta)$ for all finite words
$\mathbf a$ in $\mathcal I$, the \emph{string profile} of $\delta$.

\begin{definition}\label{D:strorder}
\label{D:string-profiles}
Write $\eta\preceq_{\mathrm{prof}}\delta$ if their two weights
agree and
$$
\mathbf s_{\mathbf a}(\eta)\le_{\mathrm{lex}}
\mathbf s_{\mathbf a}(\delta)
\quad\text{for every finite word }\mathbf a\text{ in }\mathcal I.
$$
Write $\eta\sim_{\mathrm{prof}}\delta$ if both comparisons hold.
This defines a preorder on $\mc B$, whose strict part is denoted
by $\prec_{\mathrm{prof}}$, and an order on
$\mc B/\!\sim_{\mathrm{prof}}$.
\end{definition}
The construction uses maximal string data, as in
\cite[Sections~2.1 and~2.5]{B}. It differs from Baumann's string
order, which is defined using simultaneous crystal moves.

Taking words of length one shows that
$\eta\preceq_{\mathrm{prof}}\delta$ implies
$\eta\preceq_\rho\delta$. If the two weights of $\delta$ and
$\eta$ agree and
$\rho_a(\eta)<\rho_a(\delta)$ for every $a\in\mathcal I$,
then $\eta\prec_{\mathrm{prof}}\delta$ provided
$\mathcal I\ne\varnothing$.

\begin{lemma}\label{L:crtri}
Let $B'$ be another BK-biperfect basis indexed by $\mc B$,
for the same crystal structures and weight maps, and write
$B'(\delta)=\sum_{\eta\in\mc B}T_{\delta,\eta}B(\eta)$.
Then $T_{\delta,\eta}\ne0$ implies
$\eta\preceq_{\mathrm{prof}}\delta$.
If $\mathbf s_{\mathbf a}(\delta)=\mathbf s_{\mathbf a}(\eta)$, then
\begin{equation}\label{eq:coefficient-propagation}
T_{\delta,\eta}
=
T_{r_{\mathbf a}^{\max}(\delta),\,r_{\mathbf a}^{\max}(\eta)}.
\end{equation}
These assertions also hold for the inverse transition.
\end{lemma}
\begin{proof}
Fix $a\in\mathcal I$. Both bases are adapted to the intrinsic
filtration $K_{a,n}=\ker D_a^{n+1}$ by Corollary~\ref{C:string-kernels-hw}.
Thus a nonzero $T_{\delta,\eta}$ requires
$\rho_a(\eta)\le\rho_a(\delta)$.
If both lengths equal $n$, apply $D_a^{(n)}$ to the expansion.
All terms of smaller length vanish, and
$$ D_a^{(n)}B'(\delta)=B'(r_a^n\delta),\qquad D_a^{(n)}B(\eta)=B(r_a^n\eta). $$
The map $r_a^n$ is injective on
$\{\delta\in\mc B:\rho_a(\delta)=n\}$, with inverse
$l_a^n$. Comparing coefficients gives
$T_{\delta,\eta}=T_{r_a^n\delta,r_a^n\eta}$.
This also covers $n=0$. Fix a word $\mathbf a$. As long as
the successive string lengths of $\delta$ and $\eta$ agree,
the same calculation preserves $T_{\delta,\eta}$ under
simultaneous maximal raising. At the first unequal length,
preservation of the string filtration makes the length for $\eta$
smaller. This gives the lexicographic comparison. If the lengths
agree along the whole word, the coefficient is unchanged at the
terminal pair, proving \eqref{eq:coefficient-propagation}.
Bihomogeneity gives equality of the two weights. The same argument
applies to the inverse change of basis.
\end{proof}

\begin{corollary}[A sufficient condition for unitriangularity]
\label{C:terminal-unitriangular}
Suppose that every $\delta\in\mc B$ can be taken to $0$ by a
finite sequence of maximal raising operations from $\mathcal I$.
Then $\preceq_{\mathrm{prof}}$ is an order on $\mc B$, and the
transition between any two BK-biperfect bases indexed by $\mc B$
is lower unitriangular for this order.
\end{corollary}
\begin{proof}
Suppose $\delta\sim_{\mathrm{prof}}\eta$, and choose a word whose
maximal raising operators take $\delta$ to $0$. The string
data agree for this word and every continuation, so applying the
same word to $\eta$ gives a point at which every string length
is zero. This point is fixed by all maximal raising operators.
By hypothesis, the only such point is $0$. Reversing the partial
bijections therefore gives $\delta=\eta$.
Finally, \eqref{eq:coefficient-propagation} gives
$T_{\delta,\delta}=T_{0,0}=1$, and Lemma~\ref{L:crtri}
gives triangularity.
\end{proof}

Remark~\ref{r:multiplicity-normalization} exhibits distinct tropical
points with the same string profile. For compatibly pointed bases,
unitriangularity follows from Proposition~\ref{T:relative-good}.

\begin{question}
For $\delta,\eta\in\mc B$ of equal weights under both maps,
does $\eta\prec_{\mf T}\delta$ imply
$\eta\preceq_{\mathrm{prof}}\delta$?
\end{question}

\section*{Acknowledgement}
The author would like to thank Pierre Baumann and John Stembridge for some discussions.
He would also like to thank Xiaoyue Lin and Yueyang Yan for proof-reading the manuscript.

\bibliographystyle{amsplain}

\end{document}